\documentclass{amsart}
\usepackage{times}
\usepackage{amsfonts}
\usepackage{amsmath}
\usepackage{amscd}
\usepackage{amssymb}
\usepackage{latexsym}
\usepackage{amsthm}

\usepackage[bookmarks,colorlinks,pagebackref]{hyperref} 

\theoremstyle{plain}
\newtheorem{stelling}{Theorem}

\renewcommand\thestelling{\textnormal{\Alph{stelling}}}

\swapnumbers
\newtheorem{theorem}[subsection]{Theorem}
\newtheorem{corollary}[subsection]{Corollary}
\newtheorem{lemma}[subsection]{Lemma}
\newtheorem{proposition}[subsection]{Proposition}

\newtheorem{conjecture}[subsection]{Conjecture}

\theoremstyle{definition}
\newtheorem{definition}[subsection]{Definition}

\newtheorem{example}[subsection]{Example}

\theoremstyle{remark}

\newtheorem{remark}[subsection]{Remark}

\swapnumbers

\newcommand{\emptyprop}{q}

\newcommand \acf{algebraically closed field}
\newcommand \after{\circ}
\newcommand \ann[2]{\operatorname{Ann}_{#1}(#2)}
\newcommand \binomial[2]{{\bigl( \begin{matrix} #1\cr#2\cr\end{matrix} \bigr)}}
\newcommand \ch{characteristic}
\newcommand \CM{Coh\-en-Mac\-au\-lay}
\newcommand \complet[1]{\widehat {#1}} 
\newcommand \DVR{discrete valuation ring}

\newcommand \ext[4]{\operatorname{Ext}_{#1}^{#2}(#3,#4)}
\renewcommand \hom [3]{\operatorname{Hom}_{#1}(#2,#3)} 
\newcommand \homo{homomorphism}
\newcommand \id{\mathfrak a}
\renewcommand\iff{if and only if}
\newcommand\implication[2]{\eqref{#1}~$\Rightarrow$~\eqref{#2}}
\newcommand\into{\hookrightarrow}

\newcommand \inverse[2]{{#1^{-1}(#2)}}
\newcommand \iso{\cong}
\newcommand \loc{{\mathcal {O}}}
\newcommand \los{\L os' Theorem}
\newcommand \map[1]{{\newcommand{\tmpprop}{#1q}  \if\tmpprop\emptyprop \to\else \xrightarrow{{\phantom{i}{#1}\phantom{i}}}\fi}} 
\newcommand \maxim{\mathfrak m}
\newcommand \nat{\mathbb N}
\newcommand \norm[1]{\left|#1\right|}

\newcommand \op\operatorname
\newcommand \pol[2]{#1[#2]}
\newcommand \pow[2]{#1[[#2]]}

\newcommand \pr{\mathfrak p}
\newcommand \primary{\mathfrak g}
\newcommand \range [2]{#1,\dots,#2}

\newcommand \rij[2]{(#1_1,\dots,#1_{#2})}
\let\sub\subseteq

\newcommand \tensor{\otimes}
\newcommand \tor[4]{\operatorname{Tor}^{#1}_{#2}(#3,#4)}

\newcommand \zet{\mathbb Z}

\newcommand \exactseq [5]{0\to{#1}\:\map{#2}\:{#3}\:\map{#4}\:{#5}\to0}
\newcommand \Exactseq [3]{0\to {#1}\to {#2}\to {#3}\to 0}

\newcommand{\commdiagram}[9][]{%
\begin{equation}
{\newcommand{\tmpprop}{#1q} 
\if\tmpprop\emptyprop \relax\else \label{#1}\fi}
\begin{aligned}%
\mbox{
\begin{picture}(130,90)%
\put(120,70){\vector( 0,-1){50}}%
\put(10,80){\vector( 1, 0){100}}%
\put(0,70){\vector( 0,-1){50}}%
\put(10,10){\vector( 1, 0){100}}%
\put(115,80){\makebox(0,0)[l]{$#4$}}%
\put(5,80){\makebox(0,0)[r]{$#2$}}%
\put(115,10){\makebox(0,0)[l]{$#9$}}%
\put(5,10){\makebox(0,0)[r]{$#7$}}%
\put(-3,50){\makebox(0,0)[r]{$#5$}}
\put(123,50){\makebox(0,0)[l]{$#6$}}
\put(60,3){\makebox(0,0)[c]{$#8$}}
\put(60,88){\makebox(0,0)[c]{$#3$}}
\end{picture}}
\end{aligned}
\end{equation}}

\newcommand\commtrianglefront[7][]{%
\begin{equation}
{\newcommand{\tmpprop}{#1q} 
\if\tmpprop\emptyprop \relax\else \label{#1}\fi}
\begin{aligned}%
\mbox{
\begin{picture}(120,80)%
\put(55,68){\vector(-1,-2){30}}
\put(65,68){\vector(1,-2){30}}
\put(30,5){\vector(1,0){60}}
\put(60,75){\makebox(0,0)[c]{$#2$}}
\put(25,5){\makebox(0,0)[r]{$#4$}}
\put(95,5){\makebox(0,0)[l]{$#6$}}
\put(60,0){\makebox(0,0)[c]{$#5$}}
\put(37,43){\makebox(0,0)[r]{$#3$}}
\put(83,43){\makebox(0,0)[l]{$#7$}}
\end{picture}}
\end{aligned}
\end{equation}}

\newcommand\commtriangleback[7][]{%
\begin{equation}
{\newcommand{\tmpprop}{#1q}
\if\tmpprop\emptyprop \relax\else \label{#1}\fi}
\begin{aligned}%
\mbox{
\begin{picture}(120,80)%
\put(55,70){\vector(-1,-2){30}}
\put(65,70){\vector(1,-2){30}}
\put(30,5){\vector(1,0){60}}
\put(60,75){\makebox(0,0)[c]{$#2$}}
\put(25,5){\makebox(0,0)[r]{$#6$}}
\put(95,5){\makebox(0,0)[l]{$#4$}}
\put(60,0){\makebox(0,0)[c]{$#5$}}
\put(37,43){\makebox(0,0)[r]{$#7$}}
\put(83,43){\makebox(0,0)[l]{$#3$}}
\end{picture}}
\end{aligned}
\end{equation}}

\newcommand\NYCCT{\address{Department of Mathematics\\
City University of New York\\
365 Fifth Avenue\\ 
New York, NY 10016 (USA)}
\email{hschoutens@citytech.cuny.edu}}

\newcommand \ul[1]{#1\mathstrut_\natural}
\newcommand\ulmax{\ul\maxim}
\newcommand \sep[1]{{#1}_{\text{sep}}}
\newcommand \ulsep[1]{#1\mathstrut_\sharp}
\newcommand \seq[2]{#1\mathstrut_{#2}}
\newcommand \ord[2]{\op{ord}_{#1}(#2)}
\newcommand  \hull[1]{\mathfrak D(#1)}
\newcommand  \mult[1]{\op{mult}(#1)}
\newcommand  \pardeg[1]{\op{pardeg}(#1)}

\newcommand\INIT{INIT-degree}
\renewcommand  \c{ca\-ta}
\newcommand \tuple[1]{\mathbf{#1}} 
\newcommand  \C{Cata}
\newcommand  \pdim{geometric dimension}
\newcommand  \cheight{geometric codimension}
\newcommand  \Cheight{Geometric codimension}
\newcommand  \Pdim{Geometric dimension}
\newcommand  \cCM{\c-\CM}
\newcommand  \pCM{pseudo-\CM}
\newcommand  \creg{\c-regular}

\newcommand  \preg{pseudo-regular}
\newcommand  \cgor{\c-Goren\-stein}
\newcommand  \pgor{pseudo-Goren\-stein}
\newcommand  \cci{\c-`complete intersection'}

\newcommand \pd[1]{\op{gdim}(#1)}
\newcommand \cht[1]{\op{gcodim}(#1)}
\newcommand \kd[1]{\op{dim}(#1)}
\newcommand \cld[1]{\op{cldim}(#1)}
\newcommand \ed[1]{\op{embdim}(#1)}
\newcommand \ud[1]{\op{udim}(#1)}
\newcommand \cad[1]{\op{cdim}(#1)}
\newcommand \pid[1]{\op{pidim}(#1)}
\newcommand \frd[1]{\op{frdim}(#1)}
\newcommand  \Cech{\u{C}ech}
\newcommand \spec[1] {\op{Spec}(#1)}
\newcommand \clspec[1] {\op{CL-Spec}(#1)}
\newcommand \conspec[1]{\op{FR-Spec}(#1)}

\renewcommand \inf[1]{\mathfrak I_{#1}}

\newcommand \red[1]{#1_{\text{red}}}
\newcommand \gr[2]{\op{gr}_{#1}(#2)}
\newcommand  \fr{finitely related}
\newcommand  \con[1]{$#1$-related}

\renewcommand\nat{\mathbb N}
\renewcommand\zet{\mathbb Z}

\newcommand \fp[1]{finitely $#1$-presented}
\newcommand\bertin{Bertin-Serre regular}
\newcommand \fpt{finite presentation type}

\title {Dimension and singularity theory for local rings of finite embedding
dimension}
\author{Hans Schoutens}
\date{\today}
\NYCCT
\thanks{Partially supported by the National Science Foundation and a PSC-CUNY
grant}

\begin{document}

\begin{abstract}
In this paper, an algebraic theory  for local rings of finite
embedding dimension is developed. Several extensions of (Krull) dimension are
proposed, which are then used to generalize singularity notions from commutative
algebra. Finally, variants of the homological theorems are shown to hold  in equal
\ch. 

This theory is then applied to Noetherian local rings in order to get: (i) over a \CM\ local ring, uniform bounds on the Betti numbers of a
\CM\ module in terms of  dimension and multiplicity, and  similar bounds for the Bass numbers of a finitely generated module;
(ii)   a characterization for being respectively analytically
unramified, analytically irreducible, unmixed, quasi-unmixed, normal, \CM,  pseudo-rational, or weakly F-regular in terms of certain
uniform   arithmetic behavior;  
(iii)  in mixed \ch,  the Improved New Intersection Theorem
when the residual \ch\ or
ramification index is large with respect to dimension (and some other
numerical invariants).
\end{abstract}

\maketitle

\section{Introduction}

This paper is devoted to the study of local
rings of finite embedding dimension, where by a \emph{local ring}, we
mean a not necessarily Noetherian,
commutative ring $R$ with a   unique maximal ideal $\maxim$,\footnote{In the literature, such   rings are
also called \emph{quasi-local}.} and where the 
\emph{embedding
dimension} of $R$, denoted $\ed R$, is the minimal number of elements 
generating $\maxim$. 
We will see that there are various ways
of extending the dimension and singularity theory of Noetherian local rings to
this larger class. The motivation for this study comes from the subclass of
\emph{ultra-Noetherian} local rings: these are the ultraproducts of Noetherian
local rings of fixed embedding dimension. I had used these ultra-Noetherian rings
in my previous work on transfer
from positive to zero \ch\ (\cite{SchSymPow,SchBS}) and on
non-standard tight closure (\cite{SchAsc,SchLogTerm,SchNSTC,SchBCM,SchRatSing}),
but the actual study of their
properties was only prompted by the papers \cite{SchMixBCMCR,SchMixBCM}, where it
was essential to have a generalized dimension and  singularity theory 
to get
asymptotic versions of various homological theorems in mixed \ch. It was
this realization that led me to develop a systematic `local algebra' for these
rings. Consequently, we will be able to  derive from this study some improved
asymptotic versions   in the
final section. For some other recent papers studying ultraproducts of Noetherian
rings, see \cite{OlbSay,OlbSaySha,OlbSha}.

Closely related to a local ring of finite embedding dimension are two
local rings which are always Noetherian: its graded ring and its completion.
Especially through the latter the study of local rings of finite embedding 
dimension is greatly facilitated. Accordingly, I will use the modifier
\emph{\c-} to indicate that a property is inherited by completion. In contrast,
for ultra-Noetherian local rings, the prefix
\emph{ultra-} is used to refer to properties that are inherited by the
ultraproduct. The main goal is
now to find conditions under which both   versions agree, which often requires
the introduction of a third, intrinsic (\emph{pseudo-}) variant.
To study these variants, we introduce the
notion of a
\emph{cataproduct}, defined
as the completion of the ultraproduct. In fact, the cataproduct is
obtained from the ultraproduct by factoring out the \emph{ideal of
infinitesimals}, that is to say, the ideal of elements lying in each power of
the maximal ideal. In \cite{SchUlBook}, both the ultraproduct and the cataproduct
are called \emph{chromatic products}, inspired by our musical notation $\ul R$ and $\ulsep R$ respectively (a third chromatic product, not discussed in this paper, is called the \emph{proto-product} and
denoted $R_\flat$).

What follows is a brief outline of the present paper. To illustrate the
methods and concepts, I will here only treat the
special case that $(\ul R,\ulmax)$ is an ultra-Noetherian local ring, realized
as the ultraproduct of Noetherian local rings $(\seq Rn,\seq\maxim n)$ of the same embedding dimension.
Section~\ref{s:finemb} contains general facts of local rings of finite
embedding dimension, by far the most important of which is the already
mentioned result that its completion is Noetherian
(Theorem~\ref{T:finembcomp}). In particular,   the cataproduct $\ulsep
R$ is Noetherian.\footnote{Special cases of this result were already observed and
used by various authors \cite{SchAsc,BDDL,OlbSaySha}.}
Our first task is now to develop  a good dimension theory, which is done  in 
\S\S\ref{s:pdim}--\ref{s:ul}. Krull dimension in this context is of minor use,
as it is always infinite  for example for ultra-Noetherian local rings, except
when
almost all $\seq Rn$ are Artinian of a fixed length $l$, in which case $\ul R$
is also zero-dimensional and has length $l$. A
first variant, called
\emph\pdim,
is inspired by the geometric intuition that dimension is the least number of
hypersurfaces cutting out a finitely supported subscheme. Specifically, the
\emph\pdim, $\pd {\ul R}$,
of
$\ul R$ is the least number $d$ of non-units $x_1,\dots,x_d$ such that $\ul R/\rij
xd\ul R$ is
Artinian. Other variants are obtained by the general principle discussed above:
the \emph{ultra-dimension}, $\ud{\ul R}$,
of $\ul R$ is the common dimension of almost all $\seq Rn$; and  its
\emph{\c-dimension}   is the dimension of its completion, that is to
say, 
of   
$\ulsep R$. It turns out that the
\c-dimension is equal to the \pdim\ (Theorem~\ref{T:pdim}). These dimensions also
have a combinatorial nature: whereas Krull dimension
is the combinatorial dimension of the full spectrum $\spec{\ul R}$,  the ultra-dimension of $\ul R$ is equal to the combinatorial dimension of the
subset of all associated prime ideals of finitely generated ideals; the \c-dimension
 is equal to the  combinatorial dimension of the
subset of all $\ulmax$-adically closed prime ideals
(Theorem~\ref{T:uldim}; see also \cite{OlbSha} for some related results).
The ultra-dimension
of $\ul R$ is at most its \c-dimension, with equality precisely when almost all $\seq
Rn$ have the same
parameter degree (Theorem~\ref{T:isodim}).

Our next step  is to develop a singularity theory for local rings of finite
embedding dimension. Three options present themselves to us:
\c-singularities via completions (\S\ref{s:cata}); ultra-singularities via ultraproducts
(\S\ref{s:ulsing}); and pseudo-singularities via homological
algebra (\S\ref{s:sCM}). For instance, $\ul R$ is called \emph{cata-regular} if 
$\ulsep R$ is regular; \emph{ultra-regular} if almost all $\seq Rn$ are
regular; and \emph{pseudo-regular} if its depth equals its embedding
dimension.    Requiring each of the quantities
\begin{equation}\label{eq:uldimineqintro}
\op{depth}(R) \leq \ud{\ul R}\leq \pd{\ul R}\leq \ed{\ul R}
\end{equation}
in this chain of inequalities  to be equal to the
last   turns out to determine these regularity conditions, in decreasing order
of strength: pseudo-regularity, ultra-regularity, and cata-regularity
respectively
(note that we do not observe such a distinction in the Noetherian case). In
fact,
the two first conditions are equal  (Theorem~\ref{T:psing}). Moreover,
Serre's criterion for regularity
extends to this larger class (Theorem~\ref{T:sreg}). In particular, for
coherent local rings of finite embedding dimension,  regular in
the sense of Bertin  (\cite{Bert,Glaz})  implies \preg, and the converse holds
for uniformly
coherent local rings of finite embedding dimension (Theorem~\ref{T:unifcoh}).
Next, variants of the \CM\ property are analyzed---for instance, by equating
the first quantity
in \eqref{eq:uldimineqintro} with respectively the second and third, we get
the notions of  ultra-\CM\ and pseudo-\CM\ local rings. Unfortunately, 
these variants behave less well. For instance, although the class of local
\CM\   rings of  fixed dimension and multiplicity is
closed under cataproducts (Corollary~\ref{C:CMulsep}), the converse
need not be true, that is to say, $\ulsep R$ can be \CM\ without the $\seq Rn$
being \CM. At the source of these discrepancies lies the fact that a sequence can
be quasi-regular without being regular in non-Noetherian rings.  In
\ref{E:uldimineq}, we present 
an example showing that all   of the four quantities in
\eqref{eq:uldimineqintro} can be different. 
 Although $\ul R$ is rarely coherent,   under an additional 
\pCM\ assumption, it behaves much like one: any $\ulmax$-primary
ideal, and more generally, any
finitely generated ultra-\CM\ module  is finitely presented. 
Another
generalization of the \CM\ condition for local rings of finite embedding
dimension, motivated by model-theoretic considerations, was introduced in
\cite{SchAA};
we show that up to a Nagata extension of the ring (which can be taken to be
trivial in the ultra-Noetherian case),
this condition is equivalent with being \pCM\ (Theorem~\ref{T:CMde}).
Some further
characterizations of the various types of \CM\ singularities are given in 
\S\ref{s:catanorm} by means of an analogue of  Noether Normalization
for the class of local rings of finite embedding dimension.

Once  we have developed a sufficiently well-behaved singularity theory, we
analyze
the homological theory of the class of local rings of finite embedding dimension;
this is the contents of \S\ref{s:INIT}. We show that most
homological theorems, properly restated, hold in an arbitrary equi\ch\ local
ring of finite embedding dimension. The main tool is the existence of  an analogue of 
big \CM\ algebras for this class of rings. In fact, it suffices to assume
that only the completion is equi\ch, which is a strictly weaker condition, as
I will explain below. As an application, we provide the following
partial answer to a question raised by Glaz (\cite{GlazCohReg}) about the
extent to which   split  subrings of
coherent regular local rings   are \CM\ (note that in the Noetherian case, either condition would imply that $R$ is \CM; for a different answer, see  \cite[Corollary 4.5]{HamMar}). 

\begin{corollary}\label{C:GlazConj}
If $(R,\maxim)$ is a local ring of finite embedding dimension containing a
field, and if $S$
a coherent regular local ring locally of finite type over $R$, such that $R\to
S$ is cyclically pure (e.g., split), then there exists a (Noetherian) regular local subring $(A,\pr)$ of $R$ such that each maximal $A$-regular sequence is a maximal quasi-regular sequence in $R$, and each  $R/\pr^nR$ is a finite,  free $A/\pr^n$-module.
\end{corollary}

In the final three sections, we apply the theory to ultra-Noetherian rings to obtain new results about
Noetherian local rings. In \S\ref{s:ub}, we derive     uniform bounds on
 Betti and Bass
numbers. In the literature, one often studies the asymptotic growth of the
\emph{Betti numbers} $\beta_n(M)=\op{dim}_k(\tor RnMk)$,
as $n$ goes
to infinity, for $M$  a finitely generated module    over a
Noetherian local   ring $R$ with residue field $k$. In contrast,  varying the
module and fixing $n$, we show in
Theorem~\ref{T:unifbetti}  that over a local \CM\ ring $R$, the
$n$-th Betti number  of a module $M$ of finite length  is bounded by a
function which only depends on the dimension and multiplicity of $R$ and the
length of $M$. In particular, if $P_R(t):=\sum_n\beta_n(k)t^n$ denotes the \emph{Poincare
series} of   $R$, then we show:

\begin{corollary}\label{C:Poincare}
For each $d,e\geq 0$, there exists a power series $P_{d,e}(t)\in\pow\zet t$
such that the Poincare series $P_R(t)$ of any $d$-dimensional local \CM\ ring
$R$ of multiplicity $e$ is   dominated by $P_{d,e}(t)$, meaning that $P_{d,e}(t)-P_R(t)$ has non-negative coefficients.
\end{corollary}

Recall that a \CM\ local ring $R$ is called of \emph{bounded multiplicity type} if there is a bound on the multiplicity of all of its indecomposable maximal \CM\ modules. According to the Brauer-Thrall conjectures   such a ring is   expected to be of \emph{finite representation type}, meaning that there exist only finitely many indecomposable maximal \CM\ modules (see, for instance, \cite{DieRed,PopMCM,YoBT} for some known cases of this conjecture). In support of this, we prove the following `universal resolution' for maximal \CM\ modules:

\begin{corollary}\label{C:BT}
Suppose $R$ is a local \CM\ ring of bounded multiplicity type. There exists an $R$-algebra $Z$, and a complex of finite free $Z$-modules $\mathcal F_\bullet$, such that for every   indecomposable maximal \CM\ module $M$,   there exists a section $Z\to R$, such that $\mathcal F_\bullet\tensor_ZR$ is a free resolution of $M$.\end{corollary}

The theory also gives applications to preservation of properties under
infinitesimal deformations,  of which the next result is but an example
(recall that an \emph{invertible ideal} is a principal
ideal generated by a non zero-divisor):

\begin{corollary}\label{C:invertible}
Let $R$ be a local \CM\ ring and let $I\sub R$ be an invertible ideal. There exists a positive integer
$a:=a(I)$ with the property that if $J\sub R$ such that $R/J$ is \CM\ of
multiplicity at most the multiplicity of $R/I$, and such that
$I+\maxim^a=J+\maxim^a$, then $J$ is invertible.
\end{corollary}

It is not clear yet whether similar bounds exist if we drop the \CM\ assumption in these results. In \S\ref{s:ua}, we characterize ring-theoretic properties in
terms of uniform arithmetic in the ring. For instance,
in Theorem~\ref{T:domCh}, we reprove, as an
illustration of our methods, that multiplication is bounded in $R$ \iff\ $R$ is
analytically irreducible.  Whereas the
ultraproduct method   only gives the existence of a
uniform bound, we know in this particular case, by the work of H\"ubl-Swanson
\cite{HueSwa,Swa},
that  these bounds  can be taken to be linear. Nonetheless, our method is far
more versatile, 
allowing us to derive in \S\ref{s:charub} many more  characterizations of ring-theoretic properties
  in terms of  certain  uniform asymptotic
behavior of ($\maxim$-adic) \emph{order} and (parameter) \emph{degree}. For
instance, one can characterize the
\CM\
property as follows:

\renewcommand\thestelling{\textbf{\ref{T:ubCM}}}
\begin{stelling}
 For each quadruple $(d,e,a,b)$ there exists a bound $\delta(d,e,a,b)$ with the
following property. A $d$-dimensional Noetherian local ring $(R,\maxim)$ of
multiplicity $e$ is \CM\ \iff\ for each ideal $I$ generated by $d-1$ elements, and for
 any two elements $x,y\in R$, if $R/(I+xR)$ has
length at most $a$ and $y$ does not belong to $I+\maxim^b$, then $xy$ does not
belong to $I+\maxim^{\delta(d,e,a,b)}$.
\end{stelling}

 As already
mentioned, our methods only prove the existence of uniform bounds (and
possibly   their dependence on other invariants), but say nothing about the
nature of these bounds. It would
be interesting to see whether for instance these new bounds also have a   linear
character.

However, the main application of this paper is discussed in the final section.
Here we derive some asymptotic versions of  the homological theorems in
mixed \ch.
 Whereas the papers  \cite{SchMixBCMCR,SchMixBCM} relied on a deep
result from model theory, the so-called Ax-Kochen-Ershov theorem, to carry out
transfer from mixed to equal \ch,\footnote{In fact, although not mentioned
explicitly in these papers (but see \cite[\S14]{SchUlBook}), these methods make
heavily use of proto-products,
one of the chromatic products not studied in this paper.} 
the
present paper departs from the following simple observation: if the $(\seq Rn,\seq\maxim n)$ have mixed
\ch\ $\seq pn$, then their cataproduct $\ulsep R$ is equi\ch\  in
the following two cases: (i)
the  $\seq pn$ grow unboundedly (in which case the  ultraproduct $\ul R$ is already equi\ch), or
(ii),
 almost all $\seq
pn$ are equal to a fixed prime number $p$, but the ramification index, that is to say, the
  $\seq\maxim n$-adic order of $p$, grows
unboundedly (in which case $\ul R$ still has mixed \ch\ $p$).  Thus we prove:

\renewcommand\thestelling{\textbf{\ref{T:asymhc}}}
\begin{stelling}[Asymptotic Improved New Intersection Theorem]
 For each triple of positive integers $(m,r,l)$ there exists a bound
$\kappa(m,r,l)$
with 
the following property.  Let $(R,\maxim)$ be  a mixed \ch\ Noetherian local ring
 of embedding dimension $m$ and let $F_\bullet$ be a
finite complex of finitely generated free  $R$-modules of rank at most $r$.  If
each $H_i(F_\bullet)$, for $i>0$, has length at most $l$ and 
$H_0(F_\bullet)$ has a non-zero minimal generator generating a
submodule of length at most $l$, then the length of $F_\bullet$ is at 
least the dimension of $R$, provided either  the residual \ch\   or the
ramification index  of $R$ is  at least $\kappa(m,r,l)$.
\end{stelling}

It should be noted that some Homological Conjectures, such as the Direct Summand
Conjecture and the Hochster-Roberts theorem on the \CM{}ness of pure subrings of regular local
rings, at present elude our methods, and so no asymptotic versions in the style of this
paper are known (but see \cite[\S9 and \S10]{SchMixBCM} for   different
asymptotic versions).
 
I conclude the paper with a sketch of an argument
that  
derives the full version from its asymptotic counterpart, provided   the 
bounding function does not grow too fast. For example, if for some prime  $p$,
 the bound $\kappa(m,r,l)$ on the ramification in the above theorem can be taken to be
of the form $c(m,r)l^{\alpha(m,r)}$, for some real valued functions $c(m,r)$ and
$\alpha(m,r)$ with $\alpha(m,r)<1$, for all $m$ and $r$, then the Improved New Intersection Theorem
holds in mixed \ch\ $p$.

\section{Finite embedding dimension}\label{s:finemb}

Although we will mainly be interested in the maximal adic topology of a local
ring, we start our exposition in a more general setup.

\subsection{Filtrations}\label{s:adic}
Recall that a \emph{filtration} $\mathfrak I=(I_n)_n$ on a ring $A$ is a
descending
chain of ideals $A=I_0\supseteq  I_1 \supseteq\dots \supseteq I_n \supseteq
\dots$. An important instance of a filtration is obtained
by taking the  powers of a fixed ideal $I\sub A$, that is to say, $I_n:=I^n$; we
call this the \emph{$I$-adic filtration} on $A$. A
 filtration $\mathfrak I$ defines a   topology on $A$, called the \emph{$\mathfrak
I$-adic topology} of $A$, by taking for basic open
subsets all   cosets of all  $I_n$. If $B$ is an $A$-algebra, then $\mathfrak IB$ is the \emph{extended
filtration} on $B$ given by the ideals $I_nB$, and hence the natural \homo\
$A\to B$ is continuous with respect to the respective  adic topologies.
The intersection of all $I_n$ will be denoted by $\mathfrak
I_\infty$. Hence the $\mathfrak I$-adic topology is Haussdorf (separated) \iff\
$\mathfrak I_\infty=(0)$. Accordingly,  the quotient
$A/ {\mathfrak I}_\infty$ is called  the \emph{${\mathfrak I}$-adic separated quotient} of $A$. 
The \emph{${\mathfrak I}$-adic completion}  of $A$ is defined
as the inverse limit of the $A/I_n$ and is denoted $\complet A_{\mathfrak I}$. There is a natural map $A\to \complet
A_{\mathfrak I}$ whose kernel is equal to $ {\mathfrak I}_\infty$. In fact, $A$ and its ${\mathfrak I}$-adic 
separated quotient have the same ${\mathfrak I}$-adic completion. In 
general, $\complet A_{\mathfrak I}$, although complete in the
inverse limit topology, need not be complete in the ${\mathfrak I}\complet 
A_{\mathfrak I}$-adic topology.

Given a filtration $\mathfrak I=(I_n)_n$ we define its
\emph{associated graded module}, where we view $A$ with its trivial grading, as
the direct sum
\begin{equation*}
\gr {\mathfrak I}A:= \bigoplus_{n=0}^\infty I_n/I_{n+1}.
\end{equation*} 
The \emph{initial form} $\op{in}_{\mathfrak I}(a)\in\gr {\mathfrak I}A$ and the
\emph{$\mathfrak I$-adic order} $\ord{\mathfrak I}{a}$ of
an  element $a\in A$ are defined as follows. If $a\in
I_n\setminus I_{n+1}$ for some $n$, then we set $\ord{\mathfrak I}a:=n$ 
and we let $\op{in}_{\mathfrak I}(a)$ be the 
image of $a$ in $I_n/I_{n+1}$; otherwise $a\in {\mathfrak I}_\infty$, in which case we set
$\ord{\mathfrak I}a:=\infty$ and $\op{in}_{\mathfrak I}(a):=0$.  For $J$ an
ideal in $A$, we let $\op{in}_{\mathfrak I}(J)$   be the ideal in
$\gr {\mathfrak I}A$  generated by all $\op{in}_{\mathfrak I}(a)$ with $a\in J$.
If $J=\rij
anA$, then $\op{in}_{\mathfrak I}(J)$ is in general larger than the ideal generated by 
the $\op{in}_{\mathfrak I}(a_i)$ (even if $A$ is Noetherian!). 

Alternatively, we
may
think of a filtration as given by a function $f\colon A\to \bar\nat:=\nat\cup\{\infty\}$
such that $f(a+b)$ and $f(ab)$ are greater than or equal to respectively the
minimum and the maximum of $f(a)$ and $f(b)$; we express this by calling   
$f$  \emph{filtering}. Given a filtering function $f$, 
the ideals $I_n$
of all elements $a\in A$ for which $f(a)\geq n$  form a filtration. Conversely,
given a filtration $\mathfrak I$, the function $\ord{\mathfrak
I}\cdot$ is filtering. Suppose $f$ is filtering. If $f(ab)\geq f(a)+f(b)$, then
we call $f$   \emph{multiplicative}  (this then corresponds to the
property that $I_nI_m\sub I_{n+m}$); and if $0$ is the only element of
infinite $f$-value (so that the corresponding filtration is separated) and
$f(ab)=f(a)+f(b)$, then $f$ is called a \emph{valuation}. If $\mathfrak I$ is
multiplicative, then $\gr{\mathfrak I}A$ admits the structure of a ring and 
as such is graded. This applies in particular to any ideal adic filtration.

We now specify these notions to the case of interest, where $\mathfrak  I$ is 
the $\maxim$-adic filtration of a local ring  $(R,\maxim)$.
The topology on $R$ is always assumed to be 
the $\maxim$-adic topology, so that when we say
that $R$ is  separated  or complete, we are always referring to this 
topology. With this in mind, the \emph{ideal of
infinitesimals} of $R$ is  the intersection of all $\maxim^n$, and will be 
denoted $\inf R$. The $\maxim$-adic order of an element $x\in R$
is denoted $\ord Rx$ or just $\ord{}x$. The ($\maxim$-adic) separated quotient
$R/\inf R$ is denoted $\sep R$;  the graded ring associated to 
$\maxim$ is denoted $\gr{}R$; and the completion of $R$
is denoted $\complet R$. By construction, $\complet R$ is a 
complete local ring whose maximal ideal is equal to
the inverse limit of the $\maxim/\maxim^n$. However, this maximal 
ideal may be strictly larger than $\maxim\complet R$,
so that  $\complet R$ need not be complete in the $\maxim\complet 
R$-adic topology.

Let $(S,\mathfrak n)$ be a second local ring and let $R\to S$ be a 
ring \homo. We call this \homo\ \emph{local}, or we
say that $S$ is a \emph{local} $R$-algebra, if 
$\maxim S\sub\mathfrak n$; if we have equality, then we call the \homo\ \emph{unramified}. A local  \homo\ induces  local \homo{s} $\sep
R\to \sep S$ and  $\complet R\to \complet S$. The  natural map $R\to 
\complet R$  is   local. It is flat if $R$ is
Noetherian,  but no so in general.

\subsection*{Finite embedding dimension}
Suppose from now on that $R$ has moreover finite embedding dimension, that is
to say, that $\maxim$ is finitely generated.
Since $\gr{}R$ is generated by $\maxim/\maxim^2$ as an algebra over the field 
$R/\maxim$, it is itself a Noetherian local ring. For each
$n$, let $\complet\maxim_n$ be the kernel of the natural map 
$\complet R\to R/\maxim^n$. It follows that
$\maxim^n/\maxim^{n+1}\iso \complet \maxim_n/\complet\maxim_{n+1}$, 
so that $\gr {}R$ is equal to the graded ring $\gr{\mathfrak M}{\complet R}$
associated to the filtration  $\mathfrak M:=(\complet\maxim_n)_n$ on $\complet 
R$. By \cite[Proposition 7.12]{Eis}, an ideal
$I\sub \complet R$ is generated by elements $a_1,\dots,a_n$ if its 
initial from $\op{in}_{\mathfrak M}(I)$ in $\gr{\mathfrak M}{\complet R}$ is generated by the
initial forms $\op{in}_{\mathfrak M}(a_1),\dots,\op{in}_{\mathfrak M}(a_n)$. Therefore, since
$\gr{\mathfrak M}{\complet R}\iso 
\gr{}R$ is Noetherian, so is  $\complet R$. Moreover,
since $\maxim^n\complet R$ has the same initial form as 
$\complet\maxim_n$, both ideals are equal. In
particular, for each $n$, we have an isomorphism $R/\maxim^n\iso 
\complet R/\maxim^n\complet R$. In conclusion, we
  have proven:

\begin{theorem}\label{T:finembcomp}
If $(R,\maxim)$ is a local ring of finite embedding dimension, then its
completion $\complet R$ is a complete Noetherian local ring with maximal ideal
$\maxim\complet R$. \qed
\end{theorem}

\begin{corollary}
If a local ring $(R,\maxim)$ has finite embedding dimension, then each
$\maxim$-primary ideal is finitely generated.
\end{corollary}
\begin{proof}
Immediate from the fact that $R/\maxim^n$
is Artinian and $\maxim^n$ is finitely generated, for every $n$.
\end{proof}

An ideal $I$ in a local ring $(R,\maxim)$ is called \emph{closed} if it is
closed in the $\maxim$-adic topology, that is to say, if $I$ is equal to the
intersection of all $I+\maxim^n$ with $n\in\nat$.

\begin{lemma}\label{L:quot} 
Let $(R,\maxim)$ be a local ring of 
finite embedding dimension and let $I$  be an arbitrary
ideal in $R$. The completion of $R/I$ is $\complet R/I\complet R$. In 
particular, $I\complet R\cap R=I$ \iff\ $I$ is
closed.
\end{lemma}
\begin{proof} Let $\bar R:=R/I$ and let $S:=\complet R/I\complet 
R=\complet R\tensor_R\bar R$. The isomorphism
$R/\maxim^n\iso \complet R/\maxim^n\complet R$ induces by base change 
an isomorphism $\bar R/\maxim^n\bar R\iso
S/\maxim^n  S$. Hence $\bar R$ and $S$ have the same completion. 
However, since $\complet R$ is   complete, so is $S$,
showing that it is the completion of $\bar R$.

Applied with $I$ an $\maxim$-primary ideal, we get an isomorphism 
$R/I\iso \complet R/I\complet R$ showing that
$I\complet R\cap R=I$, that is to say, that $I$ is \emph{contracted 
from $\complet R$}. Since this property is preserved
under arbitrary intersections,  every  closed ideal $I$ is 
contracted from $\complet R$, as it is   the
intersection of the $\maxim$-primary ideals  $I+\maxim^n$. 
Conversely, if $I\complet R\cap R=I$, then $R/I$ embeds in
$\complet R/I\complet R$, and by the first assertion, this is its 
completion. In particular, $R/I$ is separated, that is
to say, $I$ is closed.
\end{proof}

The above proof shows that the closure of an ideal $I$ is equal to $I\complet 
R\cap R$. In particular, any closed ideal is the closure of a finitely generated ideal, since $\complet R$ is Noetherian by Theorem~\ref{T:finembcomp}. Moreover, the ascending chain
condition holds for closed ideals  in $R$:  if $I_1\sub 
I_2\sub\dots$ is an increasing chain of closed ideals in
$R$, then, since $\complet R$ is Noetherian, their extension to 
$\complet R$ must become stationary, say $I_n\complet
R=I_{n+k}\complet R$ for all $k$, and hence contracting back to $R$ 
gives $I_n =I_{n+k} $ for all $k$. This immediately yields:

\begin{corollary}
A local ring is Noetherian \iff\ it has finite embedding dimension and every ideal is closed.\qed
\end{corollary}

\begin{corollary}\label{C:clmin} 
A  closed ideal in a local ring $R$ 
of finite embedding dimension has finitely many
minimal primes and each of them is  closed. 
\end{corollary}
\begin{proof} Let $I$ be a closed ideal and let $\mathfrak 
Q_1,\dots,\mathfrak Q_s$ be the minimal prime ideals of
$I\complet R$. Let $\mathfrak q_i:=\mathfrak Q_i\cap R$ and let $J$ 
be their product. Hence $J^n\sub  I\complet R$
for some $n$. By Lemma~\ref{L:quot}, we have $J^n\sub 
I\complet R\cap R=I$. Hence any prime ideal $\pr$ of $R$
containing $I$ contains   one of the $\mathfrak q_i$. This shows that 
all minimal prime ideals of $I$ must be among the
$\mathfrak q_i$.
\end{proof}

\begin{corollary}\label{C:spec} 
If $(R,\maxim)$ is a local ring  of 
finite embedding dimension, then the image of the
map $\spec{\complet R}\to \spec R$ consists precisely of the 
closed prime ideals of $R$.
\end{corollary}
\begin{proof} By Lemma~\ref{L:quot}, the image of the map consists of 
closed prime ideals. To prove the converse, let
$\pr$ be an arbitrary closed prime ideal of $R$. By 
Lemma~\ref{L:quot}, we have $\pr=\pr\complet R\cap R$. Let
$\mathfrak N$ be maximal in $\complet R$ with the property  that 
$\pr=\mathfrak N\cap R$. I claim that $\mathfrak N$ is
a prime ideal, showing   that $\pr$ lies in the image of 
$\spec{\complet R}\to \spec R$. To prove the claim,
suppose $fg\in\mathfrak N$, but $f,g\notin\mathfrak N$. By maximality, 
there exist $a,b\in R\setminus\pr$ such that
$a\in\mathfrak N+f\complet R$ and $b\in\mathfrak N+g\complet R$. 
Hence $ab\in \mathfrak N+fg\complet R=\mathfrak N$ and
since $ab\in R$, we get $ab\in\mathfrak N\cap R=\pr$, contradicting 
that $\pr$ is prime.
\end{proof}

\begin{lemma}\label{L:Art} 
If the completion of a local ring 
$(R,\maxim)$ of finite embedding dimension is Artinian,
then so is $R$.
\end{lemma}
\begin{proof} 
By assumption, $\maxim^n\complet R=0$, for some $n$. 
Since $R/\maxim^{n+1}\iso \complet
R/\maxim^{n+1}\complet R=\complet R$, we get 
$\maxim^n/\maxim^{n+1}=0$. Since $\maxim$ is finitely generated, we
may apply Nakayama's Lemma and conclude that $\maxim^n=0$, which 
implies that $R$ is Artinian.
\end{proof}

\subsection{Infinite ramification}\label{s:ram}
We conclude this section with a  note on ramification in mixed \ch, which we
will use occasionally. Let $(R,\maxim)$ be a local ring with residue field $k$.
We say that $R$ is
\emph{equi\ch} (or has \emph{equal \ch}) if $R$ and $k$ have the same \ch;
in the remaining case, that is to say, if $R$ has \ch\ $0$ and $k$ \ch\ $p$, we
say that $R$ has \emph{mixed \ch} $p$.
A local ring is equi\ch\ \iff\ it contains a field.

For the next definition, assume that the residue field of $R$ has 
  \ch\ $p$.
We call $\ord{}p$ the \emph{ramification index} of $R$. We say $R$ is
\emph{unramified} if
its ramification index   is one; and \emph{infinitely ramified}, if its
ramification index is infinite, that is to say, if $p\in\inf R$. If $R$ is
infinitely ramified and
Noetherian (or just separated),  then in fact it has equal \ch\ $p$ (in the literature
this  is also deemed as an instance of an `unramified' local ring, but for us,
it will be more useful to make the distinction). However, in
the general case, a local ring can have \ch\ zero and be infinitely ramified (see   Lemma~\ref{L:infram} below).
It follows that the separated quotient and the completion of an infinitely
ramified local ring are both equi\ch.

\section{\Pdim}\label{s:pdim}

The \emph{dimension} $\kd A$ of a ring $A$ will always mean its Krull 
dimension, that is to say, the maximal length  (possible
infinite) of a chain of prime ideals in $A$. The \emph{dimension} of an ideal
$I\sub A$ is the dimension of its residue ring $A/I$.
If $R$ is local and
 Noetherian, then its dimension is always finite, but
without the Noetherian assumption, it is generally infinite. In this section, we 
propose a first substitute for Krull dimension for
an arbitrary local ring $(R,\maxim)$; other alternatives will be discussed in   \S\ref{s:bluedim}.

\begin{definition} 
We define the \emph{\pdim}  of $R$ recursively as 
follows. We say that $R$ has \pdim\ zero, and we
write $\pd R=0$,  \iff\ $R$ is Artinian. For arbitrary $d$, we say that $\pd 
R\leq d$, if there exists $x\in \maxim$ such that $\pd
{R/xR}\leq d-1$. Finally, we say that $R$ has \pdim\ equal to $d$ if 
$\pd R\leq d$, but not $\pd R\leq d-1$, and we
simply  write $\pd R:=d$. If there is no $d$ such that $\pd R\leq d$, 
then we set $\pd R:=\infty$.
\end{definition}

It follows that $\pd R\leq \ed R$.  In fact, $R$ has finite 
\pdim\ \iff\ it has finite embedding dimension. If
$R$ has finite embedding dimension then $\pd R=0$ \iff\ $\maxim$ is 
nilpotent. The following fact is immediate from the
definition.

\begin{lemma}\label{L:dimdef} 
If $(R,\maxim)$ is a local ring and 
$a\in\maxim$, then
\begin{equation*}
\pd R-1\leq \pd{R/aR}\leq\pd R.
\end{equation*}\qed
\end{lemma}

The \pdim\ can be formulated, as in the Noetherian case, in terms of the 
minimal number of generators of an $\maxim$-primary
ideal (showing that \pdim\ and Krull dimension agree for Noetherian 
local rings):

\begin{lemma}\label{L:mingen} 
The \pdim\ of a local ring $(R,\maxim)$ 
of finite embedding dimension is the least
possible number of  elements generating  an $\maxim$-primary ideal.
\end{lemma}
\begin{proof} Let $d:=\pd R$.   By Lemma~\ref{L:dimdef}, there exists 
no sequence $\tuple y$ of length less than $d$
such that $R/\tuple yR$ has \pdim\ zero. It follows that any 
$\maxim$-primary ideal is generated by at least $d$
elements. So remains to show that there exists a tuple of length $d$ 
generating an $\maxim$-primary ideal. We induct on
  $d$, where the case $d=0$ is clear, since then  $(0)$ is 
$\maxim$-primary. By definition, we can choose $x_1\in\maxim$
such that $\pd{R/x_1R}=d-1$. By induction, there exist elements 
$x_2,\dots,x_d$ whose image in $R/x_1R$ generate an
$\maxim(R/x_1R)$-primary ideal. Hence $\rij xdR$ is $\maxim$-primary.
\end{proof}

\begin{theorem}\label{T:pdim} 
Let $(R,\maxim)$ be a local ring of 
finite embedding dimension. The following numbers are
all equal.
\begin{itemize}
\item the \pdim\ $d$ of $R$;
\item the least possible number of elements $d'$ generating an 
$\maxim$-primary ideal;
\item the dimension $\complet d$ of the completion $\complet R$ of $R$;
\item the dimension $\overline d$ of the graded ring $\gr {}R$ associated to $R$;
\item  the degree $\underline d$ of the \emph{Hilbert-Samuel 
polynomial} $\op{HS}_R$, where $\op{HS}_R$ is the
unique
polynomial  with rational coefficients for which $\op{HS}_R(n)$
equals 
the length of $R/\maxim^n$ for all large $n$;
\item the \pdim\ $\sep d$ of the separated quotient $\sep R$;
\end{itemize}
\end{theorem}
\begin{proof} 
The equality of $d$ and $d'$ is given by 
Lemma~\ref{L:mingen}. We already observed that $\gr {}R$ and
$\complet R$ are Noetherian and that we have isomorphisms 
$\maxim^n/\maxim^{n+1}\iso \maxim^n\complet
R/\maxim^{n+1}\complet R$ for all $n$. Hence  $\op{HS}_R=\op{HS}_{\complet 
R}$ and $\gr{}R\iso \gr{}{\complet R}$. It follows
that $\underline d=\complet d$, by the Hilbert-Samuel theory and that 
$\overline d=\complet d$ by \cite[Theorem
13.9]{Mats}. This shows already that $\overline d=\complet d=\underline d$.

Let $\rij yd$ be a tuple in $R$ generating an $\maxim$-primary ideal. 
Since $\rij yd\complet R$ is then $\maxim\complet
R$-primary,  $\complet d\leq d$. Finally, let $\rij \xi{\overline d}$ be 
a homogeneous system of parameters of $\gr{}R$ and
choose $x_i\in R$ such that $\xi_i=\op{in}(x_i)$. Let $I:=\rij x{\overline 
d}R$. By \cite[Exercise 5.3]{Eis}, we have an
isomorphism
\begin{equation*}
\gr{}R/\op{in}(I) \iso \gr{}{R/I}.
\end{equation*} Since $\rij\xi{\overline d}\gr{}R\sub \op{in}(I)$, we see 
that $\gr{}R/\op{in}(I)$ is Artinian, whence so is
$\gr{}{R/I}$. This in turn means that $R/I$ has a nilpotent maximal 
ideal, so that $d\leq\overline d$ by definition of
\pdim. This proves that the  first five numbers in the statement are 
equal. That they are also equal to the last,
$\sep d$, follows by applying the result to $\sep R$ together 
with the fact that $R$ and $\sep R$ have the same
completion.
\end{proof}

\begin{remark}\label{R:mult} 
If the leading coefficient of the 
Hilbert-Samuel polynomial is written as $e/d!$, with
$d:=\pd R$, then we call $e$ the \emph{multiplicity} of $R$ and we denote it
$\mult R$. It  follows that $R$ has the same multiplicity as its
completion and as its separated quotient.
\end{remark}

\begin{corollary}
If $R$ is a local ring of \pdim\ one, then there exists $N\in\nat$ such that every closed ideal is the closure of an $N$-generated ideal.
\end{corollary}
\begin{proof}
By Theorem~\ref{T:pdim}, the completion $\complet R$ is a one-dimensional Noetherian local ring, and hence by the Akizuki-Cohen theorem (\cite{AkiId,CohId}), there is some $N$ such that every ideal in $\complet R$ is generated by at most $N$ elements. Let $I\sub R$ be an arbitrary ideal. Since $I\complet R$ is generated by at most $N$ elements, we may choose by Nakayama's Lemma $a_1,\dots,a_N\in I$ such that $I\complet R=\rij aN\complet R$. Contracting this equality back to $R$ shows, by Lemma~\ref{L:quot}, that $I$ is the closure of $\rij aNR$.
\end{proof}

It is well-known that one may take $N$ to be equal to the multiplicity of $R$, in case the latter is \CM. In view of Remark~\ref{R:mult} and our definition in \S\ref{s:cata} below, the same holds true under the assumption that $R$ is \cCM.

\subsection{Generic sequences} 
A tuple $\tuple x$ is 
called \emph{generic}, if it generates an $\maxim$-primary ideal and its
length is equal to the \pdim\ of $R$; it   is called \emph{part of a generic
sequence}, if it can be
extended to a generic sequence.   If $x$ is a single element
which is part of a generic sequence, then we simply call $x$ a \emph{generic
element}.

\begin{lemma}\label{L:gen} 
Let $(R,\maxim)$ be a local ring of \pdim\ 
$d$. A tuple $\rij xe$   is part of
a generic sequence  \iff\ $R/\rij xeR$ has \pdim\ $d-e$.

In particular, $x$ is
generic \iff\ $\pd{R/xR}=\pd R-1$.
\end{lemma}
\begin{proof} 
Suppose $\rij xe$ is part of a generic sequence and enlarge it to a generic
sequence $\rij xd$. One checks that  (the image of) $(x_{e+1},\dots,x_d)$   is a generic
sequence in $R/\rij xeR$. This
shows that $\pd{R/\rij xeR}=d-e$. Conversely, assume $\pd{R/\rij xeR}=d-e$.
Choose a tuple $(x_{e+1},\dots,x_d)$ in $R$ so that its image in 
$R/\rij xeR$ is a generic sequence. Since $\rij xd$ generates an $\maxim$-primary
ideal and has length $d$, it is generic.
\end{proof}

\begin{proposition}\label{P:sop} 
Let   $(R,\maxim)$ be a local ring 
of finite embedding dimension. A sequence  in $R$ is generic \iff\
its image in $\complet R$ is a system of parameters. 
\end{proposition}
\begin{proof} 
One direction has already been noted, so let $\tuple 
x$ be a tuple in $R$ whose image in $\complet R$ is a system of parameters. By
Theorem~\ref{T:pdim}, the \pdim\ of $R$ is   equal to the length of 
this tuple. Let $J:=\tuple xR$. By
Lemma~\ref{L:quot}, the completion of $R/J$ is $\complet R/J\complet 
R$. As the latter is Artinian, so must the former
be by Lemma~\ref{L:Art}, showing that $\tuple x$ is generic.
 \end{proof}

It follows that   $\rij xd$ is generic \iff\ so is
$(x_1^{n_1},\dots,x_d^{n_d})$. However, this does in general not imply that 
$(\op{in}(x_1),\dots,\op{in}(x_d))$ is a system of
parameters in $\gr {}R$ (this even fails in the Noetherian case as 
the example $\{\xi^2,\xi\zeta+\zeta^3\}$ in $\pow k{\xi,\zeta}$ shows).
Immediately from Proposition~\ref{P:sop} and \cite[Theorem 14.5]{Mats} we get:

\begin{corollary}\label{C:anind}
Any generic sequence $\tuple x$ in $R$ is \emph{analytically independent} in the
sense
that if $F(\xi)$ is a homogeneous form over $R$ such that $F(\tuple x)=0$, then all
coefficients of $F(\xi)$ lie in the maximal ideal of $R$.\qed
\end{corollary}

\subsection{Threshold primes}\label{s:tre}   
By Proposition~\ref{P:sop}, $x$ is
generic \iff\ the image of $x$ 
in $\complet R$ is part of a system of
parameters. More concretely, let $d$ be the \pdim\ of $R$ and let 
$\pr_1,\dots\pr_s$ be the $d$-dimensional   prime
ideals of $\complet R$. Note that $\complet R$ itself has dimension 
$d$ by Theorem~\ref{T:pdim}, so that all its
$d$-dimensional primes are   minimal (but there may be other minimal 
prime ideals, of lower dimension). We call the
$\mathfrak q_i:=\pr_i\cap R$ the \emph{threshold} primes of $R$. By 
Corollary~\ref{C:spec}, every threshold prime $\mathfrak q$ is
closed and contains no proper closed prime ideals. Moreover, $R/\mathfrak q$ 
has the same \pdim\ as $R$ by Theorem~\ref{T:pdim},
since $\complet R/\mathfrak q\complet R$ has the same dimension as $\complet
R$.  By a \emph{threshold prime} of an ideal $I$, we mean a threshold prime of
its residue ring $R/I$. Proposition~\ref{P:sop} yields the following criterion for
genericity.

\begin{corollary}\label{C:tre}
An element $x\in R$ is generic  \iff\ it is not contained in any threshold
prime of $R$. In particular, the product of any two generic elements is again 
generic. \qed
\end{corollary}

\begin{corollary}\label{C:genprim}
Any $\maxim$-primary ideal  contains a generic sequence. More precisely, if
$R$ is a $Z$-algebra and $I\sub Z$ an ideal such that $IR$ is
$\maxim$-primary, then there exists a tuple $\tuple x$ over $Z$ with entries in
$I$ such that its image in $R$ is a generic sequence.
\end{corollary}
\begin{proof}
We prove the last assertion by induction on the \pdim\ $d$ of 
$R$. Since there is nothing
to show if $d=0$, we may assume $d>0$. Let $\mathfrak q_1,\dots,\mathfrak q_s$
be the threshold primes of $R$. Towards a contradiction, suppose $I$ is contained
in the union
of the $\mathfrak q_i\cap Z$. By prime avoidance, there is some $i$ such that
$I\sub\mathfrak q_i\cap Z$. But then $IR\sub\mathfrak q_i$, forcing $\mathfrak
q_i=\maxim$, thus contradicting by Corollary~\ref{C:tre} that $d>0$. Hence
there exists $x\in I$ so that its image in $R$ lies outside every threshold prime
of $R$, and therefore is generic by Corollary~\ref{C:tre}. By
Lemma~\ref{L:gen}, the \pdim\ of $R/xR$ is $d-1$. Therefore, by
induction, we can find a tuple $\tuple y$ of length $d-1$ with entries in $I$ so that its
image in
$R/xR$ is generic. The desired sequence is now given by  adding $x$ to this tuple
$\tuple y$.
\end{proof}

In \cite{HamMar}, the authors introduce the notion of a \emph{strong
parameter sequence}. It should be noted that this is different from our
present notion of generic
sequence. For example, if $V$ is an ultra-\DVR\ (see Example~\ref{e:preg} for
more details), and $x$ a non-zero infinitesimal in $V$, then $x$ is
$V$-regular, whence a strong parameter  by \cite[Proposition 3.3(f)]{HamMar},
but $x$ is clearly not generic (in fact, the unique threshold prime of $V$ is
the ideal of infinitesimals $\inf V$).

\subsection{\Cheight}
Given an ideal
$I$ in a local ring $(R,\maxim)$ of finite embedding dimension, we call its
\emph{\cheight} the maximal length of a tuple in $I$ that is part of a generic
sequence and we denote it $\cht I$. In particular, an ideal is $\maxim$-primary
\iff\ its \cheight\ equals the \pdim\ of $R$.
Our terminology is justified by the next result.

\begin{proposition}\label{P:ght}
Let $(R,\maxim)$ be a local ring of finite embedding dimension. For every ideal
$I\sub
R$, we have an equality $\cht I=\pd R-\pd{R/I}$.
\end{proposition}
\begin{proof}
%
Let $d$   be the \pdim\ of  $R$ and let $h$ be the \cheight\ of $I$. Choose a
tuple $\tuple y$ in $I$ of length $h$ which is part of a generic sequence of
$R$. Put $S:=R/\tuple yR$, so that $\pd S=d-h$ by Lemma~\ref{L:gen}. Since $IS$
contains no generic element, it must be contained in some threshold prime
$\mathfrak q$ of $S$ by
Corollary~\ref{C:tre}. From the inclusions $IS\sub\mathfrak q$ we get $\pd S\geq
\pd{S/IS}\geq
\pd{S/\mathfrak
q}=\pd S$, and hence all these \pdim{s} are equal to $d-h$. Since   $S/IS=R/I$,
we are done.
\end{proof}

\subsection{Parameter degree and degree}\label{s:pardeg}
We conclude this section with another genericity criterion, in terms of an  
invariant which was introduced  for Noetherian rings in 
\cite{SchABCM,SchMixBCM} and which will play a crucial role
in what follows. The \emph{parameter degree} of a
local ring $R$ of finite embedding dimension is by definition the minimal 
length of a residue ring $R/\tuple xR$, where $\tuple x$ runs
over all possible generic sequences of $R$. We denote the parameter
degree of $R$ by $\pardeg R$.
We will show in Lemma~\ref{L:pardegmult} below that the multiplicity of $R$ is
bounded by its parameter degree  and indicate when they are equal.

Closely related to this is an invariant, which for want of a better name, we 
call \emph{degree} and which is defined as follows. Let $R$ be a
 local ring of \pdim\ $d\geq 1$. We define the \emph{degree}
$\op{deg}_R(x)$ of an element $x$  to be the least possible length of a residue
ring $R/(xR+\tuple yR)$,
where $\tuple y$ runs over all tuples of length $d-1$ inside the maximal
ideal. Hence, if $x$ is a unit, its degree is zero;  if $x$  is generic, its
degree  is the parameter degree of $R/xR$; and in the remaining case,
its degree is infinite. In particular,  we showed:

\begin{corollary}\label{C:degfin}
An non-unit in a non-Artinian local ring $R$ of finite embedding dimension is
generic
\iff\
its degree is finite. Moreover, the parameter degree of $R$ is the minimum of the degrees of all
non-units in $R$.
\qed
\end{corollary}

\section{Extended dimensions}\label{s:bluedim}

In this section, we introduce several other dimension notions for a local
ring $(R,\maxim)$. With an \emph{extended dimension}, we mean an invariant
on the class of local rings taking values in $\bar\nat:=\nat\cup\{\infty\}$ which
agrees with Krull dimension on the subclass of all Noetherian local rings. 
Clearly, Krull dimension itself is an extended dimension, and so is \pdim\ by
the results from the previous section.
Note, however, that embedding dimension is \emph{not} an extended dimension.

Recall that a partially ordered set $\Gamma$ has
\emph{combinatorial dimension} (or, \emph{height}) $d$ if any proper
(ascending) chain
in $\Gamma$
 has length at most $d$ (meaning that it contains at most 
$d+1$ elements). Hence, the dimension of a ring $A$ is the 
combinatorial dimension of $\spec A$ (the set of
all prime ideals ordered by inclusion). Given ideals $J\sub \pr$ in $A$ with $\pr$ prime, we say that $\pr$ is an
\emph{associated} prime of $J$ if $\pr$ is of the form $(J:a)$;  a \emph{minimal}
prime of $J$ if no prime ideal is properly contained
between $J$ and $\pr$; and a \emph{minimal associated} prime of $J$ if it is
associated and no associated prime of $J$  is properly contained between $J$ and
$\pr$.

\subsection{Cl-dimension}
Let $\clspec R$ be the subset of $\spec R$ consisting of all closed prime
ideals of $R$. Note that the maximal ideal as well as the threshold primes (see
\S\ref{s:tre}) belong to $\clspec R$. In fact, we showed in Corollary~\ref{C:spec}
that $\clspec R$ is  the image of the canonical map $\spec{\complet
R}\to \spec R$.  We call the combinatorial dimension of $\clspec R$ the
\emph{cl-dimension} of $R$ and denote it $\cld R$. It is clear that $\cld
R=\kd R$ when $R$ is Noetherian, showing that
cl-dimension is an extended dimension.

\subsection{Fr-dimension}\label{s:frdim}
We say that an ideal $I\sub R$ is \emph{$n$-generated}, if there exists a
tuple $\tuple x$ of length $n$ such that $\tuple xR=I$. We say that an ideal
$\id\sub R$ is \emph{\con n} if it is of the form $\id=(I: a)$ with $I$ an
$n$-generated ideal. An ideal $\id$ is called \emph\fr\ if it is \con n for
some $n<\infty$. Let $\conspec R$ be the subset of $\spec R$ consisting of
all \fr\ prime ideals, that is to say, all associated prime ideals of finitely generated ideals of $R$. We call the combinatorial dimension of $\conspec R$ the
\emph{fr-dimension} of $R$ and denote it $\frd R$. When $R$ is
Noetherian, every ideal is \fr\ whence $\frd R=\kd R$, showing that
fr-dimension is an extended dimension. We define the related notion of a \emph{strongly \fr} prime ideal as a prime ideal $\pr$ of the form $(I:a)$ with $I$ finitely generated and $a\notin\pr$. A priori, not every \fr\ prime ideal is strong, but see Corollaries~\ref{C:loculpow} and \ref{C:locup}.

\subsection{Pi-dimension}
We say that $R$ has
\emph{pi-dimension} at most $d$, if
$\maxim$ is a minimal associated prime of a $d$-generated ideal. The
pi-dimension, $\pid R$, of $R$ is then the least $d$ such that $R$ has pi-dimension at most $d$. That  pi-dimension is an extended dimension follows from   Krull's  Principal
Ideal theorem  (from which it borrows its name; see for instance
\cite[Theorem 8.10]{Mats}).

\begin{theorem}\label{T:dimineq}
For an arbitrary local ring $(R,\maxim)$, we have the following inequalities
between extended dimensions:
\begin{enumerate}

\item\label{i:kd} $\frd R,\cld R\leq \kd
R$;
\item\label{i:pidpd} $\pid R\leq \pd R$;
\item\label{i:cld} $\cld R \leq \pd
R$, with equality if $\pd R$ is finite.
\end{enumerate}
Moreover,  each of these inequalities can be strict.
\end{theorem}
\begin{proof}
Inequalities \eqref{i:kd} are immediate from the definition.  In order to
show inequality~\eqref{i:pidpd}, we may   assume that 
$\pd R=d<\infty$. By definition,
 $R/I$ is an Artinian local ring for some $d$-generated ideal
$I$. It follows that $\maxim$ is a minimal associated prime of $I$, whence the
pi-dimension of $R$ is at most $d$.   

So remains to prove
\eqref{i:cld}. There is nothing to show if $R$ has infinite \pdim, so assume
$R$ has finite \pdim, say, $d$ (whence also finite embedding dimension). By
Corollary~\ref{C:spec}, there is
 a surjective map $\spec{\complet R}\to \clspec
R$. In particular, the combinatorial dimension of $\clspec R$ is at 
most the dimension of $\complet R$, that is to say,
in view of Theorem~\ref{T:pdim}, at most $d$. So remains to prove the other
inequality by  induction on 
$d$. There is nothing to show if $d=0$, so we may
assume $d>0$. By Corollary~\ref{C:spec}, the minimal elements in 
$\clspec R$ are the contractions of the minimal primes
of $\complet R$. Hence there are only finitely many of them, all 
different from the maximal ideal $\maxim$. By prime
avoidance, we may choose $x\in\maxim$ outside all these finitely many 
prime ideals. In particular, since the threshold
primes are among these, $x$ is generic and hence $R/xR$ has \pdim\ 
$d-1$. By induction, the combinatorial dimension of
$\clspec{R/xR}$ is $d-1$. By Lemma~\ref{L:quot}, the completion of $R/xR$ is
$\complet R/x\complet R$. The \homo\ $\complet
R\to\complet R/x\complet R$ induces an injection $\spec{\complet 
R/x\complet R}\into \spec{\complet R}$, whose image
is the subset of all prime ideals of $\complet R$ containing $x$. It 
follows that the canonical injection
$\spec {R/xR}\into \op {Spec}R$ maps   $\clspec{R/xR}$ into the 
subset of $\clspec R$ consisting of all closed prime
ideals containing $x$. Using this and the fact that the combinatorial 
dimension of $\clspec {R/xR}$ is $d-1$, we can
find a proper chain of closed primes ideals  $\mathfrak 
q_1\varsubsetneq\mathfrak
q_2\varsubsetneq\dots\varsubsetneq\mathfrak q_d=\maxim$ in $R$ with 
$x\in\mathfrak q_1$. Let $\mathfrak q_0$ be a
minimal element  of $\clspec R$ lying inside  $\mathfrak q_1$. Since 
by construction $x\notin\mathfrak q_0$, the
$\mathfrak q_i$ form a proper chain of length $d$, showing that the 
combinatorial dimension of  $\clspec R$ is at least
$d$. This proves \eqref{i:cld}.

Finally, the local ring in Example~\ref{E:pid} (respectively, in
Example~\ref{E:cld}) shows that in general, the
inequalities~\eqref{i:kd} and \eqref{i:pidpd} (respectively, inequality~\eqref{i:cld})
are strict. 
\end{proof}

\begin{example}\label{E:pid}
Let  $\ul R$ be the
ultraproduct (see \S\ref{s:ul} for more details) of the $A/\pr^n$ for
$n=1,2\dots$, where
$(A,\pr)$ is a $d$-dimensional Noetherian local ring, for $d>0$. Its  pi-dimension
and fr-dimension are equal to zero, its  \pdim\ and cl-dimension are equal
to $d$, and its Krull dimension is infinite. 
\end{example}

\begin{example}\label{E:cld}
Let   $(\ul R,\ulmax)$ be the
ultraproduct of the $A_n/\maxim_n^2$ for $n=1,2\dots$, where $(A_n,\maxim_n)$ is the power series
ring over a field $k$ in  $n$ indeterminates.
Since $\ulmax^2=0$ in $\ul R$, the
 local ring $\ul R$ has cl-dimension  and Krull dimension   equal
to zero, but its embedding dimension, whence its \pdim, is  infinite.
\end{example}

There is a more instructive way to see \eqref{i:pidpd}: the
\pdim\
of a local ring $(R,\maxim)$ of finite embedding dimension
is at most $d$ \iff\ $\maxim$ is a minimal prime of a $d$-generated ideal
(that is to say, the same definition as for pi-dimension, but omitting the
term `associated').

 Let `e-dim' be some  extended dimension. We call e-dim
\emph{first-order} if
the property $\op{e-dim}(\cdot)=d$ is first-order in the sense of \S\ref{s:fop}
below, for every $d\in\nat$. Moreover, to prove this, it suffices to show that
the property $\op{e-dim}(\cdot)\leq d$ is first-order.

\begin{lemma}\label{L:fodim}
Fr-dimension and pi-dimension are first-order; \pdim, cl-dimension and Krull
dimension are not. 
\end{lemma}
\begin{proof}
The assertion is obvious for pi-dimension, since we can express in a first-order way
that the maximal ideal $\maxim$ of a local ring
is of the form $(I:a)$ for some $d$-generated ideal $I$ such
that no prime ideal of the form $(I:b)$ is properly contained in $\maxim$  (note
that $\maxim$ admits a
first-order definition as the collection of all non-units). As
for fr-dimension,   for each $n$, let $\tau_{n,d}$ be the statement expressing
that there does
not exist a proper chain of length $d+1$ consisting of \con n prime ideals.
Hence a local ring has fr-dimension at most $d$ \iff\ $\tau_{n,d}$ holds in
it, for all $n$. 

The local ring in Example~\ref{E:pid} shows that Krull dimension,
cl-dimension and
\pdim\ are not first-order. 
\end{proof}

\section{Ultra-Noetherian rings}\label{s:ul}

Before we further develop the `local algebra' of local rings of finite embedding
dimension, we introduce an important subclass, arising as
ultraproducts of Noetherian local rings. Fix an infinite index set $W$ and a 
non-principal ultrafilter on $W$.  We will moreover assume that the ultrafilter
is countably incomplete. This is equivalent  with the existence of a function
$f\colon W\to\nat$ such that for each $k$,  the set of all $w\in W$ for
which $f(w)\geq k$  belongs to the ultrafilter.  If $W$ is countable, then any
non-principal ultrafilter is countably incomplete, and this is the situation 
we will find ourselves in all applications.\footnote{In fact, it is consistent
with ZF to assume that every non-principal ultrafilter on any infinite set is
countably incomplete. Moreover, for most of what we say, we will not need
to assume that the ultrafilter is countably incomplete; it is only used
explicitly in Lemma~\ref{L:ulcomp} below.} 
For each $w\in W$, let $\seq  Rw$ be a local ring and let $\ul R$ be the 
\emph{ultraproduct} of the $\seq Rw$ (for a quick review on ultraproducts, see
\cite[\S1]{SchNSTC}; for
more details see for instance, \cite{EkUP,Hod,Roth,SchUlBook}). It is important to note that  $\seq Rw$   are
not uniquely defined by $\ul R$ (not even almost all; see the example  in \S\ref{s:fop}). 
By
\los, $\ul R$ is a local ring with maximal ideal $\ulmax$ equal to the 
ultraproduct of the maximal ideals $\seq\maxim w$. If for some $m$, almost all
$\seq Rw$ have embedding dimension at most $m$, then we say  that the $\seq Rw$
have \emph{bounded embedding dimension}; a similar usage
will be applied to other numerical invariants. Hence if the $\seq Rw$ have
bounded embedding dimension, then $\ul R$ has finite embedding dimension, whence
finite \pdim.   
In case all $\seq Rw$ are equal to a single local ring $R$, we refer to $\ul R$  
as the \emph{ultrapower} of $R$.

When dealing with ultraproducts, \los\ is an extremely useful tool for
transferring properties between almost all $\seq Rw$ and $\ul R$. However,
this only applies to first-order properties (see \S\ref{s:fop} below for more
details). In view of this, we introduce the following more general set-up for discussing transfer through
ultraproducts.  Let $\mathbf P$ be a property of local rings of finite embedding
dimension
and let $R$ be a local ring. We call $R$ \emph{\c-$\mathbf P$} if it has finite embedding dimension and its completion
has property $\mathbf P$. In particular, by Theorem~\ref{T:finembcomp},  
any such ring is, in our newly devised terminology,  \c-Noetherian. We
call a local ring   \emph{ultra-$\mathbf P$} if it is equal to an
ultraproduct $\ul R$ of local rings $\seq
Rw$ of bounded embedding dimension  almost all of which satisfy property
$\mathbf P$. In particular, $\ul R$
has finite embedding dimension too. In fact, according to this terminology, 
an \emph{ultra-ring} is any ultraproduct of local rings   of bounded embedding
dimension; and
 an \emph{ultra-Noetherian} ring is   any ring isomorphic to an ultraproduct of
Noetherian local rings   of bounded embedding dimension.   It is important to 
notice that the well-known duality between rings and affine
schemes breaks down under ultraproducts:

\begin{proposition}\label{P:ulscheme}
Let $\seq Rw$ be Noetherian local rings of bounded embedding dimension and let
$\ul R$ be their ultraproduct. Then the ultraproduct of the $\spec{\seq Rw}$ is
equal to $\conspec{\ul R}$.
\end{proposition}
\begin{proof}
Recall that $\conspec{\ul R}$ consists of all \fr\ prime ideals of $\ul R$
(see \S\ref{s:frdim}).
If $\ul I$ is a finitely generated ideal in $\ul R$, say of the form
$(\ul{x_1},\dots,\ul{x_n})\ul
R$, and if $\seq{x_i}w\in\seq Rw$ are such that their ultraproduct is
equal to $\ul {x_i}$, then the ultraproduct of the ideals
$\seq Iw:=(\seq{x_1}w,\dots,\seq{x_n}w)\seq Rw$ is equal to $\ul I$. Moreover,
if $\ul y\in \ul R$ is the ultraproduct of elements $\seq yw\in\seq Rw$, then
$(\ul I:\ul y)$ is equal to the ultraproduct of the $(\seq Iw:\seq yw)$. Since $(\ul
I:\ul y)$ is prime, so are almost all $(\seq Iw:\seq yw)$ by \los. Hence any \fr\
prime ideal in $\ul R$ lies in the ultraproduct of the $\spec{\seq Rw}$. 

Conversely, for each $w$, let $\seq \pr w$ be a prime ideal in $\seq Rw$, and 
let $\ul\pr$ be their ultraproduct. By \los, $\ul\pr$ is prime. Since the
$\seq Rw$ have bounded embedding dimension, they also have bounded dimension.
Therefore, there is a $d$ such that  almost each $\seq Rw$ has dimension $d$ (in
the terminology
of \S\ref{s:uldim} below, $d$ is the ultra-dimension of $\ul R$). By 
Krull's Principal Ideal theorem, almost each $\seq\pr w$ is 
$d$-related, whence so is $\ul\pr$ by \los.
\end{proof}

In particular, the ultraproduct of the $\spec{\seq Rw}$ does not depend on
the choice of the $\seq Rw$ having as ultraproduct $\ul R$. 
The local algebra of rings of finite embedding dimension is hampered by the fact that very few localizations have finite embedding dimension. We will discuss one case here (see Corollary~\ref{C:locpreg} for another one). We first prove a bound for Noetherian   rings.\footnote{In \S\S\ref{s:ub} and \ref{s:ua}, we adopt the reverse strategy, by developing bounds from our local algebra results.} For a Noetherian ring $A$,  let $\gamma(A)\in\nat\cup\{\infty\}$ be the supremum of all $\ed{A_\pr}$, where $\pr$ runs through all   prime ideals of $A$.

\begin{proposition}\label{P:unifembdim}
If $A$ is a $d$-dimensional, excellent ring, then   $\gamma(A)<\infty$. In fact, if $A$ is equi\ch\ and local, then $\gamma(A)\leq d+\rho$, where   $\rho$ is the parameter degree of $A$.
\end{proposition}
\begin{proof}
We prove the first statement by induction on $d$. 
Let $\pr_1,\dots,\pr_s$ be the minimal prime ideals of $A$, and let $N$ be a bound on their number of generators. Since any prime ideal $\pr$ contains one of the $\pr_i$,  we see that $\gamma(A)$ is bounded by the maximum of all $\gamma(A/\pr_i)+N$.  Hence we may assume without loss of generality that $A$ is an excellent domain. Therefore,  its regular locus is non-empty and open. Let $U=\op{Spec}A_f$ be a non-empty affine open contained in the regular locus of $A$. By regularity, $\ed{A_\pr}\leq d$, for any $\pr\in U$, and so we only need to show a bound for  those prime ideals containing $f$. Put $\bar A:=A/fA$. Note that $\bar A$ has  Krull dimension $d-1$ and is again excellent, so that by induction $\gamma(\bar A)<\infty$.  Therefore, for any prime ideal $\pr$ of $A$ containing $f$, we have an estimate $\ed{A_\pr}\leq\gamma(\bar A)+1$, finishing the proof of the first assertion.

Assume next that $A$ is moreover equi\ch\ and  local, with parameter degree $\rho$. I claim that $\gamma(A)\leq\gamma(\complet A)$, where $\complet A$ is the completion of $A$. Assuming the claim, we may take $A$ to be complete, since parameter degree does not change under completion. By the Cohen structure theorem, $A$ contains a $d$-dimensional regular local subring $R$ over which it is finite. Moreover, by \cite[Proposition 3.5]{SchABCM}, we may choose $R$ so that $A$ is generated by $\rho$ elements as an $R$-module. Let $\pr$ be a prime ideal in $A$ and put $\primary:=\pr\cap R$.  By base change, the fiber ring $A_\primary/\primary A_\primary$ has dimension $\rho$ over the residue field of $\primary$. Moreover, $A_\pr/\primary A_\pr$ is a direct summand of $A_\primary/\primary A_\primary$ by the structure theorem of Artinian local rings (\cite[Corollary 2.16]{Eis}), whence has length at most $\rho$. In particular, $\ed{A_\pr/\primary A_\pr}\leq \rho$. Since $R$ is regular, $\primary R_\primary$ is generated by at most $d$ elements, whence so is $\primary A_\pr$. It follows that $\ed{A_\pr}\leq \rho +d$, as we wanted to show.

To prove the claim, let $\mathfrak q$ be a minimal prime ideal of $\pr \complet A$. Since $A/\pr$ is excellent, its completion $\complet A/\pr\complet A$ is reduced. Therefore, the localization of $\complet A/\pr\complet A$ at $\mathfrak q$ is a field, showing that $\pr\complet A_{\mathfrak q}=\mathfrak q\complet A_{\mathfrak q}$, an ideal generated by at most $\gamma(\complet A)$ elements. Since $A_\pr\to \complet A_{\mathfrak q}$ is faithfully flat, $\pr A_\pr$ is therefore also generated by at most $\gamma(\complet A)$ elements, showing that $\gamma(A)\leq\gamma(\complet A)$.
\end{proof}

\begin{corollary}\label{C:loculpow}
If $R$ is an excellent local ring, then any localization of its ultrapower  $\ul R$ at a \fr\ prime ideal has finite embedding dimension. Moreover, every \fr\ prime ideal of $\ul R$ is strong.
\end{corollary}
\begin{proof}
Let  $\pr$ be a \fr\ prime ideal of $\ul R$. By Proposition~\ref{P:ulscheme}, we can find prime ideals $\seq\pr w$ in $R$ with ultraproduct equal to $\pr$. Let $\gamma(R)$ be the bound given by Proposition~\ref{P:unifembdim} on the embedding dimension of all $R_{\seq\pr w}$. Since $(\ul R)_\pr$ is the ultraproduct of the   $R_{\seq\pr w}$, its embedding dimension is at most $\gamma(R)$ as well. In fact, we can find ideals $\seq Iw\sub\seq \pr w$ generated by at most $\gamma(R)$ elements, so that $\seq IwR_{\seq\pr w}=\seq\pr wR_{\seq\pr w}$. Hence, there exists $\seq aw\notin\seq\pr w$, such that $(\seq Iw:\seq aw)=\seq\pr w$. Taking ultraproducts, we see that $\pr$ is strongly \fr\ (see \S\ref{s:frdim} for the definition).
\end{proof}

In fact, we have the following more general version of the second assertion.

\begin{proposition}\label{P:sfr}
A \fr\ prime ideal $\pr$ in an ultra-Noetherian local ring   $\ul R$ is strongly \fr\ \iff\ $(\ul R)_\pr$ has finite \pdim.
\end{proposition}
\begin{proof}
Note that a local ring has finite \pdim\ \iff\ it has finite embedding dimension. One direction is true in any ring $A$: if $\pr$ is strongly \fr, say, of the form $(I:s)$ with $I\sub A$ finitely generated and $s\notin\pr$, then $\pr A_\pr=IA_\pr$, showing that $A_\pr$ has finite embedding dimension.

Conversely, suppose $(\ul R)_\pr$ has finite \pdim, whence finite embedding dimension. In particular, there exists a finitely generated ideal $I\sub \pr$ such that $I(\ul R)_\pr=\pr(\ul R)_\pr$. By \los\ and Proposition~\ref{P:ulscheme}, we can find   ideals $\seq Iw\sub\seq \pr w$ so that their respective ultraproducts are $I$ and $\pr$. In particular, almost all $\seq\pr w$ are prime and $\seq Iw(\seq Rw)_{\seq\pr w}=\seq\pr w(\seq Rw)_{\seq\pr w}$, for almost all $w$. Hence, we can find $\seq sw\notin\seq\pr w$ such that $\seq\pr w=(\seq Iw:\seq sw)$. Letting $\ul s$ be the ultraproduct of the $\seq sw$, we get $\pr=(I:\ul s)$ and $\ul s\notin\pr$, showing that $\pr$ is strong.
\end{proof}

\subsection{First-order properties}\label{s:fop}
A property $\mathbf P$ of rings is called \emph{first-order} if  there
exists a first-order theory $\Pi$, in the language of rings, such that $R$ is a
model of $\Pi$ \iff\ $R$ satisfies $\mathbf P$.   \los\ states that if
$\mathbf P$ is
first-order, then   ultra-$\mathbf P$ implies $\mathbf P$. Although we will not use this here, the
converse is also true,  due to a theorem of Keisler-Shelah (see for instance
\cite[Theorem 9.5.7]{Hod}). It follows that if $\mathbf P$ is not first-order,
then there exists an ultra-ring $\ul S$ which is at the same time ultra-$\mathbf
P$
and ultra-non-$\mathbf P$. Indeed, by what we just said, there exist $\seq
Rw$ of bounded embedding dimension satisfying $\mathbf
P$ so that there ultraproduct $\ul R$ does not satisfy $\mathbf P$. Let $\ul S$ be any
ultrapower of $\ul R$. Since $\ul S$ is then also an ultraproduct of the $\seq
Rw$, but for a larger underlying index set, $\ul S$ is both ultra-$\mathbf P$
and ultra-non-$\mathbf P$. 

For an ultra-Noetherian example, consider the property $\mathbf
C_0$: `being a Noetherian local ring of \ch\
zero'. The ultraproduct $\ul V$ of all the rings of $p$-adic integers $\zet_p$
(with respect to some non-principal ultrafilter on the set of prime numbers)
is ultra-$\mathbf C_0$, but by the Ax-Kochen-Ershov theorem, this ring can also
be realized as the ultraproduct of non-$\mathbf C_0$ local rings, to wit, the
$\pow{\mathbb F_p}t$, where $t$ is a single indeterminate and $\mathbb F_p$ is the
$p$ element field (see also Example~\ref{E:ulNN} below).  

\subsection*{Cataproducts} 
Let $\seq Rw$ be Noetherian local rings of bounded embedding dimension and let
$\ul R$ be their ultraproduct. The separated quotient
 of $\ul R$, that is to say, the factor ring $\ulsep R:=\ul R/\inf {\ul R}$, is
called the \emph{cataproduct} of the $\seq Rw$. If all $\seq Rw$ are equal to a single ring $R$, then we call $\ulsep R$ the \emph{catapower} of $R$. This terminology is justified
by:

\begin{lemma}\label{L:ulcomp} 
The cataproduct of  local 
rings of bounded embedding dimension is equal to the  completion of their
ultraproduct, whence in particular is  Noetherian.
\end{lemma}
\begin{proof} 
Let $(\ul R,\ulmax)$ be the ultraproduct of
Noetherian local rings $(\seq Rw,\seq\maxim w)$ of embedding dimension at most
$e$,
and let $\ulsep R$ be their cataproduct, that is to say, $\ul R/\inf{\ul R}$.
We start with showing
that any Cauchy sequence $\ul{\tuple a}\colon\nat\to \ul R$ has a limit.  After
taking
 a subsequence if necessary, we may assume that
$\ul{\tuple a}(n)\equiv\ul{\tuple a}(n+1)\mod\ulmax^n$, for all $n$.
 For each $n$, choose
  $\seq {\tuple a}w(n)\in
\seq Rw$ such that their ultraproduct is equal to $\ul{\tuple a}(n)$.
 By \los, we have for a fixed $n$ that
\begin{equation}\label{eq:almostCS}
\seq {\tuple a}w(n)\equiv \seq {\tuple a}w(n+1)\mod\seq\maxim w^n
\end{equation}
for almost all $w$, say, for all $w$ in $D_n$. I claim that we can modify the 
$\seq
{\tuple a}w(n)$ in such way that \eqref{eq:almostCS} holds for all $n$ and all
$w$. More precisely, for each $n$ there exist
$\seq{\tilde{\tuple a}}w(n)$ with ultraproduct equal to $\ul{\tuple a}(n)$, such
that 
\begin{equation}
\label{eq:allCS}
\seq {\tilde{\tuple a}}w(n)\equiv \seq {\tilde{\tuple a}}w(n+1)\mod\seq\maxim w^n
\end{equation}
for all $n$ and $w$.
We will construct the $\seq{\tilde{\tuple a}}w(n)$ recursively from the
$\seq{\tuple a}w(n)$.
When $n=0$, no modification is required 
(since by assumption $\seq\maxim w^0=\seq Rw$), and hence we set $\seq{\tilde{\tuple
a}}w(0):=\seq{{\tuple a}}w(0)$ and $\seq{\tilde{\tuple
a}}w(1):=\seq{{\tuple a}}w(1)$. So assume we have defined already the
$\seq{\tilde{\tuple a}}w(j)$ for $j\leq n$  such that \eqref{eq:allCS}  holds
for all $w$. Now, for those $w$ for which
\eqref{eq:almostCS} fails for some $j\leq n$, that is to say, for $w\notin
(D_0\cup\dots\cup D_n)$,
let $\seq {\tilde{\tuple a}}w(n+1)$ be equal to $\seq {\tilde{\tuple a}}w(n)$;
for the remaining $w$, that is to say, for  almost all $w$, we make no
changes: $\seq{\tilde{\tuple a}}w(n+1):=\seq{{\tuple
a}}w(n+1)$. It is now easily seen that \eqref{eq:allCS} holds for all $w$. Since, for every  $n$, almost each $\seq{\tilde{\tuple a}}w(n)$ is equal to $\seq {\tuple a}w(n)$, their ultraproduct   is  $\ul{\tuple a}(n)$,
thus establishing our claim.

So we may assume \eqref{eq:almostCS} holds for all $n$ and $w$.  
Let $f\colon W\to \nat$ be a function on the index set 
$W$ such that for each $n$, almost all $f(w)\geq n$ (this is where we use that the ultrafilter is countably
incomplete; if $W=\nat$, we can of course simply take the identity map). Let $\ul
b$ be the ultraproduct of the $\seq {\tuple a}w(f(w))$. Since $\seq {\tuple
a}w(f(w))\equiv\seq {\tuple a}w(n)\mod\seq\maxim w^n$ for almost all $w$ by \eqref{eq:allCS},
\los\
yields
$\ul b\equiv\ul{\tuple a}(n)\mod\ulmax^n$, for each $n$, showing that $\ul
b$ is a
limit of $\ul{\tuple a}$. Although this limit might not be unique, it will be
in the separated quotient $\ulsep R$, showing that the latter is a complete local
ring, equal therefore to
$\complet {\ul R}$. Noetherianity now follows from Theorem~\ref{T:finembcomp}.
\end{proof}

\begin{corollary}\label{C:clos}
The closure of an ideal $I$ in an ultra-Noetherian ring $\ul R$ is equal to
$I+\inf {\ul R}$. In particular, if $\ulsep R$ is the cataproduct of the $\seq Rw$, and $\ul I$ the ultraproduct of ideals $\seq Iw\sub\seq Rw$, 
then $\ulsep R/\ul I\ulsep R$ is the cataproduct of the $\seq Rw/\seq Iw$.
\end{corollary}
\begin{proof}
Since $\ulsep R:=\ul R/\inf{\ul  R}$ is Noetherian by Lemma~\ref{L:ulcomp}, the
ideal $I\ulsep R$ is closed by Krull's intersection theorem.  All assertions now
follow from   Lemma~\ref{L:quot}.
\end{proof}

\begin{corollary}\label{C:compulsep}
The cataproduct $\ulsep R$ of Noetherian local rings $\seq Rw$ of
bounded embedding dimension is equal to the cataproduct $\ulsep S$ of
their completions.
\end{corollary}
\begin{proof}
Let $(\ul R,\ulmax)$ and $(\ul S,\ul{\mathfrak n})$ be the ultraproduct of
respectively the $\seq Rw$ and the $\seq {\complet R}w$. By \los, $\ulmax
\ul S=\ul{\mathfrak n}$ and $\ul R$ is dense in $\ul S$. Hence both rings have
the same completion, which  by Lemma~\ref{L:ulcomp} is respectively the cataproduct
of the $\seq Rw$ and of the $\seq{\complet R}w$.
\end{proof}

However, this is not the only case in which different rings can have the same 
cataproduct. Let $(R,\maxim)$ be a local ring of
finite embedding dimension. A filtration $\mathfrak I=(I_n)_n$ on $R$ is called
\emph{analytic} if its extension $\mathfrak I\complet R$ induces a Haussdorf
topology  on $\complet R$, or, equivalently, if the intersection of all
$I_n\complet R$ is zero. In particular, the $\maxim$-adic filtration is
analytic by Theorem~\ref{T:finembcomp}. Given two filtrations $\mathfrak
I=(I_n)_n$ and $\mathfrak J=(J_n)_n$,
we say that $\mathfrak I$ is \emph{bounded} by $\mathfrak J$,  if the $\mathfrak I$-adic
topology is stronger than or equal to the $\mathfrak J$-adic topology, that is
to say, for each fixed $N$, we have $I_n\sub J_N$ for all sufficiently big
$n$. 

\begin{lemma}[Chevalley]\label{L:Chev}
A filtration on a Noetherian local ring $(R,\maxim)$ is analytic \iff\ it is
bounded by the $\maxim$-adic filtration.
\end{lemma}
\begin{proof}
If $\mathfrak I=(I_n)_n$ is analytic, then the intersection of all $I_n\complet
R$ is zero. By Chevalley's theorem (see for instance \cite[Exercise
8.7]{Mats})  we have for fixed $N$ an inclusion $I_n\complet
R\sub\maxim^N\complet R$ for
$n$ sufficiently big. By faithful flatness, $I_n\sub \maxim^N$ for $n\gg0$. The
converse is immediate from Krull's intersection theorem (see for instance
\cite[Theorem 8.10]{Mats}).
\end{proof}

\begin{corollary}\label{C:filt}
If  $(\seq In)_n$ is an analytic  filtration on  a   Noetherian  local ring
$R$, then the catapower $\ulsep R$ of $R$ is isomorphic to the cataproduct
$\ulsep S$ of the $R/\seq In$. 
\end{corollary}
\begin{proof}
Without loss of generality, we may assume $R$ is complete. The natural
surjections $R\to R/I_n$ induce a map
$\ulsep R\to\ulsep S$, which is again surjective by \los. Let $\ul x$ be an
element in the ultrapower $\ul R$ of $R$ so
that its image in $\ulsep R$ is in the kernel of $\ulsep R\to\ulsep S$. Choose
$\seq xn\in R$ with ultraproduct equal to $\ul x$ and fix   $N$. Since
$\ul
x\in\inf{\ul S}$, almost each $\seq xn\in\maxim^N(R/\seq In)$.  By Lemma~\ref{L:Chev},
almost each $I_n\sub\maxim^N$ and hence almost each $\seq xn\in\maxim^N$. By \los, $\ul x\in\maxim^N\ul
R$.  Since  
$N$ was arbitrary,   $\ul x$ lies in $\inf{\ul R}$ and hence its image is zero
in $\ulsep R$, showing that $\ulsep R\to\ulsep S$ is also injective. 
\end{proof}

It should be noted that the corresponding ultraproducts $\ul R$ and $\ul S$, however, are far from equal, as, for instance,  $\conspec{\ul S}$ is always a singleton by Proposition~\ref{P:ulscheme}.
Contrary to the Noetherian case, the natural map $R\to\complet R$ does
not need to be flat if $R$ has finite embedding dimension.
We nevertheless expect
some vestige of (faithful) flatness to hold. One example of this is given by
Lemma~\ref{L:quot}, namely $I=I\complet R\cap R$ for any closed
ideal $I$. It is well-known (see for instance
\cite[Theorem 2.2]{SchBetti}) that the latter property already follows from 
the vanishing of $\tor R1{\complet R}k$, where $k$ is the residue field of
$R$. For ultra-Noetherian local rings, where   completion and separated
quotient coincide by Lemma~\ref{L:ulcomp},  this latter property does indeed
hold:

\begin{proposition}\label{P:betti1}
For every ultra-Noetherian local ring $\ul R$ with residue field $\ul k$, we have  $\tor
{\ul R}1{\ulsep R}{\ul k}=0$.
\end{proposition}
\begin{proof}
 From the exact sequence 
$$
\Exactseq {\inf {\ul R}}{\ul R}{\ulsep  R}
$$
we get after tensoring   over $\ul k$ an exact sequence
$$
0\to \tor {\ul R}1{\ulsep  R}{\ul k}\to \inf {\ul R}/\ulmax\inf {\ul
R}\to {\ul k}\to {\ul k}\to 0, 
$$
where $\ulmax$ is the maximal ideal of ${\ul R}$. In particular, the first
Betti number of $\ulsep  R$ vanishes \iff\ $\ulmax\inf {\ul R}=\inf {\ul R}$. To
prove the
latter equality, let $(\seq Rw,\seq\maxim w)$ be Noetherian local rings with ultraproduct   $\ul R$.
Let  $\ul a$ be a non-zero element in $\inf {\ul R}$ and choose non-zero $\seq
aw\in\seq Rw$ so
that their ultraproduct is equal to $\ul a$. Let $\ulmax$ be generated by
$\ul{x_1}, \dots,\ul{x_e}$
and,
for each $i$, choose $\seq{x_i}w\in\seq Rw$ whose ultraproduct equals $\ul
{x_i}$.
By \los, $\seq\maxim w=(\seq{x_1}w,\dots,\seq{x_e}w)\seq Rw$. If $\seq aw$ has
order $\seq nw$, then we can find  $\seq{b_i}w\in\seq Rw$ of order $\seq nw-1$
such that $\seq aw=\seq{x_1}w\seq{b_1}w+\dots+\seq{x_e}w\seq{b_e}w$. Let $\ul 
{b_i}$ be the
ultraproduct of the $\seq{b_i}w$. Fix some $N$. Since $\ul a\in\inf {\ul R}$, its order
is strictly bigger than $N$ and hence so is almost each  $\seq nw$. Therefore,
almost each $\seq{b_i}w$ has order at least $N$ and hence $\ul {b_i}\in\ulmax^N$.
Since this holds for all $N$, we get $\ul {b_i}\in \inf {\ul R}$. Since
$\ul a=\ul {x_1}\ul {b_1}+\dots+\ul {x_e}\ul {b_e}$ by \los, we are done.
\end{proof}

\begin{corollary}
Let $(\ul R,\ulmax)$ be an ultra-Noetherian local ring and $I$ an ideal in
$\ul R$. If $I$ is closed, then so is   $I\ulmax^n$ for every $n$.
\end{corollary}
\begin{proof}
By Corollary~\ref{C:clos}, we have $\inf{\ul R}\sub I$. Since $\inf{\ul
 R}=\ulmax^n\inf{\ul R}$ by the proof
of Proposition~\ref{P:betti1}, we get   $\inf{\ul R}\sub I\ulmax^n$,
showing that $I\ulmax^n$ is closed by another application of
Corollary~\ref{C:clos}.
\end{proof}

We may extend the notion of cataproduct to modules as well: for each $w$, let
$\seq Mw$ be an $\seq Rw$-module, and let $\ul M$ be their ultraproduct. It follows that $\ul M$ is an $\ul R$-module. We
define the \emph{cataproduct} of the $\seq Mw$ as the $\ulsep R$-module $\ulsep M:=\ul M\tensor_{\ul R}\ulsep R=\ul M/\inf{\ul R}\ul M$ given by
base change. If $\seq Nw\sub\seq Mw$ are submodules, then $\ul N\sub \ul
M$. However, the induced \homo\ $\ulsep N\to\ulsep M$ may fail to be injective.
The following result is an exercise on \los\   (see for instance \cite{SchEC}), and the proof is left to the reader.

\begin{proposition}\label{P:ullen}
Let $\ul M$ and $\ulsep M$ be the respective ultraproduct and cataproduct of the $\seq Mw$. Almost each $\seq Mw$ is minimally generated by   $s$ elements (respectively,
has length $s$),
\iff\ $\ul M$ is minimally generated by  $s$ elements (respectively,
has length $s$), \iff\ so does $\ul M$. \qed
\end{proposition}

\subsection*{Flatness of catapowers}

A key result about catapowers, one which will be used frequently in our
characterizations through uniform behavior in \S\ref{s:ua}, is the following
theorem and its corollary:

\begin{theorem}\label{T:flatsep}
Let $R$ be a Noetherian local ring and  $\ulsep R$ its catapower. 
There is a canonical \homo\ $R\to \ulsep R$ which is faithfully flat and
unramified.
\end{theorem}
\begin{proof}
Let $\ul R$ be the ultrapower of $R$ and $R\to \ul R$ the diagonal embedding.
Composed with the canonical surjection $\ul R\to \ulsep R=\ul R/\inf{\ul R}$, we
get the map
$R\to \ulsep R$. By Corollary~\ref{C:compulsep} and the fact that completion is
faithfully flat, we may already assume that $R$ is complete.  Since $\maxim\ulsep R$ is the maximal
ideal of $\ulsep R$, the map   $R\to \ulsep R$ is unramified.
So remains to show that this map  is flat. Let us first
prove this under the additional assumption that $R$ is regular. We
induct on its dimension. Let $x$ be a regular parameter of $R$, that is to
say, an element of order one.  I claim that $x$ is $\ulsep R$-regular. This
  follows for instance from the results in \S\ref{s:ulsing} (proving among other
things that $\ulsep R$ is then regular), but we can give a direct argument here.
Indeed,
suppose $\ul s\in \ul R$ is such that $x\ul s\in\inf {\ul R}$. If $\seq sw\in R$
have ultraproduct   equal to $\ul s$, then for a
fixed $N$, almost each $x\seq sw\in\maxim^N$. Since $R$ is
regular and $x$ has
order one, $\seq sw\in \maxim^{N-1}$ and hence by \los, $\ul s\in\maxim^{N-1}\ul
R$. Since this holds for all $N$, we get $\ul s\in\inf {\ul R}$, showing that $x$ is $\ulsep
R$-regular. It is not hard to see that $\ulsep R/x\ulsep R$ is the catapower of the regular local ring $R/xR$, so that by induction, $R/xR\to
\ulsep R/x\ulsep R$ is faithfully
flat. Since any $R/xR$-regular sequence is then $\ulsep R/x\ulsep R$-regular,
$\ulsep R$ is a balanced big \CM\ algebra over $R$. Since $R$ is regular, $R\to
\ulsep R$ is therefore faithfully flat (see for instance  \cite[Theorem
IV.1]{SchFPD} or \cite[Lemma 2.1(d)]{HHbigCM2}). 

For the general case, we may write $R$ as a homomorphic image
$S/I$ of a complete regular local ring $S$ by Cohen's theorem. By what
we just proved, $S\to \ulsep S$ is faithfully flat, where $\ulsep S$ is the
catapower of $S$. Hence the base change $R=S/I\to \ulsep
R=\ulsep S/I\ulsep S$ is also flat.
\end{proof}

\begin{corollary}\label{C:ulsepreg}
Let $R$ be  an    excellent local ring (e.g., a complete Noetherian local
ring) with catapower $\ulsep R$. The
natural map $R\to\ulsep R$  is regular. In particular, $R$ is regular    (respectively, normal, reduced, \CM\
or Gorenstein), \iff, so is $\ulsep R$.
\end{corollary}
\begin{proof}
The second assertion is a well-known consequence of the first (see for instance
\cite[Theorem 32.2]{Mats}). 
As for the first, let us first show this in the special case that $R=k$ is a
field. Note that in this case, the catapower  is equal to the
ultrapower $\ul k$ of $k$.  Hence, we need to show that $k\to\ul k$ is
separable, and so we may assume that $k$ has positive \ch\ $p$. We will
establish separatedness by verifying  MacLane's criterion (see for instance \cite[Theorem 26.4]{Mats}). Let
$b_1,\dots,b_n$ be elements
 in $k^{1/p}$ which are linearly independent over $k$. Suppose
$\ul {x_1}b_1+\dots+\ul {x_n}b_n=0$ for some
$\ul {x_i}\in\ul k$. Choose 
$\seq{x_i}w\in k$ with ultraproduct equal to $\ul {x_i}\in\ul k$. Taking $p$-th
powers, using \los\ and then taking $p$-th roots, we get
$\seq{x_1}wb_1+\dots+\seq{x_n}wb_n=0$ for almost all $w$.
Since the $b_i$ are linearly independent over $k$, almost all $\seq{x_i}w$
are zero. By \los, each $\ul {x_i}$ is zero, showing that the $b_i$, viewed as
elements in $\ul k^{1/p}$, remain linearly independent over $\ul k$, as
we wanted to show.

For $R$ arbitrary,  Theorem~\ref{T:flatsep}  yields that $R\to \ulsep R$ is faithfully flat and unramified.
By what we just proved, the induced residue field extension  is separable. Therefore
$R\to\ulsep R$ is formally smooth by \cite[Theorem 28.10]{Mats}. Regularity 
then follows from a result by Andr\'e in \cite{And} (see also \cite[p.
260]{Mats}).
\end{proof}


\begin{proposition}\label{P:int}
Let  $R\sub S$ be an injective, local \homo\  between Noetherian local rings
and let $\ulsep R\to\ulsep S$ be the induced map of catapowers.
\begin{enumerate}
\item If $R\sub S$ is finite, then   $\ulsep R\to\ulsep S$ is   finite and
injective. 
\item If $R\sub S$ is \c-injective, that is to say, if
$\complet R\to \complet S$ is injective, then  $\ulsep R\to\ulsep S$ is
injective too.
\end{enumerate}
\end{proposition}
\begin{proof}
Let $\maxim$ and $\mathfrak n$ be the maximal ideals of respectively $R$ and
$S$. Assume $R\sub S$ is finite, so that $\mathfrak n^a\sub\maxim S$ for some
$a$. By the Artin-Rees Lemma, $\maxim^nS\cap R\sub\maxim^{n-c}$ for some $c$ and
all $n\geq c$. Hence $\mathfrak n^{na}\cap R\sub\maxim^{n-c}$ for all $n\geq c$
and hence by \los, the same inclusions hold in the   extension $\ul
R\sub\ul S$ of ultrapowers.
Using this, it is not hard to show that $\inf{\ul S}\cap \ul R=\inf{\ul R}$,
showing that $\ulsep R\sub \ulsep S$ is injective (and clearly also finite).

If $R\sub S$ is \c-injective, then the filtration $\mathfrak n^k\cap R$, for $k=0,1,\dots$, is
easily seen to be analytic, whence bounded by the $\maxim$-adic filtration by
Lemma~\ref{L:Chev}. Again one derives from this that $\inf{\ul S}\cap \ul R=\inf{\ul
R}$, whence that $\ulsep R\sub \ulsep S$ is injective.
\end{proof}

\subsection{Extended dimensions in ultra-Noetherian local
rings}\label{s:uldim}

We extend the nomenclature introduced in the beginning of this section to include
invariants. In particular, we define
the  \emph{\c-dimension} of $R$, denoted $\cad R$, as the (Krull) dimension of
its completion $\complet
R$.  For an  ultra-Noetherian local ring $\ul R$ given as the ultraproduct of
Noetherian local rings $\seq Rw$ of   embedding dimension at most $m$, we define 
its \emph{ultra-dimension}, denoted
$\ud{\ul R}$,  as the dimension of almost all $\seq Rw$. Since almost all 
$\seq Rw$   have  dimension  at most $m$, the ultra-dimension of $\ul R$ is finite.

\begin{theorem}\label{T:uldim}
For an ultra-Noetherian  local ring $\ul R$, we have  inequalities
\begin{multline}\label{eq:uldimineq}
\op{depth}(\ul R)\leq \pid {\ul R}=\frd{\ul R}=\ud{\ul R} \\
\leq\cld {\ul R}=\pd
{\ul R}=\cad{\ul R}\leq \ed {\ul R}. 
\end{multline}
\end{theorem}
\begin{proof}
By Theorems~\ref{T:pdim} and \ref{T:dimineq}, the 
\c-dimension of $\ul R$ is equal to its \pdim\
and to its cl-dimension. On the other hand,  \los\ and 
Lemma~\ref{L:fodim} yield that the ultra-dimension of $\ul R$ coincides with
its  
pi-dimension and its fr-dimension.   Depth is
 also  first-order, as it is cast in terms of the vanishing of the
Koszul homology of a generating set of $\maxim$ (see \S\ref{s:depth} below
for more details). Since in a Noetherian local ring depth never exceeds dimension,
the first inequality is then also clear.
\end{proof}

There are no further constraints on the above invariants of an ultra-Noetherian 
ring, as the following examples show (in the discussion of these examples,
we will also  use some terminology from later sections).\footnote{One should note
that for Noetherian rings, other than
the obvious restriction that   pi-dimension and dimension agree, we also have
the remarkable fact that when   dimension and embedding
dimension agree, that is to say, when the ring is regular, then this common
value must also be equal to its depth.}

\begin{example}\label{E:uldimineq}
 Let $e\leq h \leq d\leq m$. We will construct an ultra-Noetherian  local
ring  $\ul R$ with depth $e$, ultra-dimension $h$, \c-dimension $d$, and
embedding dimension $m$. First we introduce some notation. Let $\ul R$ be the
ultraproduct of the $\seq Rw$ and let
${\ul n}$ be a non-standard positive
 integer, that is to say, an ultraproduct of an unbounded sequence of positive
integers $\seq nw$. For an element $\ul a\in \ul R$, realized as an
ultraproduct of
elements $\seq aw\in\seq Rw$,  we write $\ul a^{\ul n}$ to denote the ultraproduct of the
elements $\seq aw^{\seq nw}$; one verifies that this is independent of the
choice of $\seq aw$ or $\seq nw$. Let
$\ul S$ be the ultrapower of $S:=\pow k\xi$, for some indeterminates  $\xi:=\rij \xi m$ and   some
field $k$,  let 
$$
I:=(\xi_{e+1}^{\ul n} \xi_m, \dots, \xi_h^{\ul n} \xi_m, \xi_{h+1}^{\ul n},\dots,\xi_d^{\ul n},\xi_{d+1}^2,\dots,\xi_m^2)\ul S
$$
and put $\ul R:=\ul S/I$. By \los, 
$\rij \xi e$ is
$\ul R$-regular and since the maximal ideal of $\ul R/\rij \xi e\ul R$ is
annihilated by the 
element $\xi_{e+1}^{{\ul n}-1}\cdots \xi_d^{{\ul n}-1} \cdot \xi_{d+1}\cdots
\xi_m$, we see that $\ul R$ has depth $e$. Since $\xi_{d+1},\dots,\xi_m$ are nilpotent,
we get from Proposition~\ref{P:nil} below that the ultra-dimension of $\ul R$ is the
same as the ultra-dimension of 
$$
\ul R/(\xi_{d+1},\dots,\xi_m)\ul R=\ul S/(\xi_{h+1}^{\ul n},\dots,\xi_d^{\ul n},\xi_{d+1},\dots,\xi_m)\ul S,
$$
that is to say,   equal to $h$. 
On the other hand, $I\ulsep
S=(\xi_{d+1}^2,\dots,\xi_m^2)\ulsep S$, where $\ulsep S$ is the catapower of $S$ (note that  $\ulsep S\iso\pow{\ul k}\xi$, where $\ul k$ is the ultrapower
of $k$; see for instance \cite[Proposition 3.1]{SchAsc}). Hence the catapower $\ulsep R$ of $R$ has dimension
$d$. By Lemma~\ref{L:ulcomp}, the \c-dimension of $\ul R$ is therefore $d$. 
Finally, it follows from \los\ that $\ul R$ has
embedding dimension $m$. Note that since $\ulsep R$ is \CM, $\ul R$ is 
\cCM.

More generally, let $q$ be any number between $e$ and $d$ and let $\ul R':=\ul S/I'$,
where $I'$ is the sum of the ideal $I$ above and the ideal
$(\xi_{q+1}\xi_m,\dots,\xi_d\xi_m)\ul S$. Then $\ul R'$ has still the same depth,
ultra-dimension, \c-dimension and embedding dimension as $\ul R$, but now the depth
of $\ulsep  R$, that is to say, the \emph{\c-depth} of $\ul R$, is $q$, since $\rij
\xi q$ is a regular sequence.
\end{example}

\begin{example}
The previous example might one lead to think that the depth of $R$ is always
at most its \c-depth. However, this is not the case
as the
following example shows. Let $\ul S$ be as in the previous example with $m=3$, and
let $\ul R:=\ul S/(\xi_1^2,\xi_1\xi_2,\xi_1\xi_3-\xi_2^{\ul n})\ul S$, with ${\ul n}$ a non-standard positive
integer. Since $\xi_3$ is $\ul R$-regular and since $\ul R$ has ultra-dimension one, the
depth of $\ul R$
is one by Theorem~\ref{T:uldim}. On the other hand, $\ulsep R$
is equal to $\ulsep S/(\xi_1^2,\xi_1\xi_2,\xi_1\xi_3)\ulsep S$, whence has depth zero.
Note that $\ulsep R$ has dimension two, so that $\ul R$ itself has \c-dimension two.
Hence $\ul R$ is \emph{ultra-\CM}, but not \cCM.
\end{example}

\subsection*{Isodimensionality} 
We call a local ring $R$ of finite embedding dimension \emph{isodimensional}
if \eqref{i:pidpd} is an equality, that is to say, 
if  the \pdim\ of $R$ is equal to its pi-dimension. In view of
Theorem~\ref{T:uldim}, an ultra-Noetherian local ring is isodimensional \iff\ its
ultra-dimension is equal to its \c-dimension.

\begin{proposition}\label{P:nil}
Let $\ul R$ be an ultra-Noetherian local ring. If $\id$ is a \fr\ ideal
contained in
$\op{nil}(\ul R)$, then $\ul R$ and $\ul R/\id$ have the same ultra-dimension. In particular,
$\ul R$ is isodimensional \iff\ $\ul R/\id$ is.
\end{proposition}
\begin{proof}
Let $h$ be the ultra-dimension of $\ul R$, so that  $\ul R$ is the ultraproduct of
$h$-dimensional Noetherian local rings $\seq Rw$ of bounded embedding dimension. 
  Since $\id$ is \fr, it     can be
realized as the ultraproduct of \fr\ ideals $\seq \id w$ by the argument in
the proof of Proposition~\ref{P:ulscheme}.
By
\los, almost each
$\seq \id w$ is nilpotent, and therefore $\seq Rw/\seq \id w$ has again dimension
$h$. Hence $\ul R/\id$ has ultra-dimension $h$ as well. 

The final assertion follows from the fact that $\ul R$ and $\ul R/\id$ have the same
\pdim\ (this is true in general, since $\id$ is contained in every threshold prime  of $\ul R$).
\end{proof}

For ultra-Noetherian local rings,  we have the following important criterion
for isodimensionality:

\begin{theorem}\label{T:isodim} 
Let $\ul R$ and $\ulsep R$ be the respective ultraproduct and cataproduct of
Noetherian local rings $\seq Rw$ of bounded embedding dimension. The following
are equivalent:
\begin{enumerate}
\item\label{i:isodim} $\ul R$ is isodimensional;
\item\label{i:dimw} almost all $\seq Rw$ have
dimension equal to $\pd {\ul R}$;
\item\label{i:dimsep} almost all $\seq Rw$ have the same 
dimension as $\ulsep R$;
\item\label{i:pardeg} almost all $\seq Rw$ have the
same parameter degree  (which is then also the parameter degree of $\ul R$ and of
$\ulsep R$).
\end{enumerate} 
\end{theorem}
\begin{proof} 
The equivalence of \eqref{i:dimw} and \eqref{i:dimsep} follows from
Lemma~\ref{L:ulcomp} and Theorem~\ref{T:pdim}. Let
$d\leq m$ be the respective
\pdim\ and embedding dimension of
$\ul R$. By Theorem~\ref{T:uldim}, the \c-dimension of $\ul R$ is $d$. Since
$\kd{\seq Rw}\leq m$, almost all $\seq
Rw$ have a common dimension $h\leq m$, which is then  the ultra-dimension of
$\ul R$ by definition, from which we get the equivalence of \eqref{i:isodim} and
\eqref{i:dimw}. 

So remains to show that equivalence of
\eqref{i:dimw} and \eqref{i:pardeg}. Suppose 
  $\pardeg{\seq Rw}=e$ for almost all $w$. In each $\seq Rw$, choose an 
$h$-tuple $\seq{\tuple x}w$ so that almost all
$\seq Rw/\seq{\tuple x}w\seq Rw$ have length $e$. Let $\tuple {\ul x}$ be 
the ultraproduct of the $\seq{\tuple x}w$. By Proposition~\ref{P:ullen}, the length 
of $\ul R/\tuple {\ul x}\ul R$, being the ultraproduct of the $\seq Rw/\seq{\tuple x}w\seq Rw$, is also $e$. It
follows that $\ul R$ has \pdim\ at most $h$. We 
already argued that its \pdim\ is at least $h$, so
that we get  $h=d$. In particular,  the parameter degree of $R$ is at most
$e$, and by reversing this argument, one can also show 
that it cannot be less than
$e$, whence must be equal to $e$.

Conversely, assume $h=d$. Let $\tuple {\ul x}$ be a generic sequence in 
$\ul R$ and choose $d$-tuples $\seq{\tuple x}w$ whose
ultraproduct is $\tuple {\ul x}$. By \los, almost each $\seq {\tuple x}w$
generates an $\seq\maxim w$-primary ideal, and therefore must
be a system of
 parameters in $\seq Rw$, since almost each $\seq Rw$ has dimension $h=d$. Let
$l$ be the length of $\ul R/\tuple {\ul x}\ul R$.
 By   Proposition~\ref{P:ullen}, almost each $\seq
Rw/\seq{\tuple x}w\seq Rw$ has length $l$, showing that $\pardeg{\seq
Rw}\leq l$,  for almost $w$. 
\end{proof}

\begin{example}
We cannot replace parameter degree with multiplicity in the previous result
as the following example shows. Fix some $e>0$ and put $\seq
Rw:=S/(\xi^w,\xi^e\zeta^{w-e})S$ for each $w\geq e$,  
where $S:=\pow k{\xi,\zeta}$ and $k$ is a field. Let $\ul R$ be the ultraproduct
of the 
$\seq Rw$, let $\ul k$ be the ultrapower of $k$ and let $\ulsep S\iso\pow{\ul
k}{\xi,\zeta}$ be  the
 catapower
of $S$. Since the
ultraproduct
of the $\xi^w$ and the $\xi^e\zeta^{w-e}$ are infinitesimals, the cataproduct
of the $\seq Rw$ is $\ulsep R=\ulsep S$,  showing that $\ul R$ is not
isodimensional (since the $\seq Rw$ are one-dimensional and $\ulsep R$ is
two-dimensional). Therefore, by the theorem, the parameter degree of the $\seq
Rw$ is unbounded (in fact, equal to $w$).  On the other hand, $\zeta$ is a
parameter in each $\seq
Rw$ so that we can calculate the multiplicity of $\seq Rw$ by Lech's lemma
(\cite[Theorem 14.12]{Mats}) as the limit of $e_{wn}/n$ as $n$ tends to
infinity, where $e_{wn}$ is the 
length of $\seq Rw/\zeta ^n\seq
Rw$. One calculates that $e_{wn}=w(w-1)+e(n-w+2)$ and hence $\mult{\seq
Rw}=e$. This shows, in view of Remark~\ref{R:mult}, that multiplicity is in
general not first-order.
\end{example}

\begin{remark}\label{R:isodim}
In view of Theorem~\ref{T:isodim}, we will often require that a collection of
Noetherian local rings $\seq Rw$ have (almost all) the same embedding dimension
and the same  parameter degree, to ensure that their cataproduct is again
Noetherian of the same dimension. In fact, we can replace this requirement
with the more natural requirement that (almost all) $\seq Rw$ have the same
  dimension and parameter degree. Indeed, if a Noetherian local ring $R$ has
dimension $d$ and parameter degree $e$, then its embedding dimension is at
most $d+e-1$.

Note that by Lemma~\ref{L:pardegmult} below, if almost all $\seq Rw$ are \CM\ we
may further simplify this to the requirement that almost all $\seq Rw$ have
the same dimension and multiplicity. The previous example shows that this is no
longer true without the \CM\ assumption.
\end{remark}

\begin{corollary}\label{C:deggen}
If $\ul R$ is an isodimensional ultra-Noetherian local ring and $\ul x$ the
ultraproduct of elements $\seq xw$, then $\ul x$ is generic \iff\
$\op{deg}(\seq xw)$ is bounded.
\end{corollary}
\begin{proof}
Let $\seq Rw$ be Noetherian local rings with ultraproduct   $\ul R$. By Theorem~\ref{T:isodim}, almost
each $\seq Rw$ has dimension $d:=\pd{\ul R}$. Suppose $\ul x$ is generic.
Hence, $\ul R/\ul x\ul R$ has \pdim\ $d-1$, whence ultra-dimension at most $d-1$. In
particular, almost each $\seq Rw/\seq xw\seq Rw$ must have dimension $d-1$.
Hence $\seq xw$ is generic in $\seq Rw$ and $\ul R/\ul x\ul R$ is again
isodimensional. By Theorem~\ref{T:isodim}, this means that the $\seq Rw/\seq
xw\seq Rw$ must have bounded parameter degree, proving the direct implication.

Conversely, suppose the $\op{deg}(\seq xw)$ are bounded, that is to say,
almost all $\seq xw$ are generic and the parameter degrees of the $\seq
Rw/\seq xw\seq Rw$ are bounded.
By Theorem~\ref{T:isodim} once more, $\ul R/\ul x\ul R$ has \pdim\ $d-1$,
showing that $\ul x$ is generic. 
\end{proof}

Without the isodimensional assumption, the result is false: for instance if
$\ul R$ has ultra-dimension zero (e.g., the ultraproduct of the $R/\maxim^n$),
then no element in $\ul R$ is realized as an ultraproduct of elements of finite degree.

Conform with our previous nomenclature, we call a local ring \emph{ultra-excellent}, if it is the ultraproduct of excellent local rings of bounded embedding dimension. We can now give the following improvement of Corollary~\ref{C:loculpow}.

\begin{corollary}\label{C:locup}
Let $\ul R$ be an ultra-Noetherian local ring, realized as the ultraproduct of equi\ch\ excellent local rings $\seq Rw$. If $\ul R$ is   isodimensional, then any localization at a \fr\ prime ideal has finite embedding dimension, and any \fr\ prime ideal is strong.
\end{corollary}
\begin{proof}
Let $\pr\in\conspec{\ul R}$. By Proposition~\ref{P:ulscheme}, there exist prime ideals $\seq\pr w\sub\seq Rw$ with ultraproduct equal to $\pr$. By Theorem~\ref{T:isodim}, there is some $\rho$, such that almost each $\seq Rw$ has   parameter degree $\rho$. Hence, by Proposition~\ref{P:unifembdim}, almost each $(\seq Rw)_{\seq\pr w}$ has embedding dimension at most $d+\rho$, where $d$ is the common dimension of almost all $\seq Rw$ (that is to say, the ultra-dimension, whence \pdim, of $\ul R$). Since $(\ul R)_\pr$ is the ultraproduct of the $(\seq Rw)_{\seq\pr w}$, its embedding dimension is at most $d+\rho$.   Proposition~\ref{P:sfr} then implies that $\pr$ is strong.
\end{proof}

We actually showed that each stalk of $\spec{\ul R}$ at a point belonging to $\conspec{\ul R}$ has embedding dimension  at most $d+\rho$, where $d$ is the \pdim\ of $\ul R$ and $\rho$ its   parameter degree.   Inspecting the proof of Proposition~\ref{P:unifembdim}, we see that almost each $(\seq Rw)_{\seq\pr w}$ has parameter degree at most $\rho$, showing that each stalk is also isodimensional, of ultra-dimension, whence \pdim, at most $d$.

\section{\C-singularities}\label{s:cata}

According to the definitions in \S\ref{s:ul}, a local ring  of finite
embedding dimension is \emph{\creg} if its completion is a regular
(Noetherian) local ring.

\begin{theorem}\label{T:preg}
 Let $(R,\maxim)$ be a local ring of 
\pdim\ $d$ and let $k$ be its residue field. The following are equivalent:
\begin{enumerate}
\item\label{i:preg} $R$ is \creg;
\item\label{i:sep} $\sep R$ is \creg;
\item\label{i:emb} $\pd R=\ed R$;
\item\label{i:maxgen} $\maxim$  is generated by a generic
sequence;
\item\label{i:qreg} $\maxim$ is generated by a quasi-regular sequence;
\item\label{i:gr} $\gr{}R$ is isomorphic to $\pol k\xi$, with $\xi$ a 
$d$-tuple of indeterminates.
\end{enumerate}
\end{theorem}
\begin{proof}
 The equivalence of \eqref{i:preg} and  \eqref{i:sep} is clear
since $\sep R$ has the same completion as $R$, and their equivalence with
\eqref{i:gr}
 follows from \cite[Theorem 14.4]{Mats},  since we
have an isomorphism of graded rings $\gr{}R\iso \gr{}{\complet R}$. The
equivalence of \eqref{i:emb} and \eqref{i:maxgen} is clear from the definition
of \pdim. Suppose \eqref{i:maxgen} holds, so that $\maxim$ is generated by a
generic sequence $\rij xd$.
 There is a natural surjective \homo\ $\pol k\xi\to
\gr{}R$ which maps $\xi_i$ to $\op{in}(x_i)$, where $\xi=\rij \xi d$. Since both 
rings have the same dimension by
Theorem~\ref{T:pdim}, the kernel must be zero, proving \eqref{i:gr}. 
Conversely, assume $\gr{} R\iso \pol k\xi$. Hence
$\maxim/\maxim^2$ is generated by $d$ elements, and therefore, by 
Nakayama's Lemma $\maxim$ is generated by $d$
elements, showing that \eqref{i:maxgen} holds.

Remains to show the equivalence of the other conditions with 
\eqref{i:qreg}. Recall that $\tuple x$ is quasi-regular if
$F(\tuple x)=0$, for a homogeneous polynomial $F\in\pol R\xi$, implies 
that $F$ has all its coefficients in $I:=\tuple
xR$. This is equivalent with the natural epimorphism 
$\pol{(R/I)}{\xi_1,\dots,\xi_d}\to \gr IR$ being injective, whence an
isomorphism (see for instance \cite[\S16]{Mats}). Hence taking 
$I=\maxim$, we see that \eqref{i:qreg} is equivalent with
\eqref{i:gr}.
\end{proof}

\begin{remark}\label{R:preg}
 In the above proof, we actually showed 
that if  $R$ is \creg\ of \pdim\ $d$, then any
$d$-tuple generating $\maxim$ is quasi-regular. We will shortly show 
(Theorem~\ref{T:pCM} below) that then  every generic sequence is quasi-regular.
The ring $R$ in the next example shows that a generic sequence
generating the maximal ideal in a \creg\ local ring is not necessarily a regular
sequence.
\end{remark}

\begin{example}\label{e:preg}
 A local ring of \pdim\ zero is \creg\ 
\iff\ it is a field. A local ring of \pdim\ one is
\creg\ \iff\ its maximal ideal is generated by a non-nilpotent 
element. For instance, let $\ul V$ be an ultraproduct of
\DVR{s} (an \emph{ultra-DVR} for short), or more generally, a valuation ring of
finite embedding
 dimension (which is then automatically one). If
$x$ is an element in the ideal of infinitesimals $\inf {\ul V}$ of 
$\ul V$, then $R:=\ul V/x \ul V$ is \creg\ of \pdim\ one. If
$x\neq 0$, then $R$ is not a domain. In fact, $R$ has then depth 
zero (and so is not \preg\ in the sense of
\S\ref{D:sCM} below).
\end{example}

The following fact, however, is noteworthy: if $R$ is moreover 
separated, then any quasi-regular element is regular; see
for instance \cite[Theorem 16.3]{Mats}. In fact, we have the following result:

\begin{corollary}  If a  \creg\ local ring is separated, then
it is a
 domain. More generally, the separated quotient of a
\creg\ local ring is a domain.
\end{corollary}
\begin{proof} 
Immediate from the fact that $\sep R$ embeds in $\complet R$  and
the fact that Noetherian regular local rings are always domains.
\end{proof}

\begin{corollary}
If $R$ is   \creg, then so is any homomorphic image $R/I$, for
$I\sub\inf R$.
\end{corollary}
\begin{proof}
Since $R$ and $R/I$ have the same separated quotient,   the result follows from Theorem~\ref{T:preg}.
\end{proof}

\begin{corollary} 
For each $d$, the class of \creg\ local rings of 
\pdim\ $d$ is first-order definable.
\end{corollary}
\begin{proof}
 Observe that a ring is local \iff\ any sum of two 
non-units is again a non-unit. In fact, an element lies
in the maximal ideal of a local ring \iff\ it is not a unit. 
Therefore, the maximal ideal of a local ring is definable,
as is expressing that some  element lies in the maximal ideal. In 
particular,  the formula $\lambda_{d,n}(\tuple x,\tuple a)$ is first
order, where $\lambda_{d,n}(\tuple x,\tuple a)$ is the formula in the variables 
$\tuple x:=\rij xd$ and $\tuple a:=(a_\nu)_\nu$, for  $\nu$ running
over all $d$-tuples in $\nat^d$ whose sum $\norm\nu$ is $n$, expressing that
\begin{equation}
\label{eq:dn}
\begin{split} &\textsl{if $\tuple x$ generates the maximal  ideal and if 
$\sum_{\norm\nu=n} a_\nu  \tuple x^\nu=0$,}\\
&\textsl{then some $a_\nu$ lies in the maximal ideal.}
\end{split}
\end{equation}
Let $T_d$ be the theory  consisting of all sentences $(\forall 
\tuple x,\forall \tuple a)\lambda_{d,n}(\tuple x,  \tuple a)$, for $n=1,2,\dots$,
together with the sentence $\sigma_d$ expressing that the maximal 
ideal is generated by some $d$-tuple. I claim that
$T_d$ axiomatizes the class of \creg\ local rings of \pdim\ $d$. Indeed, suppose 
that $(R,\maxim)$   satisfies $T_d$. By $\sigma_d$,
there is a $d$-tuple $\tuple x$ such that $\maxim=\tuple xR$. Since 
$\lambda_{d,n}(\tuple x,\tuple a)$ holds for all
tuples $\tuple a$ in $R$, we see that $\tuple x$ is quasi-regular. 
Hence $R$ is \creg\ by Theorem~\ref{T:preg}.
Conversely, if $R$ is \creg\ of \pdim\ $d$, then it satisfies $T_d$ 
by Remark~\ref{R:preg}.
\end{proof}

This immediately gives a large class of \creg\ local rings. Namely, 
any ultraproduct of   regular local rings
of dimension $d$ is \creg, of \pdim\ $d$. We will address this 
situation further in \S\ref{s:ulsing} below.

\begin{corollary} 
A  local ring $R$ of \pdim\ one   is \creg\ \iff\ 
$\sep R$ is a \DVR.
\end{corollary}
\begin{proof} 
Assume $R$ is \creg\, so that  $\complet R$ is  a \DVR\ with
valuation $\ord{\complet R}\cdot$.  
Since $\ord
{\sep R}a=\ord{\complet R}a$ for all $a\in\sep R$, also  $\ord{\sep R}\cdot$ is a
valuation, showing 
that $\sep R$ is a \DVR. Conversely, if $\sep R$ is a \DVR,
then $R$ is \creg\ by Theorem~\ref{T:preg}.
\end{proof}

\subsection*{Cata-\CM\ local rings}
We now turn to the study of \cCM\ local rings of finite embedding dimension, that
is to say, local rings whose completion is \CM. Clearly, any \creg\ local ring is
\cCM.

\begin{theorem}\label{T:pCM}
A local ring   of finite embedding dimension is \cCM\ \iff\ its separated
quotient is \cCM\ \iff\ some (equivalently, every) generic sequence is
quasi-regular.
\end{theorem}
\begin{proof} 
Let $(R,\maxim)$ be a local ring of \pdim\ $d$ and let $\tuple
x$ be a generic sequence. Since $\gr{\tuple xR}R\iso \gr{\tuple x\complet
R}{\complet R}$, the sequence $\tuple x$ is   quasi-regular in $R$ \iff\ it
is so in $\complet R$. Since $R$ and $\sep R$ have the same completion, we 
only need to show the equivalence of the first and
last condition. Suppose that $\tuple x$ is generic.  
Since $\tuple x$ is  $\complet R$-quasi-regular, it is   $\complet R$-regular
 by \cite[Theorem 16.3]{Mats} and the fact that $\complet R$ is Noetherian. Since
 $\complet R$ has dimension $d$ by Theorem~\ref{T:pdim}, it is
\CM, showing that $R$ is \cCM.

Conversely, suppose $\complet R$ is \CM. Since $\tuple x$ is a system
of parameters in $\complet R$, it is $\complet R$-regular, whence 
$\complet R$-quasi-regular. By our previous observation, $\tuple x$ is then
quasi-regular in $R$.
\end{proof}

\begin{corollary}\label{C:cregmult}
A local ring of finite embedding dimension is \creg\ \iff\ it is \cCM\ and has
multiplicity one.
\end{corollary}
\begin{proof}
If a local ring $R$ is \creg, its completion $\complet R$
is regular, whence has multiplicity one. Since $R$ and its completion $\complet
R$ 
have the same multiplicity by
Remark~\ref{R:mult}, the direct implication is clear. Conversely, if $R$ is
\cCM\  and $\mult R=1$, then $\complet R$ is \CM\ with $\mult{\complet R}=1$ by
Remark~\ref{R:mult}. Since $\complet R$ is unmixed, it is regular by
\cite[Theorem 40.6]{Nag}, showing that $R$ is \creg.
\end{proof}

\begin{lemma}\label{L:pardegmult} 
The multiplicity of $R$ is at most 
its parameter degree. If $R$ has infinite residue field then we have equality
\iff\ $R$ is \cCM.
\end{lemma}
\begin{proof} 
Let $\tuple x$ be a generic sequence of $R$. By 
Proposition~\ref{P:sop}, it is a system of parameters in
$\complet R$ and $R/\tuple xR\iso \complet R/\tuple x\complet R$ by 
Lemma~\ref{L:quot}. The common length of the
latter two quotients is at least the multiplicity of the ideal 
$\tuple x\complet R$ by \cite[Theorem 14.10]{Mats} which
in turn is at most $\mult{\complet R}$ by \cite[Formula 
14.4]{Mats}. The desired inequality now follows
from this, since $R$ and $\complet R$ have the same multiplicity by 
Remark~\ref{R:mult}.

The last assertion holds if $R$ is Noetherian by \cite[Lemma 3.3]{SchABCM}.
The general case follows from this since $R$ and $\complet R$ have the same
multiplicity and the same parameter degree. 
\end{proof}

\begin{theorem}\label{T:pgor}
A local ring   of finite embedding dimension is \cgor\ (respectively, a \cci)
\iff\
so is its separated quotient,  \iff\ it admits a quasi-regular,
generic sequence $\tuple x$ such that $R/\tuple xR$ is Gorenstein
(respectively, a complete intersection).
\end{theorem}
\begin{proof}
  Let $(R,\maxim)$ be a local ring of \pdim\ $d$. Since $R$ and $\sep R$ have
the same completion, we 
only need to show the equivalence of the first and last condition. Suppose
$\tuple x$ is a quasi-regular, generic sequence. In particular, $R$ is \cCM\ by
Theorem~\ref{T:pCM}, whence $\complet
R$ is \CM\ and $\tuple x$ is $\complet R$-regular. Moreover, $R/\tuple xR\iso
\complet R/\tuple x\complet R$
 by Lemma~\ref{L:quot}.
Therefore the former is Gorenstein (respectively, a complete intersection) \iff\
the latter is, \iff\   $\complet R$ is (see
\cite[Theorem 2.3.4 and Proposition 3.1.19]{BH}).
\end{proof}

\begin{proposition}
A local ring of finite embedding dimension is \cgor\ \iff\ there exists a
quasi-regular, generic sequence generating an irreducible ideal. When this is
the case, every generic sequence is quasi-regular and generates an irreducible
ideal.
\end{proposition}
\begin{proof}
Let $\tuple x$ be a quasi-regular, generic sequence. The result is now
immediate from the fact that $\tuple xR$ is irreducible \iff\  $R/\tuple xR$
is Gorenstein.
\end{proof}

\section{Pseudo-singularities}\label{s:sCM}

The \c-singularities from the previous section do not always correspond to
their `ultra' versions (which will be treated in the next  section). To this end
we will define some stronger versions
of these \c-singularities, defined intrinsically, that is to say, without
reference to the completion. Throughout this section, $(R,\maxim)$ is a local
ring of finite embedding dimension.

\subsection{Grade and depth.}\label{s:depth}
 Let $A$ be an arbitrary ring and
$I$ a
 finitely generated ideal in $A$. Choose a tuple of
generators $\tuple x=\rij xn$ of $I$. The \emph{grade} of $I$, 
denoted $\op{grade}(I)$, is by definition equal to
$n-h$, where $h$ is the largest value $i$ for which the $i$-th Koszul 
homology $H_i(\tuple x;A)$ is non-zero. One shows
that the grade of $I$ does not depend on the choice of generators 
$\tuple x$. For   a local ring $R$ of finite
embedding dimension, we define its \emph{depth} as the grade of its 
maximal ideal; it is non-zero   \iff\
its maximal ideal is not an associated prime.   

Grade, and hence 
depth, \emph{deforms well}, in the sense that the
\begin{equation}
\label{eq:defdep}
\op{grade}(I(A/\tuple xA))= \op{grade}(I) -\norm{\tuple x}
\end{equation}
for every $A$-regular sequence $\tuple x$ contained in $I$. If $R$ has 
\pdim\ $d$, then its depth is at most $d$. Indeed, by
definition, the grade of a finitely generated ideal never exceeds its 
minimal number of generators, and by
\cite[Proposition 9.1.3]{BH}, the depth of $R$ is equal to the grade 
of any of its $\maxim$-primary ideals. 

The relationship between depth and the
length of a regular sequence (sometimes called the \emph{naive depth} of $R$) is
less straightforward in the non-Noetherian
case and requires an additional definition. For a local ring $(R,\maxim)$ and a finite tuple
of  indeterminates $\xi:=\rij \xi n$, we will denote the localization of 
$\pol R\xi$ at the ideal $\maxim\pol R\xi$ by $R(\xi)$ (this is sometimes called the
$n$-fold \emph{Nagata extension} of $R$).
It follows that $R\to R(\xi)$ is faithfully flat and unramified, with closed fiber   
equal to the residue field extension $k\sub k(\xi)$, where $k$ is the residue field of $R$ and $k(\xi)$ the field of 
fractions of $\pol k\xi$.

\begin{lemma}\label{L:dd}
 Let $(R,\maxim)$ be a local ring of finite embedding
dimension   and let $\xi$ be a tuple of indeterminates.
Then $R$ and $R(\xi)$ have the same \pdim\ and the same depth.
\end{lemma}
\begin{proof} 
Let $d$ be the \pdim\ of $R$ and $e$ its depth. We will 
induct on $d$ to show that $\pd{R(\xi)}=d$. It is
easy to see that $R$ is Artinian \iff\ $R(\xi)$ is, thus proving the 
case $d=0$. In the general case, we may choose
$x\in\maxim$ so that $\pd{R/xR}=d-1$. By induction, $(R/xR)(\xi)\iso R(\xi)/xR(\xi)$ has 
\pdim\ $d-1$, showing that $\pd {R(\xi)}\leq d$. On
the other hand, induction also shows that $\pd{R(\xi)}>d-1$, so that we 
get $\pd{R(\xi)}=d$, as required.

As for depth, this follows from \cite[Proposition 9.1.2]{BH} since 
$R\to R(\xi)$ is faithfully flat.
\end{proof}

We can now characterize depth in terms of regular sequences:

\begin{lemma}\label{L:NC} 
For a local ring  $R$ of finite embedding
dimension, its depth is equal to  the maximal length of an
$R(\xi)$-regular sequence, where $\xi$ runs over all finite tuples of 
indeterminates. More precisely, if $R$ has depth $e$, then we can find a regular
sequence  $\rij ye$ in $R\rij \xi e$ which  is part of a generic sequence.
\end{lemma}
\begin{proof} 
In view of Lemma~\ref{L:dd}, it suffices to prove the second assertion.  To
this end, we need to construct,  by Lemma~\ref{L:gen}, an $R(\xi)$-regular
 sequence $\rij ye$ such that the
\pdim\
of $R(\xi)/\rij
yeR(\xi)$ is  $d-e$, where $\xi:=\rij \xi e$. We induct on
the 
depth $e$ of $R$, where there is nothing to show if $e=0$. 
Let $\rij xd$ be a generic sequence and let
$\mathfrak n$ be the ideal generated by this sequence. Since $\mathfrak n$ is
then  $\maxim$-primary, its grade is   $e$. By
\cite[Proposition 9.1.3]{BH}, the element 
$$
y_1:=x_1+x_2\xi_1+\dots+x_d\xi_1^{d-1}
$$
 is an $\pol R{\xi_1}$-regular 
element.
Since $\pol R{\xi_1}\to R(\xi_1)$ is flat, $y_1$ is   $R(\xi_1)$-regular. 
Let $S:=R(\xi_1)/y_1R(\xi_1)$. Since
$S/(x_2,\dots,x_d)S\iso (R/\mathfrak n)(\xi_1)$, it is Artinian. 
Therefore, the \pdim\ of $S$ is at most $d-1$. By Lemma~\ref{L:dd}, the
\pdim\  of $S$ cannot be less, and hence it is equal to $d-1$. In particular,   we are done in case $e=1$.

Assume therefore $e>1$.   It follows from Lemma~\ref{L:dd} and 
\eqref{eq:defdep} that $S$ has depth $e-1$. By induction,
there exists an $S(\xi_2,\dots,\xi_e)$-regular sequence $(y_2,\dots,y_e)$ such that
$S(\xi_2,\dots,\xi_e)/(y_2,\dots,y_e)S(\xi_2,\dots,\xi_e)$ has \pdim\ $d-e$. 
Hence with $\xi:=\rij \xi e$, the sequence $\rij ye$ is
$R(\xi)$-regular and part of a generic sequence.
\end{proof}

\begin{remark}\label{R:NC} 
The argument even shows that, for a given generic
sequence $\rij xd$,  we may choose an $R(\xi)$-regular
 sequence 
$\rij ye$ so that
\begin{equation*} (y_1,\dots,y_e,x_{e+1},\dots,x_d)R(\xi)=\rij xdR(\xi).
\end{equation*}
In particular, if $R$ is moreover \creg, then we may take $\rij ye$ 
equal to a generating set of the maximal ideal of
$R(\xi)$. 
\end{remark}

For ultra-Noetherian rings, no such extension is necessary, since depth is first-order
definable:

\begin{proposition}\label{P:uldep}
The depth of an ultra-Noetherian local ring $R$ is equal to the maximal length
of an $R$-regular sequence.\qed
\end{proposition}

\subsection{Pseudo-singularities}\label{D:sCM}
We now introduce  some singularity
variants that are based on depth. Let $R$ be a local ring of finite embedding
dimension.  If the depth of  $R$ is equal to its embedding dimension, then   we call $R$
\emph\preg, and  if it is equal to its  \pdim, we call $R$ \emph\pCM. Immediate
from the definitions we get:

\begin{proposition}\label{P:pregCM}
A local ring of finite embedding dimension is \preg\ \iff\ it is \creg\
and \pCM.\qed
\end{proposition}

In order to derive  a homological characterization of
\preg{ity} analogous to Serre's characterization for regularity, we need some
additional definitions.

\subsection{Finite presentation type}\label{s:FFR}
We say that an $R$-module $M$ 
admits a \emph{finite free resolution} (of length $n$),
if there exists an exact sequence
\begin{equation}\label{eq:ffr}
 0\to F_n\to F_{n-1}\to\dots\to F_1\to F_0\to M\to 0
\end{equation} 
with each $F_i$ a finitely generated free $R$-module. 
The alternating sum of the ranks of the $F_i$ is
called the \emph{Euler number} $\op{Eul}(M)$ of $M$. It follows from 
Schanuel's Lemma that $\op{Eul}(M)$ does not depend on the
choice of finite free resolution, and by \cite[Theorem 19.7]{Mats}, it 
is always non-negative. Also,  if
\begin{equation*} 
0\to H\to G_m\to G_{m-1}\to\dots\to G_1\to G_0\to M\to 0
\end{equation*} 
is an arbitrary exact sequence with all $G_i$ 
finitely generated free $R$-modules, then  $H$   is also
finitely generated, and $\op{Eul}(M)$ is the alternating sum of the ranks of the 
$G_i$ and of $\op{Eul}(H)$ (see
\cite[\S19]{Mats} for more details).

In general, very few modules admit a finite free resolution, and hence we
introduce the following weaker version: we say that an $R$-module is
\emph{\fp
n}, if it admits finitely generated    $i$-th syzygies   for $i=\range 0n$, or
equivalently, if there exists an exact sequence as in \eqref{eq:ffr}, but
without the initial zero, with all $F_i$ finitely generated free $R$-modules.
Hence $M$ is \fp 0 \iff\ it is finitely generated, and $M$ is \fp 1 \iff\ it
is finitely presented. We will say that an $R$-module has \emph\fpt, if it 
is \fp n, for all $n$. Although
these definitions do not require $R$ to be local, the next one does: we
call a $R$-module complex $(G_\bullet,d_\bullet)$
   \emph{minimal} if the kernel of each morphism $d_i$ lies inside
$\maxim G_i$.

\begin{lemma}\label{L:min}
Let $(R,\maxim)$ be a local ring with residue field $k$. An $R$-module $M$  is
\fp n
\iff\
there exist a minimal exact sequence
\begin{equation*}
F_n\to F_{n-1}\to\dots\to F_1\to F_0\to M\to 0\tag{$F_\bullet$}
\end{equation*}
with each $F_i$ a finitely generated free $R$-module. Moreover, if this is the case then the
$i$-th Betti number $\beta^R_i(M)$ 
of $M$, that is to say, the vector space dimension of $\tor RiMk$, is equal to
the rank of $F_i$, for all $i\leq n$, showing that $F_\bullet$ is unique up to
isomorphism.
\end{lemma}
\begin{proof}
One
direction is immediate and the other can by induction be reduced to the
case that $M$ is \fp0, that is to say, finitely generated. This case is then
simply a reformulation of Nakayama's Lemma.  To prove the last assertion,
augment $F_\bullet$ by adding on the left a free module $F_{n+1}$, possibly of
infinite rank, which maps onto the kernel of
$F_n\to F_{n-1}$. Tensoring this exact sequence with $k$ gives a complex in
which all morphisms are zero and hence its  $i$-th homology is  
$F_i\tensor k$, for $i=\range 0n$. Since this homology is also
equal to $\tor RiMk$, we proved the second assertion.
\end{proof}

Since a projective module over a local ring is always free 
(\cite[Theorem 2.5]{Mats}), a necessary and sufficient
condition for an $R$-module $M$ to have a finite free resolution is 
that $M$ has projective dimension $n<\infty$ and is \fp n. By the previous
result, such a module then admits a unique minimal finite free resolution.

\begin{lemma}\label{L:FFR} 
Any direct summand of an $R$-module with a 
finite free resolution has itself a finite free
resolution. Similarly, any direct summand of a \fp n module is again \fp
n.
\end{lemma}
\begin{proof} 
We prove both results simultaneously. Suppose $M\oplus N$ has a finite free
resolution of
length $n$ as in \eqref{eq:ffr} (respectively, of the form $F_\bullet$). We will
show by induction on $n$
that $M$ has a finite free resolution (respectively, is \fp n). If $n=0$,
that is to say, if $M\oplus
 N$ is free, then $M$ is projective whence free (respectively, if $n=0$, that
is to say, $M\oplus N$ is finitely generated, then so is $M$).
Hence assume $n>0$ and choose an exact sequence
\begin{equation}
\label{eq:MN}
\Exactseq K {R^m} {M\oplus N}
\end{equation}
such that $K$  admits a finite free resolution of length $n-1$
(respectively is \fp{n-1}). 
Clearly, $M$ and $N$ must also be finitely
generated, so that we can choose exact sequences
\begin{align*} 
&\Exactseq G{R^a}M\\
&\Exactseq H{R^b}N.
\end{align*} 
Taking the direct sum of these last two exact sequences 
and comparing it with \eqref{eq:MN}, we get from
Schanuel's Lemma  an isomorphism $K\oplus R^a\oplus R^b\iso G\oplus 
H\oplus R^n$. Since the module at the left hand side
has a finite free resolution of length $n-1$ (respectively, is \fp{n-1}), our
induction
 hypothesis yields that $G$ has a  finite free resolution (respectively, is
\fp{n-1}),
whence so does $M$ (respectively, whence $M$ is \fp n).
\end{proof}

\begin{theorem}\label{T:sreg} 
A local ring of finite embedding dimension is \preg\ \iff\ its 
residue field admits a finite free resolution.
\end{theorem}
\begin{proof} 
Suppose first that $(R,\maxim)$ is \preg\ of \pdim\ 
$d$. Let $\tuple x$ be a generic sequence generating
$\maxim$. Since $R$ has depth $d$, all $H_i(\tuple x;R)$ vanish, 
showing that the Koszul complex $K_\bullet(\tuple x)$ of $\tuple x$ is
exact, yielding the desired finite free resolution of the residue field $k$.

Conversely, assume that $k$ has a finite free resolution
\begin{equation*} 
0\to R^{a_n}\to R^{a_{n-1}}\to\dots\to R^{a_1}\to 
R\to k\to 0.
\end{equation*} 
Let $m$ be the  embedding dimension of $R$ (so that we 
may choose $a_1=m$). Observe that both hypothesis
and conclusion are invariant under taking a Nagata extension of the form 
$R\sub R(\xi)$ (by faithful flatness), so that at
any time we may make such an extension if needed. There is nothing to 
show if $m=0$, so we induct on $m>0$. By
\cite[Theorem 19.6]{Mats}, the depth of $R$ must be positive. By 
Lemma~\ref{L:NC}, we may assume after making a Nagata
extension,  that some  minimal generator $x$ of 
$\maxim$ is $R$-regular. Put $S:=R/xR$, so that its
embedding dimension  is $m-1$. For each $i>1$, we have an isomorphism
\begin{equation*}
\tor Ri Sk\iso \tor R{i-1}S\maxim\iso\tor S{i-1}S{\maxim/x\maxim}=0
\end{equation*} 
 since   $x$ is 
$R$-regular, whence also $\maxim$-regular. This implies that
the complex
\begin{equation*} 
0\to S^{a_n}\to S^{a_{n-1}}\to\dots\to S^{a_1}
\end{equation*} 
is acyclic, that is to say, is a finite free 
resolution of $\maxim\tensor S=\maxim/x\maxim$. I claim
that $k$ is a direct summand of $\maxim/x\maxim$. Assuming the claim, 
  Lemma~\ref{L:FFR} then yields that $k$ admits a
finite free resolution as an $S$-module. Therefore, by our induction 
hypothesis, $S$ is \preg, whence has depth $m-1$.
It follows from \eqref{eq:defdep} that $R$ has depth $m$, showing 
that it is   \preg.

To prove the claim, choose $x_2,\dots,x_m\in\maxim$ so that 
$(x,x_2,\dots,x_m)R=\maxim$. Let $H$ be the $R$-submodule  of $\maxim/x\maxim$
 generated by the image of $x$. Hence $H\iso k$ and we want to show that $H$ is a
direct summand of $\maxim/x\maxim$.   Let $N$ be the submodule generated by the images of the 
$x_2,\dots,x_m$ in $\maxim/x\maxim$, so that $\maxim/x\maxim=H+N$. Let
$a\in\maxim$ and suppose its image    in $\maxim/x\maxim$ lies in 
$H\cap N$. It follows that we can write $a$ in two
different ways, namely as $a=a_1x=a_2x_2+\dots+a_mx_m+rx$ with 
$a_i\in R$ and $r\in\maxim$. By Nakayama's lemma, we
therefore must have $a_1\equiv r\equiv0\mod\maxim$, that it so say, 
$a\in x\maxim$. In other words, we showed that
$H\cap N=0$ and hence that $\maxim/x\maxim\iso H\oplus N$, as required.
\end{proof}

\begin{remark}\label{R:sreg}
Under the assumptions of the theorem,  $k$ has 
projective dimension equal to the \pdim\ of $R$ and
$\op{Eul}(k)=0$ (use the Koszul complex to calculate both numbers). The 
Koszul complex is minimal and therefore $\tor Rikk$
has dimension equal to $\binomial ni$ for all $i$.
\end{remark}

\begin{remark}
Using a similar argument, one can show that $R$ is 
\pCM\ \iff\ there exists a generic sequence $\tuple
x$ such that $R/\tuple xR$ has a finite free resolution (which then 
can be chosen to be the Koszul complex $K_\bullet(\tuple x)$ of $\tuple
x$). For a related result, see Proposition~\ref{P:uCMfpt} below.
\end{remark}

To not confuse with our present terminology we deviate from  \cite{Bert} or
\cite[\S5]{Glaz} by calling a ring  \emph{\bertin}, if every finitely generated
ideal has finite projective
dimension. If $R$ is moreover coherent, then it is shown that any finitely
generated ideal admits a finite free resolution. Applied to the maximal ideal,
we get immediately from Theorem~\ref{T:sreg}:

\begin{corollary}\label{C:bertin}
A coherent \bertin\ local ring   of finite
embedding dimension is
\preg.\qed
\end{corollary}

For the converse, we have the following:

\begin{corollary}\label{C:fp}
Let $(R,\maxim)$ be a \preg\ local ring of \pdim\ $d$,  and let be $M$ an $R$-module.
If
 $M$ is \fp {d+1},
then $M$ has finite projective dimension (at most $d$).
\end{corollary}
\begin{proof}
By Lemma~\ref{L:min}, there exists a minimal exact sequence $F_\bullet$
with $n=d+1$, and the $i$-th Betti number of $M$ is the rank of $F_i$.
However, $k$ has projective dimension $d$ by Remark~\ref{R:sreg}, and hence  
$\beta_{d+1}(M)=0$, showing that $F_{d+1}=0$.
\end{proof}

\begin{corollary}
Let $R$ be a \preg\ local ring of \pdim\ one. If $R$ is coherent, then it is
\bertin.
\end{corollary}
\begin{proof}
Let $I$ be a finitely generated ideal. Since $R$ is coherent, it is finitely
presented. Hence $R/I$ is \fp 2, and therefore has finite projective dimension
by Corollary~\ref{C:fp}.
\end{proof}

We cannot expect for this result to also hold if the \pdim\ $d$ is strictly
bigger than one,
since a
coherent \bertin\ ring is   \emph{\CM\ in the sense of \cite{HamMar}} and
therefore
admits a regular sequence of length   $d$ (that is to say, in such a ring,
naive depth always  equals depth).  To obtain a converse, we require a
stronger coherence condition:

\begin{theorem}\label{T:bertin}
A local ring of finite embedding dimension is coherent and  \bertin\ \iff\ it is
\preg\ and  every finitely
generated ideal has \fpt.
\end{theorem}
\begin{proof}
If $R$ is coherent and \bertin, then any finitely generated ideal has a finite
free resolution by \cite{GlazCoh}, whence has in particular \fpt. Moreover, $R$ is
\preg\
by Corollary~\ref{C:bertin}. To prove the converse, let $I$ be a finitely
generated ideal. By assumption, $I$, whence also $R/I$, is \fp n, and therefore
has finite projective dimension by Corollary~\ref{C:fp} applied with $n$
sufficiently large.
\end{proof}

In \cite{Soub}, Soublin calls a ring $R$ \emph{uniformly coherent}\footnote{This
is quite a strong
hypothesis, even for Noetherian rings, for which it forces, among other things,
that the dimension is at most two.} if
there
exists a function $\alpha\colon\nat\to\nat$ such that any morphism $R^n\to R$
has a kernel generated by at most $\alpha(n)$ elements. 

\begin{theorem}\label{T:unifcoh}
Let $R$ be a uniformly coherent local ring of finite embedding dimension. Then
every finitely generated ideal of $R$ has \fpt. In particular, $R$ is \preg\
\iff\
it is \bertin.
\end{theorem}
\begin{proof}
By \cite{Soub} or \cite[Corollary 2.3]{AschBou}, the countable direct
product $R^\nat$ is coherent. Since a finitely generated submodule of a
finitely generated free $R$-module embeds in $R^\nat$, it
is  finitely presented. Applied to the syzygies of a finitely generated
ideal $I$, we see that $I$ has \fpt. The second
assertion then follows from Theorem~\ref{T:bertin}. 
\end{proof}

\subsection*{Pseudo-\CM\ local rings}

Recall that we called $R$   \pCM, if its depth equals its \pdim.

\begin{theorem}\label{T:pCMcomp} 
A \pCM\ local ring  is \cCM.
\end{theorem}
\begin{proof} 
Let $R$ be a \pCM\ local ring of \pdim\ $d$ and let $\tuple x$ be a generic
sequence. Since $R$ has depth $d$, the grade of
$\mathfrak n:=\tuple xR$ is $d$, implying that all $H_i(\tuple 
x;R)$ vanish, for $i>0$. For $i=1$, this yields that
$\tuple x$ is   quasi-regular by \cite[Ch. X, \S9, Th\'eor\`eme 
1]{BourAlg}. Hence $R$ is \cCM\ by Theorem~\ref{T:pCM}.
\end{proof}

The converse is in general false: $R$ can be \cCM\ without 
being \pCM; an example is provided by the depth
zero \creg\ ring in \ref{e:preg}.   On the other hand, neither is it the 
case that in a \pCM\ local ring $R$ every
$R$-regular element is $\complet R$-regular. For instance, $R$ could 
be a non-separated domain, in which case any
non-zero element in the ideal of infinitesimals is $R$-regular, but 
zero in $\complet R$. This also gives an example of
an $R$-regular element which is not part of a generic subset. From the proof of
\cite[Theorem 16.3]{Mats}, it follows that if $R$ is separated and 
\cCM, then  every generic element is $R$-regular. In
particular, we showed that if $R$ has \pdim\ one, then $R$ is \cCM\ \iff\ $\sep
R$ is \pCM.

\begin{example}\label{e:isoCM}
Let $\seq Rw:=A/(\xi^2,\xi\zeta^w)A$ where $A:=\pow k{\xi,\zeta}$. It follows
that all $\seq
Rw$ have depth zero and dimension one. Hence their ultraproduct $\ul R$ has depth
zero
and ultra-dimension one. The cataproduct $\ulsep R$ is isomorphic to
$\pow{\ul k}{\xi,\zeta}/\xi^2\pow {\ul k}{\xi,\zeta}$, where $\ul k$ is the
ultrapower of $k$. This is
a one-dimensional \CM\ local ring. Hence $\ul R$ is \cCM\ and has \pdim\ one. In
conclusion, $\ul R$ is isodimensional and \cCM, but not \pCM.
%
\end{example}

\begin{corollary}\label{C:pregmult}
A local ring of finite embedding dimension is \preg\ \iff\ it is \pCM\ and has
multiplicity one.
\end{corollary}
\begin{proof}
The direct implication follows from Proposition~\ref{P:pregCM} and
Corollary~\ref{C:cregmult}. Conversely, if $R$
has multiplicity one and is \pCM, then it is \cCM\ by Theorem~\ref{T:pCMcomp},
whence \creg\ by  Corollary~\ref{C:cregmult}, and the result now follows from
Proposition~\ref{P:pregCM}.
\end{proof}

\begin{corollary} 
Let $R$ be a  \pCM\ local ring $R$ of \pdim\ two. 
If $R$ is either a domain or separated, then any generic sequence is $R$-regular.
\end{corollary}
\begin{proof} 
Let $(x,y)$ be a generic sequence in $R$. If $R$ is a domain, then $x$ is
$R$-regular. Let us show that the same holds if $R$ is separated. Since
$R$ has 
depth two by assumption, $H_2(x,y;R)=0$. This
means that whenever $ax+by=0$ for some $a,b\in R$ then 
$(a,b)=r(y,-x)$ for some $r\in R$. In particular, if $a\in \ann
Rx$, then $(a,0)=r(y,-x)$ for some $r\in R$, showing that $a\in y\ann 
Rx$. In other words, $\ann Rx=y\ann Rx$ so that by
induction $\ann Rx=y^n\ann Rx$ whence $\ann Rx\sub \inf R=0$. This 
concludes the proof that $x$ is $R$-regular. Using once more the
above characterization of $H_2=0$, we see  that in either case, $y$ is $R/xR$-regular, 
whence $(x,y)$ is $R$-regular.
\end{proof}

We can generalize Proposition~\ref{P:betti1} substantially  under an additional
\CM\
assumption.

\begin{proposition}\label{P:tor} 
Let $R$ be a local ring   of finite embedding dimension and let $M$ be an
$R$-module of finite length.
If $R$ is \pCM, then $\tor  Ri{\complet R}M$ vanishes for all $i>0$.
\end{proposition}
\begin{proof} 
Since $M$ has finite length, its annihilator is $\maxim$-primary, and hence
contains a generic sequence by Corollary~\ref{C:genprim}. 
Since $R\to R(\xi)$ is faithfully flat, the vanishing of 
the Tor's is unaffected by such an extension. Hence, after some Nagata
extension, we may assume, using   Remark~\ref{R:NC}, that $R$ admits  an
$R$-regular, 
generic sequence $\tuple x$ contained in the annihilator of $M$. Since $\complet
R$ is
\CM\ by Theorem~\ref{T:pCMcomp}, the sequence $\tuple x$ is also $\complet R$-regular. By a well-known deformation
property of Tor modules, we get
\begin{equation*}
\tor Ri{\complet R}M\iso\tor {R/\tuple xR}i{\complet R/\tuple x\complet R}M
\end{equation*} 
for all $i>0$. From this the vanishing then follows 
since ${R/\tuple xR}\iso{\complet R/\tuple
x\complet R}$ by Lemma~\ref{L:Art}.
\end{proof}

 Given a module $M$ over a
local ring $R$ of finite embedding dimension,  we
define its \emph\pdim\    to be  the \pdim\ of   $R/\ann RM$,  
and we denote it  $\pd M$. Since the notions of grade and depth also extend to
modules, we may call a
finitely generated $R$-module $M$   \emph\pCM, if
its \pdim\ equals its depth. 
 
\begin{corollary}\label{C:tor}
Let $R$ be a local ring   of finite embedding dimension and let $M$ be a
finitely generated $R$-module. If both $R$ and $M$ are \pCM, then $\tor
Ri{\complet R}M=0$, for all $i>0$.
\end{corollary}
\begin{proof}
We induct on the \pdim\ $e$ of $M$. If $e=0$, then $M$ is a finitely generated
module over the Artinian local ring $R/\ann RM$, whence has finite length, and
the result follows from Proposition~\ref{P:tor}. So assume $e>0$. As far as
proving the vanishing  is concerned, we may always, by faithfully flat descent,
take a Nagata extension of $R$. Hence, by the module analogue of
Lemma~\ref{L:NC} (the proof of which is left to the reader), we may assume,
after possibly taking a Nagata extension, that $x$ is an $M$-regular element.
From the exact sequence
$$
\exactseq MxM{}{M/xM}
$$
we get, by tensoring with $\complet R$, part of a long exact sequence
$$
  0=\tor R{i+1}{\complet R}{M/xM}\to \tor  Ri{\complet R}M\map x\tor Ri{\complet
R}M\to \tor R{i+1}{\complet R}{M/xM}=0 
$$
where the two outer modules are zero by induction. Fix $i$ and put $T:= \tor 
Ri{\complet R}M$. Since $T=xT$, we have $T=\inf RT$. As $\complet R$ is
Noetherian, $\inf R\complet R$ vanishes, whence so does $\inf RT$, since $T$ is
the homology of a complex of modules over $\complet R$. This shows $T=0$,
completing our proof.
\end{proof}

\begin{example}\label{s:aa} 
In \cite{SchAA}, a class of local rings 
was introduced  which extends the class of \CM\
local rings. More precisely, for each $d,e\geq 0$, let $CM_{d,e}$ be 
the class of all local rings $R$ such that there
exists an $R$-regular sequence of length $d$ and such that the 
minimal length of a homomorphic image $R/\tuple xR$ is
$e$, where $\tuple x$ is an arbitrary tuple in $R$ of length $d$. 
The latter condition implies that $R$ has \pdim\
at most $d$, and   the former that its depth is at least $d$. It 
follows that $R$ is \pCM\ of \pdim\ $d$. Let $\tuple
x$ be an arbitrary tuple of length $d$. Suppose   $R/\tuple xR$ is 
Artinian of length $l$ (by assumption $l\geq e$).
Hence $R/\tuple xR\iso\complet R/\tuple x\complet R$ and $\tuple 
x$ is generic in $R$. Moreover, $\tuple x$ is
$\complet R$-regular, since $\complet R$ is \CM. It follows that the 
ideal $\tuple x\complet R$ has multiplicity $l$.
For a general choice of system of parameters $\tuple y$ in $\complet 
R$, the ideal $\tuple y\complet R$ is a reduction
of $\maxim\complet R$ (\cite[Theorem 14.14]{Mats}), so that the 
multiplicity of $\tuple y\complet R$ is equal to $\mult{\complet R}$. By
assumption, the minimal value of the 
multiplicity of an ideal generated by a $d$-tuple
from $R$ is $e$. Since these form a general subset of all $d$-tuples 
in $\complet R$, we showed that $\complet R$ has
multiplicity $e$, whence so does $R$ by Remark~\ref{R:mult}. In fact, we have
the following characterization of these classes:
\end{example}

\begin{theorem}\label{T:CMde}
A local ring $R$ is \pCM\ of \pdim\ $d$ and multiplicity 
$e$ \iff\ $R(\xi)$ belongs to the class $CM_{d,e}$
for some ($d$-)tuple of indeterminates $\xi$.
\end{theorem} 
\begin{proof}
Since $R$ and $R(\xi)$ 
are easily seen to have the same multiplicity (by
comparing their completions), one direction follows from the previous 
discussion. On the other hand, suppose $R$ is
\pCM\ of \pdim\ $d$ and multiplicity $e$. By the same argument as 
above, we may choose a generic sequence $\tuple x$ in
$R$ such that $\tuple x\complet R$ is a reduction of $\maxim\complet 
R$, whence has multiplicity $e$. It follows from
Lemma~\ref{L:Art} that $R/\tuple xR$ has length $e$ and by a similar 
argument that this is the least possible length.
In order to construct an $R$-regular sequence, we have to go to an 
extension $R(\xi)$ by Lemma~\ref{L:NC} and this
extension is then in the class $CM_{d,e}$. 
\end{proof}

In particular, by Corollary~\ref{C:pregmult}, a local ring $R$ is \preg\ \iff\
$R(\xi)$ belongs to $CM_{d,1}$ for some $d$ and some $d$-tuple of
indeterminates $\xi$. Moreover, by
Proposition~\ref{P:uldep}, an ultra-Noetherian local ring belongs to
$CM_{d,e}$ \iff\ it is \pCM\ of \pdim\ $d$ and multiplicity $e$.

Let $R$ be a local ring of finite embedding dimension. We  say that $R$  is 
\emph\pgor, if it is   \pCM\ and  there exists a generic sequence $\tuple x$
such that $R/\tuple xR$ is an Artinian Gorenstein ring.

\begin{proposition}
A \pCM\ local ring is \pgor\ \iff\ it is \cgor.

In fact, let  $(R,\maxim)$ be a \pCM\ local ring  of \pdim\ $d$ and let $k$
be its residue field.  If $R$ is \pgor, then $\ext RikR=0$, for all
$i\neq d$ and $\ext RdkR\iso k$. Conversely, if $\ext RikR$ vanishes for some
$i>d$ or if $\ext RdkR\iso k$, then $R$ is
\pgor.
\end{proposition}
\begin{proof}
Let $\tuple x$ be a generic sequence in $R$. By Lemma~\ref{L:dd}, the extension
$R(\xi)$ is again \pCM\ and $\tuple x$ is generic in $R(\xi)$. Since 
 $$
 R/\tuple xR\to (R/\tuple xR)(\xi)\iso R(\xi)/\tuple xR(\xi)
 $$
  is   faithfully flat and unramified, the former is  Gorenstein \iff\ the latter
is.
 Since the Ext-functors commute with faithfully flat base change, 
 we may replace $R$ by $R(\xi)$ everywhere and assume by Lemma~\ref{L:NC} that  
$\tuple x$ is a regular sequence.

In particular, $R$ is \pgor\ \iff\ $R/\tuple xR\iso \complet R/\tuple
x\complet R$ is Gorenstein \iff\ $\complet R$ is Gorenstein, since   $\tuple x$ is  $\complet
R$-regular. This already proves the first assertion. Since $\tuple x$ is $R$-regular, we have 
\begin{equation}\label{eq:Rees}
\ext RikR\iso \ext {R/\tuple xR}{i-d}k{R/\tuple xR}
\end{equation}
where we let $\ext
Rj\cdot\cdot$   be the zero functor for negative $j$ (see for instance \cite[Lemma
3.1.16]{BH} and the proof of (3)~$\Leftrightarrow$~(1) of \cite[Theorem
16.6]{Mats}). The final assertion now follows from \cite[Theorem
18.1]{Mats} applied to the Artinian local ring $R/\tuple xR$.
\end{proof}

 It follows that if $R$ is \pgor, then $R/\tuple xR$ is Gorenstein for every
generic sequence $\tuple x$.

\section{Ultra-singularities}\label{s:ulsing}

We now compare the `cata' and `pseudo' versions from the previous two sections
with their `ultra' counterparts. Throughout this section, unless mentioned
explicitly, $\ul R$ is an ultra-Noetherian local ring with maximal ideal
$\ul\maxim$ and residue field $\ul k$,
realized
as the ultraproduct of Noetherian local rings $(\seq Rw,\seq\maxim w)$ of bounded
embedding
dimension and  residue field $\seq kw$. Recall (Lemma~\ref{L:ulcomp})  
that the cataproduct $\ulsep R$ of the $\seq Rw$ is the separated quotient as
well as the completion of $\ul R$, and it is in particular Noetherian.

\begin{theorem}\label{T:psing} 
For an ultra-Noetherian local ring $\ul R$, the following are equivalent
\begin{enumerate}
\item\label{c:sreg} $\ul R$ is  \preg;
\item\label{c:areg} $\ul R$ is ultra-regular;
\item\label{c:preg} $\ul R$ is  \creg\ and isodimensional.
\end{enumerate} 
\end{theorem}
\begin{proof} 
Let $\ul R$ be the ultraproduct of Noetherian local rings $\seq Rw$ of bounded
embedding dimension. If $\ul R$ is
\preg, then it is isodimensional by Theorem~\ref{T:dimineq} and therefore
\creg\ by Theorem~\ref{T:preg}. Moreover, by \los, almost
all $\seq Rw$ are regular since embedding dimension and depth are first-order
definable. This shows that $\ul R$ is ultra-regular, and  the
converse follows along the same lines. Finally, if $\ul R$ is \creg\
and isodimensional, then it is \preg, again by Theorem~\ref{T:dimineq}.
\end{proof}

The same proof also shows that $\ul R$ is ultra-regular \iff\ it is not
ultra-singular. In view of Lemma~\ref{L:ulcomp}, we may rephrase the theorem
as follows:

\begin{corollary}
Let $\seq Rw$ be Noetherian local rings of the same  dimension and
parameter degree and let $\ulsep R$ be their cataproduct. Then almost
all $\seq Rw$ are regular \iff\ $\ulsep R$ is. \qed
\end{corollary}

\begin{corollary}\label{C:locpreg}
Any localization of an ultra-regular local ring at a \fr\ prime ideal is ultra-regular. 
\end{corollary}
\begin{proof}
Let $\ul R$ be an ultra-regular local ring, obtained as the ultraproduct of $d$-dimensional regular local rings $\seq Rw$, and let $\pr\in\conspec{\ul R}$. By Proposition~\ref{P:ulscheme}, there exist prime  ideals $\seq\pr w\sub\seq Rw$ whose ultraproduct is equal to $\pr$. Since almost each $(\seq Rw)_{\seq\pr w}$ is regular of dimension at most $d$,   their ultraproduct $(\ul R)_\pr$ is ultra-regular (of \pdim\ at most $d$).
\end{proof}

We conclude our discussion of ultra-regular rings with an ultraproduct version of Corollary~\ref{C:ulsepreg}. 

\begin{corollary}
The canonical embedding    $R\to \ul R$ of an excellent local ring in its ultrapower has ultra-regular fibers at \fr\ prime ideals: for every $\pr\in\conspec {\ul R}$, the fiber ring 
$(\ul R/\primary\ul R)_\pr$ is ultra-regular, where $\primary=\pr\cap R$.
\end{corollary}
\begin{proof}
To show that $(\ul R/\primary\ul R)_\pr$ is ultra-regular,  we may replace  $R$ by $R/\primary$, since $\ul R/\primary \ul R$ is the ultrapower of $R/\primary$, and assume without loss of generality that $R$ is a domain  and $\pr\cap R=(0)$. By Corollary~\ref{C:loculpow}, the localization $(\ul R)_\pr$ has finite embedding dimension, and $\pr$ is the ultraproduct of prime ideals $\seq\pr w\in\spec R$. Since $R$ is an excellent domain, its singular locus is a proper, closed subset, say, defined by a non-zero ideal $I\sub R$. If almost all $\seq\pr w$ would belong to this singular locus, then they would almost all  contain $I$, whence so would $\pr$, contradicting that $\pr\cap R=(0)$. Hence almost all $\seq\pr w$ are in the regular locus, and the result now follows from the proof of Corollary~\ref{C:locpreg}.\end{proof}


%

\subsection*{Ultra-\CM\ local rings}
Recall that $\ul R$ is called \emph{ultra-\CM} if almost all $\seq Rw$ are
\CM. We can characterize this property in terms of the fundamental inequalities \eqref{eq:uldimineqintro}.

\begin{theorem}\label{T:ulCM}
For an ultra-Noetherian local ring $\ul R$, the following are equivalent
\begin{enumerate}
\item  $\ul R$ is ultra-\CM;
\item the depth of $\ul R$ equals its ultra-dimension.
\end{enumerate}
In particular, $\ul R$ is \pCM\ \iff\ it is ultra-\CM\ and isodimensional.
\end{theorem}
\begin{proof}
The first assertion follows immediately from the fact that depth is 
first-order. The second assertion is now also clear,
since a \pCM\ must be  isodimensional by Theorem~\ref{T:uldim}. 
\end{proof}

\begin{remark}
Note that unlike in the regular case, isodimensionality
together with being \cCM\ is not sufficient for being \pCM, as
example~\ref{e:isoCM} shows.

Also note that ultra-\CM\ does not imply \pCM\ nor even \cCM. Namely, let
$(R,\maxim)$ be a non-\CM\ local ring and let $\ul R$ and $\ulsep R$ be the
respective ultraproduct and cataproduct 
of the $R/\maxim^n$.   Corollaries~\ref{C:filt} and \ref{C:ulsepreg} together imply
that  $\ulsep R$ is  not \CM. Hence $\ul R$ is not
\cCM, although
it is clearly
ultra-\CM\ (there is no contradiction with the above theorem, since $\ul R$ is not 
isodimensional). 
\end{remark}

\begin{corollary}\label{C:CMulsep}
The cataproduct of \CM\ local rings having the same dimension
and the same multiplicity, is again \CM.
\end{corollary}
\begin{proof}
Let $\ul R$ and $\ulsep R$ be the respective ultraproduct and 
cataproduct of  Noetherian local rings $\seq Rw$ of the same
multiplicity and the same dimension. If almost all $\seq Rw$ are \CM, then $\ul R$ is
 isodimensional by Remark~\ref{R:isodim}. Therefore, $\ul R$ is \pCM\
by Theorem~\ref{T:ulCM}, and hence   $\ulsep R$ is \CM\ by Theorem~\ref{T:pCMcomp}.
\end{proof}

 Let us call an ultra-module $\ul M$, that is to say, an
ultraproduct of $\seq Rw$-modules $\seq Mw$, \emph{ultra-\CM}, if almost all
$\seq Mw$ are \CM. Although such a module need not be finitely generated, we
 have:

\begin{lemma}\label{L:pCMmod}
For each $w$, let $\seq Mw$ be a finitely generated module over $\seq Rw$, and
let $\ul M$ be their ultraproduct. If
almost all $\seq Rw$ are \CM, of the same dimension and multiplicity, then
$\ul M$ is   finitely generated and \pCM\   
\iff\ almost all $\seq Mw$ are \CM\ of the same multiplicity.
\end{lemma}
\begin{proof}
If almost
all $\seq Mw$ are \CM\ of multiplicity $l$, then there
exists, by \cite[Theorem 4.6.10]{BH}, an $\seq
Rw$-regular and $\seq Mw$-regular  sequence $\seq{\tuple x}w$ such that $\seq
Mw/\seq{\tuple x}w\seq Mw$ has length $l$. Since each sequence can have length at most $d$, almost 
all have the same length $s\leq d$.  The ultraproduct $\ul M/\ul{\tuple
x}\ul M$, too,  has length $l$ by Proposition~\ref{P:ullen}, where $\ul{\tuple x}$ is the ultraproduct
of the $\seq {\tuple
x}w$. In particular, $\ul M$ is finitely generated. Moreover, $\ul{\tuple x}$
is $\ul M$-regular, showing that $\ul M$ has depth at least $s$.   On the other hand, since $\ul
M/\ul{\tuple x}\ul M$ has finite length, the \pdim\ of $\ul M$ is at most $s$.
This proves that $\ul M$ is   \pCM. 

Conversely, assume $\ul M$ is \pCM\ and finitely generated. As depth is
first-order, by  the (module version of)
Proposition~\ref{P:uldep}, there exists an $\ul M$-regular sequence 
$\ul{\tuple x}$  such that $\ul M/\ul{\tuple x}\ul M$ has \pdim\ zero. 
As $\ul M$ is
finitely
generated,   $\ul M/\ul{\tuple x}\ul M$ has  finite length, say,
$l$. Letting $\seq{\tuple x}w$ be tuples in $\seq Rw$ having as
ultraproduct $\ul{\tuple x}$, the ultraproduct of the $\seq Mw/\seq{\tuple
x}w\seq Mw$ is equal to  $\ul M/\ul{\tuple x}\ul M$, and hence almost all $\seq Mw/\seq{\tuple
x}w\seq Mw$ have length $l$ by  Proposition~\ref{P:ullen}. Moreover, almost each $\seq{\tuple x}w$ is $\seq
Mw$-regular, showing that $\seq Mw$ is \CM, of multiplicity at most $l$, by
another application of \cite[Theorem 4.6.10]{BH}.  
\end{proof}

The next result, which is some type of coherence property for ultra-\CM\ local
rings,
will be used in
\S\ref{s:ub} to deduce some uniform bounds
on Betti numbers.
Recall that the
\emph{$i$-th Betti number} $\beta_i(M)$ of a module over a local ring $R$ with
residue field $k$ is the (possibly infinite) dimension of $\tor RiMk$; for the
notion of \fpt, see \S\ref{s:FFR}.

\begin{proposition}\label{P:uCMfpt}
If $\ul R$ is an isodimensional, ultra-\CM\ local ring, then every finitely
generated \pCM\ $\ul R$-module (e.g., every $\ul R$-module of finite length)
has
\fpt. More precisely, for any given $e$, if almost each $\seq Mw$ is a \CM\ $\seq Rw$-module of
multiplicity $e$,  then, for each $n$, 
almost all   $\seq Mw$ have the same
$n$-th Betti number as their ultraproduct $\ul M$ and as their cataproduct  $\ulsep M$.
 \end{proposition}
\begin{proof}
In view of Lemma~\ref{L:pCMmod}, it suffices to prove the second assertion. We 
will show,  by induction 
on $n$, that 
$$
\beta_n(\ul M)=\beta_n(\ulsep M)=\beta_n(\seq Mw)
$$
for almost all $w$. The case $n=0$ follows from Proposition~\ref{P:ullen}, since $\ul M$ is
finitely generated by Lemma~\ref{L:pCMmod}. So assume $n\geq1$. 

Let 
$$
\seq {F_{n,}}w \to \seq {F_{n-1,}}w\to\dots\to \seq {F_{1,}}w\to \seq Mw\to 0
$$
be a minimal finite free resolution of   $\seq Mw$, with each $\seq{F_{i,}}w$ 
a finite free $\seq Rw$-module of rank $\seq{r_{i,}}w:=\beta_i(\seq Mw)$ (see
\S\ref{s:FFR}). 
Taking ultraproducts, we get by \los\ a minimal  resolution 
\begin{equation}\label{eq:ulfpt}
 \ul {F_{n,}}\to \ul{F_{n-1,}}\to\dots\to \ul {F_{1,}}\to \ul M\to 0
\end{equation}
By induction and Lemma~\ref{L:min}, we get $\ul{F_{i,}}\iso \ul R^{r_i}$, for $i<n$,  where $r_i$ is the common value of almost all  $\beta_i(\seq Mw)$.  
Theorem~\ref{T:ulCM} implies that $\ul R$ is \pCM, and hence  by
Corollary~\ref{C:tor}, all
$\tor {\ul R}i{\ulsep R}{\ul M}$ vanish for   $i>0$. Therefore, if we tensor \eqref{eq:ulfpt} with $\ulsep
R$, we get again a minimal  resolution
$$
 \ulsep {F_{n,}}\to \ulsep R^{r_{n-1}}\to\dots\to \ulsep R^{r_1}\to \ulsep M\to 0
$$
Since $\ulsep R$ is Noetherian and the resolution is minimal,
$r_i=\beta_i(\ulsep M)$ for $i<n$, and the
last module
in this resolution, $\ulsep {F_{n,}}$,  is generated by $r_n:=\beta_n(\ulsep M)$ elements. 
Tensoring with the common residue field $\ul k$ of $\ul R$ and $\ulsep R$, we
get 
$$
\ul k^{r_n}\iso \ulsep{F_{n,}}/\ul\maxim\ulsep{F_{n,}}\iso\ul{F_{n,}}/\ul\maxim\ul{F_{n,}}.
$$
 Since the latter module
is   the ultraproduct of the $\seq {F_{n,}}w/\seq\maxim w
\seq {F_{n,}}w\iso \seq kw^{\seq {r_{n,}}w}$, where $\seq kw$ is the residue
field of $\seq Rw$, we get $\seq{r_{n,}}w=r_n$ for almost all $w$, as we wanted to show.
\end{proof}

\begin{theorem}
A  \pCM\ ultra-Noetherian local ring  $\ul R$ is \cgor\  \iff\
it is ultra-Gorenstein; and it is  a \cci\  \iff\ it is an ultra-`complete
intersection'. 

In particular,
if $\seq Rw$ are \CM\ local rings having the same dimension and
multiplicity, then their cataproduct $\ulsep R$ is respectively Gorenstein
or a complete intersection  \iff\ so are almost all $\seq Rw$.
\end{theorem}
\begin{proof}
The second assertion follows from the first in
view of Theorem~\ref{T:isodim} and   Theorem~\ref{T:ulCM}. We already observed
that $\ul R$ is isodimensional, so that $\ulsep R$ and almost all $\seq Rw$ have the
same dimension, $d$, say. Hence if $\tuple {\ul x}$ is a generic sequence in $\ul
R$,
realized as an ultraproduct of tuples $\seq {\tuple x}w$ in  $\seq Rw$, then
almost each $\seq{\tuple x}w$ is a system of parameters in $\seq Rw$, whence
$\seq Rw$-regular.
Therefore, almost all $\seq Rw$ are Gorenstein \iff\ so are almost all $\seq Rw/\seq{\tuple
x}w\seq Rw$. This in turn  is equivalent with $\ul R/\tuple
{\ul x}\ul R$ being Gorenstein by \los\ (using that these are Artinian local rings;
see
\cite{SchEC} for more details). Since $\ul R/\tuple {\ul x}\ul R\iso \ulsep R/\tuple
{\ul x}\ulsep R$, the latter is then equivalent with 
  $\ulsep R$ being Gorenstein.

By Proposition~\ref{P:uCMfpt}, we have a minimal free resolution of $\ulsep
R$-modules
 $$
\ulsep R^r \to \ulsep R^m\to \ulsep R\to \ul k\to 0
$$
 where $r=\beta_2(\ul k)=\beta_2(\seq kw)$ and
$m=\beta_1(\ul k)=\beta_1(\seq kw)$, for
almost all $w$. Moreover,  $\ulsep R$ has the same dimension $d$ as almost all
$\seq Rw$ by Theorem~\ref{T:ulCM}.   
By \cite[Theorem 2.3.3]{BH}, therefore,  
$\ulsep R$ is a complete intersection \iff\  $r=m(m+2)/2-d$, \iff\
almost all $\seq Rw$ are complete intersections.
\end{proof}

%
%

\subsection*{Lefschetz Hulls}
In \cite{SchAsc}, we showed that every Noetherian local ring $R$ of equal \ch\ 
zero (that is to say, containing the rationals) admits an ultra-Noetherian
faithfully flat extension
$\hull R$ which is \emph{Lefschetz}, meaning that $\hull R$ is the ultraproduct
of
 Noetherian local rings
$\seq Rw$ of prime \ch. In fact, the $\seq Rw$ may be chosen to be complete 
with algebraically closed residue field. We call
$\hull R$ a \emph{Lefschetz hull} of $R$. Although the construction can be made
more functorial, it still depends on a choice of a cardinal number larger than
the cardinality of $R$.
However, in case $R$ is of finite type over an uncountable\footnote{Strictly
speaking, of cardinality equal to $2^\lambda$, for some infinite cardinal
$\lambda$.} 
\acf\ of \ch\ zero, there is a canonical choice for $\hull R$, called the 
\emph{non-standard hull} of $R$ and denoted $R_\infty$; see \cite{SchNSTC} for
details. In view of our  characterizations of 
pseudo-singularities in this section, the following result is immediate from
\cite[Theorem 5.2]{SchAsc}:

\begin{theorem}
 A Noetherian local ring $R$ of equal \ch\ zero with Lefschetz hull $\hull R$  
is \CM\ (respectively, Gorenstein or   regular) \iff\   $\hull R$ is \pCM\ (respectively,
\pgor\
or \preg).
\end{theorem}

\section{\C-normalizations}\label{s:catanorm}

An extremely useful fact in commutative algebra is the existence of Noether
normalizations: any finitely generated algebra over a field or any complete
Noetherian local domain admits a regular subring over which it is module-finite.
This result is not hard to show in equal \ch, so that we will adopt this
additional assumption in this section to formulate an analogue for   local rings of finite embedding
dimension. In the sequel, let  $(R,\maxim)$ be an equi\ch\ local ring with
residue field $k$ and let $\pi\colon R\to k$ denote the induced  surjection.

\subsection*{Weak coefficient fields}  
A subfield 
$\kappa$ of   $R$ is called a
\emph{weak coefficient field} of $R$ if the restriction of $\pi$ to $\kappa$
induces an algebraic extension $\pi(\kappa)\sub k$.   If this extension is an
isomorphism, then we call $\kappa$ a \emph{coefficient field} of $R$ (in the
literature one also encounters the notion of a \emph{quasi-coefficient}
defined as a weak coefficient field $\kappa$ for which the induced extension
$\pi(\kappa)\sub k$ is also    separable). The next result is well-known, but
its proof is included for convenience.

\begin{lemma}\label{L:qc} 
Let $(R,\maxim)$ be an equi\ch\   local 
ring. For any subfield $\kappa_0$ of $R$, we can
find a weak coefficient  field $\kappa$ of $R$ containing $\kappa_0$.

If, moreover, $R$  has \ch\ zero and is Henselian, then we can choose $\kappa$ to
be a  coefficient field.
\end{lemma}
\begin{proof} 
Let $\kappa$ be maximal among all subfields of $R$ 
containing $\kappa_0$ (such a field exists by Zorn's lemma). We need to show that
the extension $\pi(\kappa)\sub k$ is algebraic, where $k$ is the residue field 
of $R$ and $\pi\colon R\to k$ the residue map.  
To this end, take an arbitrary element $u\in 
k\setminus\pi(\kappa)$. Let $a\in R$ be such that $\pi(a)=u$.  It follows that
$a\notin \kappa$. By maximality, the subring 
$\pol \kappa a$ of $R$ generated by $a$ must contain a
non-zero non-invertible element (lest $\kappa(a)$ be a larger 
subfield of $R$). This means that $P(a)\in\maxim$, for
some non-zero $P\in\pol \kappa \xi$. Hence taking reductions, we get 
$P^\pi(u)=0$ in $k$, where $P^\pi$ is the
polynomial obtained from $P$ by applying $\pi$ to each of its 
coefficients. Since $P^\pi$ is not identical zero,  $u$ is algebraic over $\pi(\kappa)$.

To prove the last assertion, assume by way of contradiction that $R$ has \ch\
zero and is Henselian, but that $\pi(\kappa)$ is strictly
contained in $k$. Take    $u\in k\setminus\pi(\kappa)$. Let 
$p$ be a minimal equation of $u$ over
$\pi(\kappa)$ and let $P\in\pol \kappa \xi$ be such that its image 
$P^\pi$ is equal to $p$. Since $u$ is a
single root of $p$, Hensel's Lemma yields  the existence of a root $a\in R$  of $P$ with 
$\pi(a)= u$. However, this
implies that the field $\pol \kappa \xi/P\pol \kappa \xi$ embeds in $R$ 
via the assignment $\xi\mapsto a$, contradicting the
maximality of $\kappa$.
\end{proof}

A local \homo\ $A\to R$ is called \emph{\c-integral} (respectively,
\emph{\c-finite},
\emph{\c-injective},
\emph{\c-surjective}, \emph{\c-flat}) if its completion 
$\complet A\to \complet R$ is integral
(respectively, finite, injective, surjective, flat).  Let $(R,\maxim)$ be a local
ring of finite embedding dimension.

\subsection*{\C-normalization} 
A \emph{\c-normalization} of $R$   is a 
\c-integral local \homo\ $\theta\colon (A,\pr)\to
(R,\maxim)$  such that $A$ is a (Noetherian) regular local ring and  $\pr R$ is 
$\maxim$-primary. We say that a \c-normalization
$\theta$  is   \emph{Cohen}, if $\pr R=\maxim$, and \emph{Noether} if 
$\theta$   is injective.

%

\begin{theorem}\label{T:qN} 
An equi\ch\ local ring of finite 
embedding dimension admits a \c-normalization, which can be
chosen to be either Cohen or Noether.
\end{theorem}
\begin{proof} 
Let $(R,\maxim)$ be an equi\ch\ local ring of 
finite embedding dimension. By Lemma~\ref{L:qc}, there
exists a weak coefficient field  $\kappa$ of $R$. Choose  a tuple 
$\tuple x:=\rij xn$ generating some $\maxim$-primary
ideal. Let $A$ be the localization of the polynomial ring $\pol 
\kappa \xi$ at the ideal generated by the indeterminates 
$\xi=\rij \xi n$. Let $\theta\colon A\to R$ 
be the (unique) $\kappa$-algebra \homo\ which sends
$\xi_i$ to $x_i$, for each $i$. To show that $\theta$ is a 
\c-normalization, we only need to show that its completion
is integral, since the other conditions are immediate. Therefore, 
without loss of generality, we may already assume that
$A$ and $R$ are complete, so that both rings are now Noetherian. Let 
$\pi\colon R\to k$ be the residue map and let $l$
be a finite extension of $\pi(\kappa)$ contained in $k$. Put 
$B_l:=\inverse\pi l$. Since $\kappa+\maxim\sub B_l$, one
checks easily that $B_l$ is a local ring with maximal ideal 
$\maxim$. The local \homo\ $A\to B_l$ induces a finite
extension of residue fields. Therefore, since $A$ is complete  and 
$B_l$ is separated, $B_l$ is finitely generated   as
an $A$-module by \cite[Theorem 8.4]{Mats}. Since $k$ is the union of 
all its finite extensions $l$ containing
$\pi(\kappa)$, so is  $R$   the union of all the $B_l$, showing 
that $R$ is integral over $A$.

It is clear that if we choose $\tuple x$ so that it generates 
$\maxim$, then $\theta$ is Cohen. Assume next that
$\tuple x$ is a generic sequence.
In particular, $\complet R$ has 
dimension $n$ by Theorem~\ref{T:pdim}. Since $\complet A$ is an
$n$-dimensional domain and  $\complet A\to\complet R$ is integral,
this map must be injective. But then so must $A\to R$ be, that is to say, $\theta$
is Noether.
\end{proof}

\begin{remark}\label{R:qnid} 
If $I$ is a finitely generated ideal of 
$R$, then we can always choose a
\c-normalization $A\to R$ with the additional property that there 
is some ideal $J\sub A$ with $JR=I$. Simply
choose $\tuple x$ so that it contains a set of generators of $I$.
\end{remark}

\begin{remark}
From the above proof it is also clear that if $R$ admits a coefficient field, 
then we can choose the \c-normalization $A\to R$ to be \c-finite.
\end{remark}

\begin{theorem}\label{T:cflatcCM}
An equi\ch\  local ring $R$ of finite embedding dimension   is 
\cCM\ \iff\ there exists a \c-flat,
\c-normalization $A\to R$.
\end{theorem}
\begin{proof} 
If  $\complet A\to \complet R$ is flat with $A$ 
regular, then $\complet R$ is \CM\ by  \cite[Corollary to Theorem 23.3]{Mats},
since   the 
closed fiber has dimension zero. This proves one
direction. To prove the converse 
implication,  assume that $\complet R$ is \CM. Let
$A\to R$ be any Noether \c-normalization. Since
$\complet A\to
\complet R$ is a local \homo\ of Noetherian local
rings of the same dimension, with  closed fiber having dimension zero,  it  is flat by \cite[Theorem 23.1]{Mats},  because
$\complet A$   is regular and $\complet R$  is \CM.
\end{proof}

 From the proof it follows that any Noether \c-normalization of a 
\cCM\ local ring is \c-flat. We conclude this section with an instance of true
Noether Normalization:

\begin{theorem}\label{T:ulNN}
If  $\ul R$ is  an ultraproduct of equi\ch\ complete $d$-dimensional Noetherian
local rings, then $\ul R$ is isodimensional  \iff\
there exists an  ultra-regular local subring $\ul S\sub \ul R$ such that $\ul R$
is module-finite over it.
\end{theorem}
\begin{proof}
Let us show that the if-direction holds for any ultra-Noetherian local ring
of finite embedding dimension. Let  
 $\ul S\sub \ul R$ be a finite extension with $\ul S$
ultra-regular, realized as the ultraproduct of regular local rings $\seq Sw$.
By Proposition~\ref{P:ullen},
if $\ul
R$ is generated by at most $N$ elements over $\ul S$, then almost each $\seq
Rw$ is generated by at most $N$ elements over $\seq Sw$. 
If $\seq{\tuple y}w$ is a regular system of parameters in $\seq Sw$, then its
image is a system of parameters in $\seq Rw$. Since $\seq Rw/\seq{\tuple
y}w\seq Rw$ has vector space dimension at most $N$ over the residue field of $\seq Sw$, 
its length is at most $N$, showing that the $\seq
Rw$ have bounded parameter degree. Hence, $\ul R$ is isodimensional by
Theorem~\ref{T:isodim}.

Conversely, assume $\ul R$ is as in the statement, so that in particular its \pdim\
is $d$. By Theorem~\ref{T:isodim}, almost all $\seq Rw$ have parameter degree $\rho$, for
some $\rho<\infty$.  By \cite[Corollary 3.8]{SchABCM}, almost each $\seq Rw$ is a
module-finite extension of a regular subring
$\seq Sw$, generated as an $\seq
Sw$-module by at most $\rho$ elements. Let $\ul S$ be the 
ultra-regular local ring given as the ultraproduct of the $\seq Sw$. Another application of
Proposition~\ref{P:ullen} yields that   $\ul R$ is generated by at most $\rho$ elements over $\ul S$.
\end{proof}

\begin{example}\label{E:ulNN}
The equi\ch\ condition is necessary as the following example shows. Fix a
prime number $p$ and an indeterminate $\xi$, and let
$\zet_p$ denote the ring of $p$-adic integers. Put $\seq
Rw:=\pol{\zet_p}\xi/(\xi^{2w+1}-p^2)\pol{\zet_p}\xi$   and
let $\ul R$ be the ultraproduct of the $\seq Rw$. Each $\seq Rw$ is a
one-dimensional complete local \CM\ domain with multiplicity (=parameter
degree) two. Hence $\ul R$ is isodimensional (indeed, the cataproduct
$\ulsep R\iso \pow {(\ulsep {\zet_p}/p^2\ulsep{\zet_p})}\xi$ is also one-dimensional, 
where   $\ulsep{\zet_p}$ is  the
catapower of $\zet_p$). 

Suppose there is an ultra-regular subring $\ul S\sub \ul R$ such that $\ul R$
is   generated as an $\ul S$-module by $N$ elements. Hence by \los, there  is a
regular subring $\seq Sw\sub \seq Rw$, such
that $\seq Rw$ is generated as an $\seq Sw$-module by $N$ elements, for almost all $w$. This,
however, contradicts \cite[Proposition 3.5 and Example 3.2]{SchABCM}, where it
is shown that the least number of generators for any regular subring of $\seq 
Rw$ must
be equal to the length of $\seq Rw/p\seq Rw$ (the so-called equi-parameter
degree of $\seq Rw$), that is to say, must be at least $2w+1$.

By varying the prime $p$ as well (say, by letting $p_w$ be an enumeration of
all prime numbers and replacing $p$ by $p_w$ in the definition of $\seq Rw$),
we can construct a similar counterexample
$\ul R$ which itself is equi\ch\ zero. This latter ring also shows the extent to
which the Ax-Kochen-Ershov theorem (\cite{AK,Ers65,Ers66}) holds. Namely,
let $\ul V$ be the ultraproduct of the $\zet_{p_w}$, so that $\ul V$ is in
particular ultra-regular.
By Ax-Kochen-Ershov, $\ul V$ is also
 the ultraproduct of the $\pow{ \mathbb F_{p_w}}t$ where $\mathbb  F_{p_w}$ is the
$p_w$-element field and $t$ a single indeterminate. Put $R'_w:=\pow
{\mathbb  F_{p_w}}{t,\xi}/(\xi^{2w+1}-t^2)\pow
{\mathbb  F_{p_w}}{t,\xi}$ and let $\ul R'$ be their ultraproduct (so that $\ul R'$ and
$R'_w$ are the equi\ch\ analogues of $\ul R$ and $\seq Rw$). Both $\ul R$ and $\ul R'$
contain $\ul V$ as
a subring in a natural way, but neither extension is   finite. However, there
is
a second embedding of $\ul V$ into $\ul R'$ making the latter a finite extension.
Namely,   $\ul V$ is also isomorphic with the subring given as the ultraproduct of
the $\pow{\mathbb  F_{p_w}}\xi$. Under this identification,  $\ul R'$
is isomorphic to $\pol {\ul V}t/(t^2-\alpha)\pol {\ul V}t$,  where
$\alpha$ is the ultraproduct of the $\xi^{2w+1}$. In conclusion, $\ul R$ and $\ul R'$
cannot be
isomorphic (note, however, that their cataproducts are isomorphic, 
to $\pow {\ul F}{t,\xi}/t^2\pow {\ul F}{t,\xi}$, where $\ul F$ is  the ultraproduct of
the $\mathbb  F_{p_w}$).
\end{example}

\begin{remark}\label{R:ulNNmix}
Using   \cite[Proposition 3.5]{SchABCM}, we can use the same argument to show
that if $\ul R$ is an
ultraproduct of complete $d$-dimensional Noetherian local rings of mixed \ch\
and of bounded
equi-parameter degree, then $\ul R$ admits an ultra-regular local subring $\ul S$
over which it is module-finite. Recall that the \emph{equi-parameter} degree of
a Noetherian local ring $A$ of mixed \ch\ $p$ is the least possible length of
a homomorphic image $A/I$ modulo a parameter ideal $I\sub A$ containing $p$. 
\end{remark}

\begin{corollary}
If $\ul S\sub \ul R$ is a local module-finite extension of ultra-Noetherian local
rings with $\ul S$ ultra-regular, then $\ul R$ is \pCM\ \iff\ it is flat over $\ul S$.
\end{corollary}
\begin{proof}
Let $(\seq Sw,\seq{\mathfrak n}w)$ and $\seq Rw$ be Noetherian local rings with
ultraproduct equal
to $(\ul S,\ul {\mathfrak n})$ and $\ul R$ respectively. By \los, almost all $\seq
Sw\sub
\seq Rw$ are finite
extensions with $\seq Sw$   regular. Suppose first that
$\ul R$ is
\pCM, whence ultra-\CM\ by Theorem~\ref{T:ulCM}. Hence almost all $\seq
Rw$
are \CM, whence flat over $\seq Sw$. To show that $\ul R$ is flat over
$\ul S$, it suffices by \cite[Theorem 7.8(3)]{Mats} to show that $\tor {\ul
S}1{\ul R}{\ul S/\ul I}$ vanishes
for all finitely generated ideals $\ul I$ of $\ul S$. Choose $\seq Iw\sub\seq 
Sw$ whose ultraproduct equals $\ul I$. Since ultraproducts commute with homology,
  $\tor
{\ul S}1{\ul R}{\ul S/\ul I}$  is the ultraproduct of the $\tor {\seq Sw}1{\seq
Rw}{\seq Sw/\seq Iw}$. Since the latter  are zero by flatness, so is the former.

Conversely, suppose $\ul S\sub \ul R$ is flat. In particular, $\ul R$
isodimensional by (the proof of) Theorem~\ref{T:ulNN}. 
By the same argument as above,
the vanishing of $\tor {\ul S}1{\ul R}{\ul S/\ul {\mathfrak n}}$ implies the
vanishing of almost all $\tor {\seq
Sw}1{\seq Rw}{\seq Sw/\seq{\mathfrak n}w}$. By the local flatness criterion, this implies that almost all $\seq Rw$
are flat over $\seq Sw$, whence are \CM. Hence $\ul R$ is ultra-\CM, and therefore
\pCM\ by Theorem~\ref{T:ulCM}.
\end{proof}

\section{Homological theorems}\label{s:INIT}

In this section, we prove for local rings of finite embedding dimension the
counterparts of the homological theorems from commutative algebra, under the
the assumption that the completion is equi\ch. We start with an immediate
corollary of the definitions:

\begin{corollary}[Monomial Theorem]\label{C:mon}
 Let $R$ be a local ring
 of   \pdim\ $d$ and let $\tuple x$ be a generic
sequence in $R$. Suppose $R$ has either equal \ch\ or otherwise is infinitely
ramified (see \S\ref{s:ram}). If
$\nu_0,\dots,\nu_s\in\nat^d$ are multi-indices
 such that ${\nu_0}$ does not belong to the semigroup  generated by
${\nu_1},\dots,{\nu_s}$, then $\tuple
x^{\nu_0}$ does not lie in the ideal in $R$ generated by $\tuple 
x^{\nu_1},\dots \tuple x^{\nu_s}$.
\end{corollary}
\begin{proof}
 If  the contrary were true, then the same ideal 
membership holds in the completion $\complet R$. However,
by Proposition~\ref{P:sop}, the image of $\tuple x$ in $\complet R$ 
is a system of parameters, thus violating the usual
Monomial Theorem (see for instance \cite{HoHT}), since $\complet R$ is equi\ch.
\end{proof}

A special instance of the assertion (which is often already referred 
to as the Monomial Theorem) is the fact that for any
generic sequence $\rij xd$ in $R$, we have
\begin{equation}\label{eq:mon}
 (x_1\cdots x_d)^t\notin (x_1^{t+1},\dots,x_d^{t+1})R
\end{equation}
 for all $t$. In the Noetherian setup, the latter 
result suffices to show the so-called Direct Summand
Theorem (see for instance \cite[Lemma 9.2.2]{BH}). However, it is not 
clear how to derive this in the present setup (presently, I can only get a
weaker version, which I omit here). 

Next we have a look at the
Hochster-Roberts theorem. Although one can formulate a more general version,
we will only give the result for local \homo{s} $R\to S$ which are \emph{locally
of finite type}, meaning that $S$ is a localization of some finitely generated
$R$-algebra. Note that the class of local rings of finite embedding dimension 
is closed under such maps: if $(R,\maxim)\to (S,\mathfrak n)$ is locally of
finite type, then so is $R/\maxim\to S/\maxim S$. In particular,  $S/\maxim S$ is
Noetherian, and   $\mathfrak n(S/\maxim S)$ is finitely generated. Hence if
$\maxim$ is finitely generated, then so is $\mathfrak n$.

\begin{theorem}[Hochster-Roberts]\label{T:wHR} 
Let  $R\to S$ be a local \homo\ between
local rings of finite embedding dimension. 
Suppose $R$ has equal \ch\ or is infinitely ramified. If
$R\to S$ is cyclically pure and     locally of finite type, and if $S$ is \creg, then $R$ is \cCM.
\end{theorem}
\begin{proof} 
It suffices to show that $\complet R\to \complet S$ is 
cyclically pure, for then the classical
Hochster-Roberts theorem shows that   $\complet R$ is \CM\ by 
\cite[Theorem 2.3]{HHbigCM2}, since $\complet S$ is
regular and equi\ch. To prove cyclical purity, we need to
show that $I=I\complet S\cap \complet R$ for each ideal $I$ in 
$\complet R$. Since any ideal is an intersection of
$\maxim\complet R$-primary ideals, it suffices to show this for $I$ 
an $\maxim\complet R$-primary ideal, where $\maxim$ is the maximal ideal of $R$.
By Lemma~\ref{L:Art}, any such ideal is extended from $R$, that is to 
say of the form $I=\id \complet R$ with $\id$ an
$\maxim$-primary ideal of $R$. Since $S/\id S$ is locally of finite type over
the Artinian local ring $R/\id$, it is Noetherian. Therefore, $\id S$ is
closed, so that 
$I\complet S\cap S=\id\complet S\cap S=\id S$ by Lemma~\ref{L:quot}. 
Hence, in the composition
\begin{equation*}
\complet R/I\iso R/\id\to S/\id S\to \complet S/I\complet S
\end{equation*}
 all maps are injective, as the first is an 
isomorphism by Lemma~\ref{L:quot} and the second is injective
  by assumption. This proves that $I\complet S\cap \complet R=I$, as required.
\end{proof}

%

\begin{remark}
Combining this result with Theorems~\ref{T:pCM}, \ref{T:psing} and \ref{T:cflatcCM}, and Corollary~\ref{C:bertin}  yields Corollary~\ref{C:GlazConj} from the introduction.
In the theorem, less than cyclical purity is required; it suffices that
$R\to S$ is \emph{pure-closed}, in the sense that $IS\cap R=I$ for every closed
(equivalently, every $\maxim$-primary) ideal $I\sub R$. Furthermore, we may
weaken the condition that $R\to S$ is locally of finite type to requiring that
its closed fiber $S/\maxim
S$ is Noetherian. In
order to apply the techniques from
\S\ref{s:mix} and deduce
an asymptotic version of the Hochster-Roberts theorem in mixed \ch, we would
like to prove a stronger result: namely, under an additional
isodimensionality assumption, may we conclude that $R$ is \pCM?
\end{remark}

To obtain other homological properties, we follow Hochster's treatment
\cite{HoHT}, by generalizing the notion of big \CM\ modules. In fact, as
in the Noetherian case, we can even put a ring structure on the latter:

\subsection*{Big \CM\ algebras} 
We call an $R$-algebra $B$ a 
\emph{big \CM\ algebra} if some
generic sequence   of $R$ is $B$-regular; we call $B$ a
\emph{balanced big \CM\ algebra} if  every generic sequence is $B$-regular.

\begin{theorem}\label{T:BCM}
 Let $R$ be a local ring of finite embedding dimension. If $R$ has equal \ch\ or
is infinitely unramified, then it admits a
 balanced big \CM\ algebra.
\end{theorem}
\begin{proof} By the work of Hochster and Huneke 
(\cite{HHbigCM,HHbigCM2}) or the more canonical construction of
\cite[\S7]{SchAsc} (note that   the algebras in the latter paper are 
local), any equi\ch\ Noetherian local ring admits a
balanced big \CM\ algebra. This applies in particular to the completion
$\complet  R$ as it is always equi\ch\ by the discussion in \S\ref{s:ram}. So
remains to show that any balanced big
\CM\
$\complet R$-algebra $B$ is a balanced big \CM\ $R$-algebra. 
However, this is clear for if $\tuple x$ is a generic
sequence, then it is a system of parameters in $\complet R$ by 
Proposition~\ref{P:sop}, whence $B$-regular.
\end{proof}

\begin{remark}
We may drop the requirement on the \ch\ when $R$ has \pdim\ at most three, since in that case, regardless of \ch, $\complet R$ admits a balanced big \CM\ algebra by \cite{HoBCM3}. In particular, all the homological theorems below also hold under this assumption.
\end{remark}

\begin{remark}\label{R:BCMfunc} 
In fact, we may choose   balanced 
big \CM\ algebras in a weakly functorial way in the
following sense. We will call a local \homo\ $R\to S$ \emph{\c-permissible}, if
$\complet R\to \complet S$ is permissible
in the sense of \cite[\S9]{HuTC} or \cite[\S7.9]{SchAsc}. In that 
case, we may choose a balanced big \CM\ $\complet
R$-algebra $B$ (whence a balanced big \CM\ $R$-algebra), a 
balanced big \CM\ $\complet S$-algebra $B'$ (whence a
balanced big \CM\ $S$-algebra) and a \homo\ $B\to B'$ extending 
$\complet R\to \complet S$, whence also extending
$R\to S$. Recall from the cited sources that any local algebra is 
permissible over an equidimensional and universally
catenary Noetherian local ring (e.g., a complete local domain).
\end{remark}

\begin{proposition}\label{P:cycpur} 
If $R$ is \preg\ with residue field $k$ and if $B$ is a balanced
big \CM\ $R$-algebra, then all $\tor RiBk$ vanish
for $i>0$,  and $IB\cap R$ is equal
to the  closure of $I$ for each 
ideal $I\sub R$.
\end{proposition}
\begin{proof} 
It is not hard to verify that $B\tensor_RS$ is a balanced big \CM\
$S$-algebra, for $S:=R(\xi)$ and $\xi$ a tuple of indeterminates. Since $R\to S$
is faithfully flat, we may pass from $R$ to $S$ and therefore assume in view
of Remark~\ref{R:NC} that the maximal ideal of $R$ is generated by  a regular   sequence
$\tuple x$. Since  $\tuple x$ is also
$B$-regular and $k=R/\tuple xR$,
\begin{equation*}
\tor RiBk\iso\tor {R/\tuple xR}i{B/\tuple xB}k=0
\end{equation*} 
for all $i>0$. Therefore, for  any $\maxim$-primary ideal $\mathfrak n$,
we get $\tor {R/\mathfrak n}1{B/\mathfrak nB}k=0$ by \cite[Lemma
2.1]{SchBetti}.  Since $R/\mathfrak n$ is Artinian, $B/\mathfrak
nB$ is faithfully flat by the Local Flatness Criterion, and hence in
particular $\mathfrak n=\mathfrak nB\cap R$. The last assertion then follows
since any closed ideal is the intersection of all $\maxim$-primary ideals
containing it.
\end{proof}

Using an argument similar to the one in the proof of Corollary~\ref{C:tor}, one can show that under the above assumptions, each $\tor RiBM$ vanishes, for $i>0$ and $M$ a finitely generated \pCM\ module: for the Artinian case,   induct on the length of $M$, and for the general case, on the depth of $M$, using that $\inf RB=0$ by construction; details are left for the reader.
Before stating the next result, we need to introduce some 
terminology. We will follow the treatment in
\cite[\S9.4]{BH} and refer to this source for more details. Let 
$F_\bullet$ be a complex
\begin{equation*} 0\to F_s\map{\varphi_s} F_{s-1}\map{\varphi_{s-1}} 
\dots\map{\varphi_2} F_1\map{\varphi_1} F_0\to 0
\end{equation*} with each $F_i$ a finitely generated free $R$-module. 
We call $s$ the length of $F_\bullet$ and we call
the cokernel of $\varphi_1$   the \emph{cokernel} of the complex. For 
each $1\leq n\leq s$, we will define the $n$-th
\emph{Fitting} ideal $I_n(F_\bullet)$ of $F_\bullet$ as follows. Fix 
$1\leq n\leq s$ and put
\begin{equation*} 
r:=\sum_{i=n}^s (-1)^{i-n}\op{rank}F_i.
\end{equation*} 
Let $\Gamma$ be a matrix representing the morphism 
$\varphi_n$ (by choosing bases for $F_n$ and
$F_{n-1}$) and let $I_n(F_\bullet)$ be the ideal in $R$ generated by 
all $r\times r$-minors of $\Gamma$. One shows that
this is independent from the choices made.

We say that $F_\bullet$ is \emph{acyclic} if all $H_i(F_\bullet)$ 
vanish, for $i>0$; if also $H_0(F_\bullet)$ vanishes
(that is to say, if the cokernel of $F_\bullet$ is zero), then we say 
that $F_\bullet$ is \emph{exact}. In particular,  if
$F_\bullet$ is acyclic,  then it is a finite free resolution of its cokernel.

\begin{theorem}\label{T:INIT} 
Let $(R,\maxim)$ be an equi\ch\ or an infinitely ramified local 
ring of finite embedding dimension.  Let $F_\bullet$
be a finite complex of finitely generated free $R$-modules of length $s$ and let $M$ 
be its cokernel. If the \cheight\ of $I_n(F_\bullet)$ is at least $n$
for each $n=\range 1s$, then the \cheight\ of $\ann R\mu$  is at most $s$, for any
non-zero minimal generator $\mu$ of $M$.
\end{theorem}
\begin{proof} 
Let $d$ and $e$ be the \pdim\ of $R$ and $R/\ann{R}\mu$ respectively. In view
of Proposition~\ref{P:ght}, we need to show that $d-e\leq s$, and we do this
by induction on $e$. Assume first that $e=0$, so
that $\ann R\mu$ is $\maxim$-primary. By Theorem~\ref{T:BCM}, there exists  a
balanced big 
\CM\ $R$-algebra $B$. By  Proposition~\ref{P:ght}, we can find part of a generic
sequence
of length $n$ in $I_n(F_\bullet)$, which is therefore $B$-regular. Hence
each $I_n(F_\bullet)B$ has grade at least $n$, and the
Buchsbaum-Eisenbud Acyclicity criterion (\cite[Theorem
9.1.6]{BH}) then yields that the complex $F_\bullet\tensor_RB$ is acyclic.   
Since $B$, whence each module
in $F_\bullet\tensor_R B$, has depth $d$, and since  $M\tensor_R B$ is the 
cokernel of $F_\bullet\tensor_R B$, the depth of $M\tensor_RB$ is at 
least $d-s$ by \cite[Proposition 9.1.2(e)]{BH}. 
By Nakayama's lemma, the image of $\mu$ in $M/\maxim M$ is non-zero, which
implies that $\mu\tensor 1$ is   non-zero in $M\tensor_R B$. Since the
annihilator of $\mu\tensor 1$ contains $\ann R\mu$, it is  $\maxim$-primary.
It follows that $M\tensor_R B$ has depth zero, and hence that $d-s\leq 0$.

Assume now that $e>0$. The threshold primes of $R$ and $\ann R\mu$ are all
different
from $\maxim$, and so are the threshold primes of those $I_n(F_\bullet)$ that
are not $\maxim$-primary. 
By prime avoidance, we may therefore choose $x\in\maxim$ outside all these
finitely many threshold primes.  Put $R_n:=R/I_n(F_\bullet)$ and  $S:=R/xR$.
We want to apply the induction
hypothesis to the complex $F_\bullet\tensor_RS$   and the image of
$\mu$ in $M\tensor_RS$. By Corollary~\ref{C:tre}, the \pdim\ of  $S$ and
$R_0\tensor_RS$ are $d-1$ and $e-1$ respectively, and the \pdim\ of $S/I_n(F_\bullet\tensor_RS)\iso
R_n\tensor_RS$ is at most $d-n-1$ (this is trivially true if $I_n(F_\bullet)$  
is  $\maxim$-primary and follows from Lemma~\ref{L:gen} in the remaining
case).  Since $S/\ann S\mu$ is a residue ring of
$R_0\tensor_RS$, its \pdim\ is at most $e-1$, so that our induction hypothesis
applies, yielding $d-1-(e-1)\leq \cht{\ann S\mu}\leq s$.
%
%
%
\end{proof}

We can now generalize  the new intersection theorems due to 
Evans-Griffith and Peskine-Szpiro-Roberts.

\begin{corollary}\label{C:INIT} 
Let $(R,\maxim)$ be an equi\ch\ or an infinitely
ramified local 
ring of finite embedding dimension. Let $F_\bullet$ be a
finite complex of finitely generated free $R$-modules of length $s$ 
and let $M$ be its cokernel.
\begin{enumerate}
\item\label{i:EG} If $F_\bullet$ is acyclic when localized at any 
closed prime ideal of $R$ different from $\maxim$ and there
exists a non-zero minimal generator of $M$ whose annihilator is 
$\maxim$-primary, then $\pd R\leq s$.
\item\label{i:PS}  If $F_\bullet$ is exact when localized at any 
closed prime ideal of $R$ different from $\maxim$ and $s<\pd R$,
then $F_\bullet$ is exact.
\end{enumerate}
\end{corollary}
\begin{proof} 
To prove \eqref{i:EG}, assume  $s<d:=\pd R$. We reach the 
desired contradiction from Theorem~\ref{T:INIT}, if we
can show that $R/I_n(F_\bullet)$ has \pdim\ at most $d-n$, for all 
$n=\range 1s$. Fix $n$ and let $I:=I_n(F_\bullet)$.
There is nothing to show if $I$ is $\maxim$-primary, so that we may 
exclude this case. By Remark~\ref{R:qnid}, we can
choose a \c-normalization $A_0\to R$ and an ideal $J\sub A_0$ such
 that $JR=I$ (note that $I$ is finitely generated by
construction). Let $A$ be the image of $A_0$ in $R$, so that $A\sub R$ is 
also \c-integral and \c-injective (although Noetherian, $A$ will,  in general,   no
longer be regular). Since $\complet A\to
\complet R$ is integral and injective, $\complet A$,
 whence also $A$, has dimension $d$. Suppose $JA$ has
height $h$ and let $\mathfrak q$ be a minimal prime of $JA$ of height 
$h$. By \cite[Theorem 9.3]{Mats}, we can find a prime ideal $\mathfrak P$ in
$\complet R$ lying above $\mathfrak q$. Let $\pr:=\mathfrak P\cap R$ (which is therefore closed by
Corollary~\ref{C:spec}).
Note that since $I$ is not
 $\maxim$-primary, $h<d$, and therefore
$\pr\neq\maxim$. By assumption, $(F_\bullet)_\pr$ is acyclic, so that 
the grade of $IR_\pr$ is at least $n$ by the
Buchsbaum-Eisenbud Acyclicity criterion (\cite[Theorem 9.1.6]{BH}). 
By \cite[Proposition 9.1.2(g)]{BH}, the grade of
$JA_{\mathfrak q}$ is therefore also at least $n$. In particular, 
$A_{\mathfrak q}$ has depth at least $n$, showing that
$n\leq h$.  This in turn implies that $A/JA$ has dimension at most 
$d-n$. Since $\complet A/J\complet A\to \complet
R/I\complet R$ is integral, the dimension of the first ring is at 
least that of the second ring. Hence we showed that
$\complet R/I\complet R$ has dimension at most $d-n$. By 
Lemma~\ref{L:quot} and Theorem~\ref{T:pdim}, this in turn
implies that $R/I$ has \pdim\ at most $d-n$, as required.

The second assertion follows from the first by a standard argument 
(see for instance the proof of \cite[Corollary
9.4.3]{BH}). Namely, it implies that the cokernel $M$ of $F_\bullet$ 
has finite length. The only way that this does not
contradict \eqref{i:EG} is that $M=0$ (by Nakayama's Lemma). This in 
turn implies that we can split of the last term in
$F_\bullet$ and then an inductive argument on $s$ finishes the proof.
\end{proof}

We can translate these results to more familiar versions  of the
homological   theorems.

\begin{theorem}[Superheight]\label{T:supht} 
Let $R\to S$ be a local \homo\ of  equi\ch\ or   infinitely ramified local
rings of finite embedding dimension  and let $M$ be an $R$-module admitting a finite free
resolution $F_\bullet$ of length $s$. If  $M\tensor_RS$ has finite 
length, then $S$ has \pdim\ at most $s$.
\end{theorem}
\begin{proof} 
Let $\maxim$ and $\mathfrak n$ be the respective maximal ideals of $R$ and $S$.
Let $\mathfrak q$ be an ideal in $S$ different from $\mathfrak n$ and put 
$\pr:=\mathfrak q\cap R$. Since the localization of $M\tensor_R
S$ at $\mathfrak q$ is zero, we get   
$$M_\pr/\pr M_\pr\tensor_{k(\pr)} 
S_{\mathfrak q}/\pr S_{\mathfrak q}=0,
$$
 where $k(\pr)$
is the residue field of $\pr$. Since $S_{\mathfrak q}/\pr S_{\mathfrak q}$ is 
non-zero, $M_\pr/\pr M_\pr$ must be zero, and
therefore $M_\pr=0$, by Nakayama's Lemma. Hence $(F_\bullet)_\pr$ is 
exact whence split exact. Therefore, this remains
so after tensoring with $S_{\mathfrak q}$. In other words, the 
conditions of \eqref{i:PS} are met for the complex
$F_\bullet\tensor_RS$ over the ring $S$,  showing that $S$ must have 
\pdim\ at most $s$.
\end{proof}

\begin{theorem}[Intersection Theorem] 
Let $R$ be an equi\ch\ or an infinitely
ramified local 
ring of finite embedding dimension  and let $M,N$ be $R$-modules.
If $M$ has a finite free resolution of length $s$, then $\pd N\leq 
s+\pd{M\tensor_RN}$.
\end{theorem}
\begin{proof} 
Assume first that  $M\tensor_RN$ has finite length and 
let   $S:=R/\ann RN$. It follows that $M\tensor_RS$
has finite length, so that the \pdim\ of $S$ is at most $s$ by Theorem~\ref{T:supht}. For the general
case, we induct on the \pdim\   of $M\tensor N$. Using 
Proposition~\ref{P:sop}, one can find $x\in R$ such that it is
part of a generic subset of both $R/\ann RN$ and $R/\ann R{M\tensor 
N}$. It follows that the \pdim{s} of $N/xN$ and
$M\tensor N/xN$ both have dropped by one, so that we are done by induction.
\end{proof}

\begin{theorem}[Canonical Element Theorem]\label{T:CE}
Let $(R,\maxim)$ be an equi\ch\ or an infinitely ramified local 
ring of finite embedding dimension. Let $F_\bullet$ be a  free
resolution  of the residue field $k$ of $R$ and let $\tuple
x$ be a generic sequence in $R$. If $\gamma$ is a complex morphism
from the Koszul complex $K_\bullet(\tuple x)$ to $F_\bullet$, extending the
natural map $\gamma_0\colon K_0(\tuple x)=R/\tuple xR\to k$, then the morphism
$\gamma_d\colon K_d(\tuple x)\to F_d$
is non-zero, where $d$ is the \pdim\ of $R$.
\end{theorem}
\begin{proof}
Suppose $\gamma_d$ is zero. Let $B$ be a local balanced big \CM\
algebra for $R$ and let $y\in B$ be such that its image in $B/\tuple
xB$ is a non-zero socle element. Define $\psi_0\colon  R \to B$ by sending $1$ to  $y$. Since
$\tuple x$ is $B$-regular, the
Koszul
complex $K_\bullet(\tuple x;B):=K_\bullet(\tuple x)\tensor B$ is acyclic. It
follows that 
$\psi_0$ extends to  a morphism of complexes $\psi\colon F_\bullet\to
K_\bullet(\tuple
x;B)$. Let $\alpha:=\psi\after\gamma $ be the composition $K_\bullet(\tuple x)
\to K_\bullet(\tuple x;B)$. In particular $\alpha_0(1)=y$  and $\alpha_d=0$. On
the other hand, $\alpha_0$ induces by
tensoring a morphism of complexes $\beta :=1\tensor\alpha_0\colon
K_\bullet(\tuple x)
\to K_\bullet(\tuple x)\tensor B=K_\bullet(\tuple x;B)$. Since
$K_\bullet(\tuple x;B)$ is acyclic,   $\alpha$ and $\beta$
differ by a homotopy $\sigma$. In particular,
$\beta_d=\beta_d-\alpha_d=\sigma_{d-1}\after \delta_d$, where $\delta_d\colon
K_d(\tuple x)=R\to K_{d-1}(\tuple x)=R^d$ is the left most  map in the Koszul
complex. Since the image of $\delta_d$ lies in $\tuple x R^d$, we get
$y=\beta_d(1)=\sigma_{d-1}\after \delta_d(1)\in\tuple x B$, contradicting our
choice of $y$.
\end{proof}

To formulate the next result, which extends a result of Eisenbud and Evans in
\cite{EE},   recall
that for an $R$-module $M$ and an element $z\in M$, the \emph{order ideal} of $z$
is the ideal $\loc_M(z)$ consisting of all images $\alpha(z)$ for $\alpha\in\hom
RMR$. Moreover, if $R$ is a domain  with field of fractions  $K$, then the \emph{rank} of $M$  is defined as the dimension of the vector space $M\tensor_R K$.

\begin{theorem}[Generalized Principal Ideal Theorem]\label{T:GPIT}
Let $(R,\maxim)$ be an equi\ch\ or an infinitely ramified local 
domain of finite embedding dimension,  and let $M$ be a
finitely generated
$R$-module. If $z\in\maxim M$, then the \cheight\ of
$\loc_M(z)$ is at most the rank of $M$.
\end{theorem}
\begin{proof}
Let $h$ be the \cheight\ of $\loc_M(z)$, let $r$ be the rank of $M$,  and let $d$ be the \pdim\ of $R$. By
definition, there exists a generic sequence $\rij xd$ with $x_i\in \loc_M(z)$,
for $i=\range 1h$. Replacing $M$ by $M\oplus R^{d-h}$ and $z$ by the element
$(z,x_{h+1},\dots,x_d)$, so that both the rank of $M$ and the \cheight\ of $\loc_M(z)$ increase by $d-h$, we may assume that $\loc_M(z)$ contains a generic
sequence $\tuple x$. Let $\tuple y$ be a
finite tuple generating $\maxim$.  
As explained in the proof of \cite[Theorem 9.3.2]{BH}, the canonical \homo\ $R/\tuple xR\to
R/\tuple yR$ induces a morphism of Koszul complexes $\alpha\colon
K_\bullet(\tuple x)\to
K_\bullet(\tuple y)$. Let $F_\bullet$ be a free resolution of the residue field
$R/\tuple yR$ of $R$ and $\beta\colon K_\bullet(\tuple
y)\to F_\bullet$ be an induced morphism of complexes.
By Theorem~\ref{T:CE}, applied to
the composition $\beta\after\alpha$, we get in degree $d$ a non-zero morphism
$\beta_d\after\alpha_d$, showing in particular that $\alpha_d$ is non-zero as
well. Since $\alpha_d$ is just the $d$-th exterior power of $\alpha_1$, the rank
of $\alpha_1$ is at least $d$. On the other hand, $\alpha_1$ factors
by construction through $\hom RMR$, whence has rank at most $r$, yielding the
desired inequality $d\leq r$ (see \cite[Theorem 9.3.2]{BH} for more details).
\end{proof}

\section{Uniform bounds on Betti numbers}\label{s:ub}

In the next two sections, we apply the previous theory to derive uniformity
results for
Noetherian local rings. In this section, we study  Betti numbers.
Recall
that given a module $M$ over a local ring $R$ with residue field $k$, its
\emph{$n$-th Betti number} $\beta_n(M)$ is defined as the vector space dimension
of $\tor RnMk\iso\ext RnMk$. It is equal to the rank of the $n$-th   module 
in a minimal free resolution of $M$ (provided such a resolution exists), and by
Nakayama's Lemma, it is then also equal to the least number of generators of an
$n$-th syzygy of $M$. 
One usually studies the behavior of these Betti
numbers for a fixed module as $n$ goes to infinity. In contrast, we will study
their behavior for fixed $n$ as we vary the module.

\begin{theorem}\label{T:unifbetti}
For each quadruple $(d,e,l,n)$ of non-negative integers, there exists a bound
$\Delta(d,e,l,n)$ with the following property. If $R$ is a $d$-dimensional local
\CM\ ring of multiplicity $e$, and $M$ is a \CM\ $R$-module of multiplicity at
most $l$, then 
$$
\beta_n(M)\leq \Delta(d,e,l,n).
$$
\end{theorem}
\begin{proof}
Suppose not, so that for some quadruple $(d,e,l,n)$, we cannot define such an
upper bound. This means that for every $w$, we can find a $d$-dimensional \CM\
local ring $\seq Rw$ of multiplicity $e$, and a \CM\ $\seq Rw$-module $\seq Mw$
of multiplicity at most $l$, such that $\beta_n(\seq Mw)\geq w$. By
Theorem~\ref{T:isodim}, the ultraproduct $\ul R$ is isodimensional, and by
Lemma~\ref{L:pCMmod}, the ultraproduct $\ul M$ is   finitely generated and \pCM. 
Since the cataproduct $\ulsep M$ is therefore finitely
generated over the (Noetherian) cataproduct $\ulsep R$, its $n$-th Betti
number $\beta_n(\ulsep M)$ is finite, and by
Proposition~\ref{P:uCMfpt},  equal to almost all $\beta_n(\seq Mw)$,
contradiction.
\end{proof}

Theorem~\ref{T:unifbetti} applied to the residue field of $R$ yields
Corollary~\ref{C:Poincare} from the introduction. We can also reformulate the
previous theorem  in terms of  universal resolutions:

\begin{corollary}\label{C:univres}
For each triple $(d,e,l)$, there exists a countably generated $\zet$-algebra $Z$
and a complex $\mathcal F_\bullet$ of finite free $Z$-modules,  with
the following property. If $R$ is a $d$-dimensional local \CM\ ring of
multiplicity $e$, and $M$   a finitely generated \CM\ module of multiplicity at
most $l$, then there exists a \homo\ $Z\to R$, such that for any $n$ and any
$R$-module $N$, we have
$$
\tor RnMN\iso H_n(\mathcal F_\bullet\tensor_ZN) \qquad\text{and}\qquad \ext
RnMN\iso H^n(\hom Z{\mathcal F_\bullet}N).
$$

If we impose furthermore that $R$ is regular (whence $e=1$) or, more generally,
that $M$ has finite projective dimension, then we may take $Z$ to be a finitely
generated $\zet$-algebra and $\mathcal F_\bullet$ a complex of length $d$.
\end{corollary}
\begin{proof}
For each $n$, let $\delta_n:=\Delta(d,e,l,n)$ be the bound given by
Theorem~\ref{T:unifbetti}, and let
$\tuple \Xi_n$ be a tuple of indeterminates   viewed as a
$\delta_{n-1}\times\delta_n$-matrix. Let $Z$ be the
polynomial ring over $\zet$ generated by all indeterminates $\tuple \Xi_n$
modulo the relations $\tuple \Xi_n\cdot\tuple \Xi_{n+1}=0$, expressing that the
product of two consecutive matrices is zero.  We then define the complex  $\mathcal F_\bullet$
by letting its $n$-th term  be $Z^{\delta_n}$, and its $n$-th  differential
  the matrix $\tuple \Xi_n$. By construction, $\mathcal F_\bullet$ is a free
complex. Now, given  $R$ and $M$ as in the statement,
Theorem~\ref{T:unifbetti} implies that we may assign to each entry in
$\tuple \Xi_n$, a value in $R$ so that under the induced map $Z\to R$, the complex $\mathcal
F_\bullet\tensor_ZR$ becomes a free resolution of $M$. The statement now follows
from the definition of Tor and Ext.
\end{proof}

The $n$-th \emph{Bass number} $\mu_n(M)$ of a finitely generated $R$-module $M$ is the vector space dimension of $\ext RnkM$, where $k$ is the residue field of $R$. The $q$-th Bass number, with $q$ equal to the depth of $M$, is also called the \emph{type} of $M$.

\begin{corollary}\label{C:unifbass}
The type (respectively, for each $n$, the $n$-th Bass number) of a finitely generated module $M$ over a local \CM\ ring $R$ is bounded above by a function (in $n$) depending only on the dimension and multiplicity of $R$, and on the minimal number of generators of $M$.\end{corollary}
\begin{proof}
Since the depth of $M$ is at most the dimension of $R$, it suffices to prove the claim for a fixed $n$. By Corollary~\ref{C:univres}, there is a resolution $F_\bullet$ of     $k$ by finite free $R$-modules $F_n$ whose ranks $\beta_n(k)$ are bounded by the dimension and multiplicity of $R$.  Since $\ext RnkM$ is the $n$-th cohomology of $\hom R{F_\bullet}M$, its length $\mu_n(M)$ is at most the number of generators of $\hom R{F_n}M\iso M^{\beta_n(k)}$, and the claim follows.
\end{proof}

%
%
%
%

Let us extend some    definitions from \cite{SchClassSing}. We will call a \homo\  $R\to S$ of Noetherian local rings \emph{formally etale} (or a \emph{scalar extension}), if it is faithfully flat and unramified (=the maximal ideal of $R$ extends to the maximal ideal of $S$).   Let $(R,\maxim)$ and $(S,\mathfrak n)$  be Noetherian local rings, and let $M$ be a finitely generated $R$-module and $N$ a finitely generated $S$-module.
We define the
\emph{infinitesimal neighborhood distance}
between   $M$ and $N$ as the real
number 
$$
 d(M,N):=e^{-\alpha}
$$
where $\alpha$ is the (possibly infinite)  supremum of all $j$ such that there exists an Artinian local ring $T$, together with formally etale extensions $R/\maxim^j\to T$ and $S/\mathfrak n^j\to T$, yielding $M\tensor_RT\iso N\tensor_ST$.   As shown in \cite{SchClassSing} (where the distance is only defined between rings), the
infinitesimal neighborhood distance is a (quasi-)metric, and, roughly speaking,  up to a formally etale
base change,  limits in this metric space can
be calculated by means of cataproducts.  

\begin{theorem}\label{T:deformbetti}
For each quadruple of positive integers $(d,e,l,n)$, there exists a
bound $\varepsilon:=\varepsilon(d,e,l,n)>0$
such that if $R$ and $S$ are $d$-dimensional local \CM\ rings of multiplicity $e$, and $M$ and $N$ are finitely generated \CM\ modules of multiplicity at most $l$ over $R$ and $S$ respectively, with $
d(M,N)\leq \varepsilon$, then $\beta_n(N)=\beta_n(M)$.
\end{theorem}
\begin{proof}
Suppose no such bound exists for the pair $(d,e,l,n)$, resulting in a
counterexample  for each $w$, given by  $d$-dimensional \CM\ local rings $(\seq Rw,\seq\maxim w)$ and $(\seq Sw,\seq{\mathfrak n}w)$ of multiplicity $e$, and finitely generated \CM\ modules $\seq Mw$ and $\seq Nw$ 
 of multiplicity at most $e$ over $\seq Rw$ and $\seq Sw$ respectively, such that $ d(\seq Mw,\seq Nw)\leq e^{-w}$, but
$\beta_n(\seq Mw)\neq\beta_n(\seq Nw)$. Since Betti numbers are preserved under formally etale extensions, the techniques in \cite{SchClassSing} allows us to reduce to the case that the distance condition means that 
\begin{equation}\label{eq:distmod}
\seq Rw/\seq\maxim w^w\iso \seq Sw/\seq{\mathfrak n}w^w\quad\text{and}\quad \seq Mw/\seq\maxim w^w\seq Mw\iso \seq Nw/\seq{\mathfrak n}w^w\seq Nw
\end{equation}
Let $\ul M$ and $\ul N$ be the respective ultraproducts of the $\seq Rw$, $\seq Sw$, $\seq Mw$, and $\seq Nw$.
By Corollary~\ref{C:CMulsep}, the 
respective ultraproducts $\ul R$ and $\ul S$ are \pCM\ local rings, and by Lemma~\ref{L:pCMmod}, the respective ultraproducts $\ul M$ and $\ul N$ are finitely generated \pCM\ modules over $\ul R$ and $\ul S$ respectively.
 Moreover, by \los\ and modding out infinitesimals, we get from \eqref{eq:distmod} that the respective cataproducts $\ulsep R$ and $\ulsep S$ are isomorphic, as are the respective cataproducts $\ulsep M$ and $\ulsep N$. By Proposition~\ref{P:uCMfpt}, we therefore get for almost
all $w$, the following contradictory equalities
$$
\beta_n(\seq Mw)=\beta_n(\ulsep M)=\beta_n(\ulsep N)=\beta_n(\seq Nw).
$$
\end{proof}

\begin{corollary}\label{C:maxCMbetti}
 Given a local \CM\ ring $R$,  there exists, for each  $n\in\nat$, a
bound $\delta:=\delta(n)>0$
such that if $M$ and $N$  are maximal \CM\ modules with 
$d(M,N)\leq \varepsilon$, then $\beta_n(N)=\beta_n(M)$.
\end{corollary}
\begin{proof}
If $d(M,N)<1$, then $M$ and $N$ have the same minimal number of generators $m$. In view of Theorem~\ref{T:deformbetti}, it suffices to show that the multiplicity of $M$ and $N$ are uniformly bounded in terms of $m$. 
Let $\tilde e$ and $\tilde q$ be respectively the maximum multiplicity of $R/\pr$ and the maximal length of $R_\pr$, where $\pr$ runs over the finitely many $d$-dimensional prime ideals of $R$. Since we have a surjective map $R^m\to M$, tensoring with one of these $d$-dimensional prime  ideals $\pr$ shows that the length of $M_\pr$ is at most $m\tilde q$. The bound on the multiplicity  now follows from   \cite[Corollary 4.6.8]{BH}. 
\end{proof}

\subsection*{Proofs of Corollaries~\ref{C:BT} and \ref{C:invertible}}
The first corollary follows immediately from the definitions and Theorem~\ref{T:unifbetti}. To prove the second,
 let $e$ be the multiplicity of $R/I$. Since $I=xR$ for some regular element
$x\in R$, the residue ring $R/I$ is \CM\ and has projective dimension one.
Hence $\beta_1(R/I)=1$ and $\beta_2(R/I)=0$. Choose some $\varepsilon>0$ as given by
Theorem~\ref{T:deformbetti} such that $d(R/I,M)\leq \varepsilon$ implies that
$R/I$ and $M$ have the same zero-th, first and second Betti number, for $M$ a \CM\
module of multiplicity at most $e$. Note that from
$\beta_0(M)=\beta_0(R/I)=1$ it   follows that $M$ is   of the
form $R/J$, so that in the statement, we did not even need to assume that $M$ was cyclic. Choose $a$ such that $e^{-a}\leq\varepsilon$. In particular,   $d(R/I,R/J)\leq
\varepsilon$, and therefore $\beta_1(R/J)=1$, yielding that $J$ is cyclic, and
$\beta_2(R/J)=0$, yielding that it is invertible.
\qed

In terms of the \emph{Poincare series} of a
module $M$, defined as $P_R(M;t):=\sum_n\beta_n(M)t^n$, our results yield:

\begin{corollary}
Over a fixed local \CM\   ring, the Poincare series is a continuous map from the metric
space of
\CM\
modules
 of multiplicity at most $e$ (respectively,   from the space of all maximal \CM\ modules), to $\pow\zet t$ with   its $t$-adic
topology.
\end{corollary}
\begin{proof}
For each $n$, we can choose by Theorem~\ref{T:deformbetti} (respectively, by Corollary~\ref{C:maxCMbetti}), an $\varepsilon>0$
such that $d(M,N)\leq \varepsilon$ implies that the first $n$ Betti numbers of $M$
and $N$ are the same, for $M$ and $N$ \CM\ modules of multiplicity at most
$e$ (respectively, maximal \CM\ modules). Hence $P_R(M;t)\equiv P_R(N;t)\mod t^n\pow\zet t$.
\end{proof}

Although we did not formulate it here, we may even extend this result by also  varying the base ring over all local \CM\ rings of a fixed dimension and multiplicity; see \cite[\S8]{SchClassSing}. Applied to a regular local ring, we immediately get:

\begin{corollary}
Let $R$ be a regular local ring. For each $e$, there exists
$\delta:=\delta(e)>0$ such that if $M$ and $N$ are \CM\ modules of
multiplicity at most $e$ and $d(M,N)\leq\delta$, then $P_R(M;t)=P_R(N;t)$. 
\qed
\end{corollary}

\section{Uniform arithmetic}\label{s:ua}

 In this section, we prove several    uniform bounds, and show that the existence of such   bounds is often 
equivalent with a certain ring theoretic property. We
start with examining
the domain  property. It is not true in general that the catapower of a domain is again a
domain: let $R$ be the local ring at the origin of the plane curve  over a field $k$
given by $f:=\xi^2-\zeta^2-\zeta^3$. The catapower of $R$ is $\pow
{\ul k}{\xi,\zeta}/f\pow {\ul k}{\xi,\zeta}$, where $\ul k$ is the ultrapower of
$k$, and this  is not a domain (since $1+\zeta$ has a square root in $\pow{\ul 
k}{\xi,\zeta}$). Clearly, the problem is that $R$ is not \emph{analytically irreducible}, that is to
say, not a \c-domain. 

Before we  give a necessary and sufficient condition for a catapower
to be a domain, let us introduce some terminology which makes for a smoother
presentation of our results. Put $\bar\nat:=\nat\cup\{\infty\}$. By an \emph{$n$-ary numerical} function, we mean
a map from $f\colon\bar\nat^n\to\bar \nat$,  
with the property that $f\rij sn=\infty$ \iff\ one of the entries $s_i$ is equal
to $\infty$. Moreover, we will always assume that a numerical function $f$ is
non-decreasing in any of its arguments, that is to say, if $s_i\leq t_i$ for
$i=\range 1n$,
then
$f\rij sn\leq f\rij tn$. To indicate that a numerical function depends on a ring $R$, we
will write the ring as a subscript.

Recall that $R$ has
\emph{bounded multiplication} if
there exists a binary numerical function $\mu_R$ (called a
\emph{uniformity function}) such that 
\begin{equation*}
\ord{}{xy}\leq \mu_R(\ord{} x,\ord{} y)
\end{equation*}  
for all $x,y\in R$ (see \S\ref{s:adic} for the definition of order).  In view
of our  definition of numerical function,  the ideal of infinitesimals in a
local ring with bounded multiplication is a prime ideal, and hence the separated
quotient is a domain.

\begin{theorem}\label{T:domCh} 
Let $(R,\maxim)$ be a Noetherian local ring. The following are equivalent:
\begin{enumerate}
\item\label{i:au} $R$ is analytically irreducible;
\item\label{i:bm} $R$ has bounded multiplication;
\item\label{i:ulsepdom} some (equivalently, any)  catapower $\ulsep
R$ of $R$ is a domain.
\end{enumerate}
\end{theorem}
\begin{proof} 
The implication \implication {i:bm}{i:au} is clear from
the above discussion, since having bounded
multiplication is easily seen to be preserved under completions. 
In order to
prove \implication{i:au}{i:ulsepdom},
assume $R$ is analytically irreducible and let   $\ulsep R$ be its
catapower. Since $\complet R$ has the same catapower by
Corollary~\ref{C:compulsep}, we may moreover assume that $R$ is a complete
Noetherian local ring.
If
$R$ is normal, then so is $\ulsep R$ by Corollary~\ref{C:ulsepreg}, and hence
again a domain. For the general case, let $S$ be the normalization of $R$, so
that $R\sub S$ is a finite extension. Since $R$ is complete, $S$ is again
local. By Proposition~\ref{P:int}, we get an extension $\ulsep R\sub\ulsep S$.
Since we argued that $\ulsep S$ is a domain, the same therefore is true for
$\ulsep R$.

Remains to show \implication{i:ulsepdom}{i:bm}. By way of contradiction,
suppose no bound exists for the pair $(a,b)$. In other words, we can find
$\seq xn,\seq yn\in R$ such that $\ord{}{\seq xn}=a$, $\ord{}{\seq yn}=b$
and $\seq xn\seq yn\in\maxim^n$. Letting $\ul x$ and $\ul y$ be their
respective ultraproducts, we get $\ord{}{\ul x}=a$, $\ord{}{\ul y}=b$ and
$\ul x\ul y\in\inf{\ul R}$. Since $\inf{\ul R}$ is by assumption prime, $\ul
x$ or $\ul y$ lies in $\inf{\ul R}$, neither of which is possible.
\end{proof}

\begin{remark}\label{R:Swa}
The equivalence of \eqref{i:au} and \eqref{i:bm} is well-known and is usually
proven by a valuation argument. By
\cite[Theorem 3.4]{Swa} and
\cite[Proposition
2.2]{HueSwa} these conditions are also equivalent with
the existence of a
\emph{linear} uniformity function: 
$\mu_R(a,b):=k_R \op{max}\{a,b\}$, for some $k:=k_R\in\nat$, in which case we say
that $R$
has \emph{$k$-bounded multiplication}. For a further result along these lines,
see \cite[Proposition 5.6]{OlbSaySha}.
\end{remark}

By the same argument proving implication \implication{i:ulsepdom}{i:bm}, we
get:

\begin{corollary}\label{C:bm}
Let $\seq Rn$ be  Noetherian local rings of bounded
embedding dimension. If (almost) all $\seq Rn$ have bounded multiplication with
respect to the same uniformity function $\mu=\mu_{R_n}$, then so do their
ultraproduct
 $\ul R$ and cataproduct
$\ulsep R$. In particular, $\ulsep R$ is a domain.\qed
\end{corollary}

Note that the converse is not true. For instance, if $R$ is a complete
Noetherian local domain, then the cataproduct of the $R/\maxim^n$
is a domain by Corollaries~\ref{C:filt} and \ref{C:bm}.     If instead of order,
we use degree (see \S\ref{s:pardeg} for the definition),
we get the following analogue of bounded multiplication, this time in terms of
a bound whose dependence on the ring is only through its embedding dimension.

\begin{theorem}\label{T:multdeg}
There exists a ternary numerical  function $\omega$ with the
following property. For every   Noetherian local ring $R$ 
and any two elements $x,y\in R$, we have an inequality 
$$
\op{deg}(xy)\leq \omega(\ed R,\op{deg}(x),\op{deg}(y)).
$$
\end{theorem}
\begin{proof}
Towards a contradiction, suppose such a function cannot be defined on the
triple $(m,a,b)$. This means that for each $n$, we can find a  
Noetherian local ring $\seq Rn$ of embedding dimension   $m$ and elements
$\seq
xn,\seq yn\in\seq Rn$ such that $\op{deg}(\seq xn)=a$, $\op{deg}(\seq yn)=b$
and $\op{deg}(\seq xn\seq yn)\geq n$. Let $\ul R$, $\ul x$ and $\ul y$ be the
respective ultraproduct of the $\seq Rn$, the $\seq xn$ and the $\seq yn$. Let
$d$ be the ultra-dimension of $\ul R$, so that almost all $\seq Rn$   have
dimension $d$. By Corollary~\ref{C:degfin}, almost each $\seq Rn$ has
parameter degree at most $a$ and hence $\ul R$ is isodimensional by
Theorem~\ref{T:isodim}. Hence $\ul x$ and $\ul y$ are both generic by
Corollary~\ref{C:deggen}, and hence so is their product $\ul x\ul y$ by
Corollary~\ref{C:tre}. However, this contradicts Corollary~\ref{C:deggen} as
the $\seq xn\seq yn$ have unbounded degree.
\end{proof}

From the exact sequence
$$
 R/xR\map y R/xyR \to R/yR\to 0
$$
where the first map is induced by multiplication by $y$, we see that
$\op{deg}(xy)\leq\op{deg}(x)+\op{deg}(y)$ for all $x,y$ in a one-dimensional
Noetherian local ring (in fact, if $R$ is \CM, then the first map  is injective
and we even have equality). I do not know what happens in higher dimensions.

\subsection{Order versus degree}\label{s:orddeg}
We next investigate the relationship between order and degree. 
  If $R$ is \CM\ and $x$ is $R$-regular, then the degree of $R$ is
just the multiplicity of $R/xR$.  By \cite[Theorem 14.9]{Mats}, we get
$\ord{}x\leq \op{deg}(x)/\mult R$. In
particular,  $\ord{}x\leq\op{deg}(x)$, and this latter inequality could very well
always be true (see also \S\ref{R:reg} below). At any rate, we have:

\begin{corollary}\label{C:orddeg}
There exists a binary numerical function $\pi$ with the
following property. For every   Noetherian local ring $R$ 
and every   element $x\in R$, we have an inequality
$$
\ord{}x\leq \pi(\ed R,\op{deg}(x)).
$$
\end{corollary}
\begin{proof}
Suppose for some pair $(m,a)$, we have for each $n$, a counterexample $\seq
xn\in\seq\maxim n^n$ of degree $a$ in the Noetherian local ring $(\seq
Rn,\seq\maxim n)$ of embedding dimension   $m$. Let $\ul x\in\ul R$ be
the ultraproduct so that by Theorem~\ref{T:isodim}, the degree of $\ul x$ is
$a$, yet $\ul x\in\inf{\ul R}$,   contradicting Corollary~\ref{C:degfin}.
\end{proof}

 Applying Corollary~\ref{C:orddeg} to a product and then using
Theorem~\ref{T:multdeg},
we get the existence of a ternary numerical
function $\eta$ such that for any Noetherian local ring $R$ and elements
$x,y\in R$, we have
\begin{equation}\label{eq:bmdeg}
\ord{}{xy}\leq \eta(\ed R,\op{deg}(x),\op{deg}(y))
\end{equation}
For analytically irreducible Noetherian local rings, order and degree are
mutually bounded, and in fact, we have the following more precise result:

\begin{theorem}\label{T:bdeg}
There exists a quaternary numerical function $\zeta$ with the
following property.
For every $d$-dimensional Noetherian local domain $(R,\maxim)$ of parameter degree
at most $e$ and  $k$-bounded multiplication, and for every $x\in R$, we have 
an inequality
$$
\op{deg}(x)\leq \zeta(d,e,k,\ord{}x).
$$ 
\end{theorem}
\begin{proof}
It suffices to show that there exists a function $\beta$ such that if $\ord{}x<a$
for some $x\in R$ and some $a\in\nat$, then $\op{deg}(x)<\beta(d,e,k,a)$.
Suppose   no such bound exists for the quadruple $(d,e,k,a)$. Hence, for each
$n$, we can find a $d$-dimensional Noetherian local domain $(\seq Rn,\seq\maxim n)$
of parameter degree at most $e$ and  $k$-bounded multiplication, and an element
$\seq xn\notin\seq \maxim n^a$ whose degree is at least $n$. Let $(\ul
R,\ulmax)$ and $\ul x$ be the respective ultraproduct of the $(\seq Rn,\seq\maxim n)$ and
the $\seq xn$. By Theorem~\ref{T:isodim}, the \pdim\ of $\ul R$ is $d$. Since the
$\seq Rw/\seq xw\seq Rw$ have dimension $d-1$, but unbounded parameter degree, 
the same theorem shows that the \pdim\ of $\ul R/\ul
x\ul R$ is strictly bigger than its ultra-dimension $d-1$, whence is also equal
to $d$. 
In particular, $\ul x$ is not generic. Since
the cataproduct $\ulsep R$ is a domain by Corollary~\ref{C:bm}, we
get $\ul x\in \inf{\ul R}$ by
Corollary~\ref{C:tre}. However, by \los, $\ul x\notin\ulmax^a$, a
contradiction.
\end{proof}


Whereas order is a filtering function (see \S\ref{s:finemb}), inducing the
$\maxim$-adic filtration on $R$, this is no longer true for degree. For
instance, let $R$ be the local ring at the origin of the curve with equation
$\xi\zeta+\xi^3+\zeta^3=0$.
Then both $\xi$ and $\zeta$ have degree three, but their sum $\xi+\zeta$ has
degree
two. As we will see below in \S\ref{R:reg}, on regular local rings, degree is
filtering. Can one characterize in general  rings for which degree is
filtering? Is, for every $n$, the set of elements having degree at least $n$
always a finite union of ideals? In other words, as far as its properties are
concerned, degree
is still a mysterious function. However, its main use in this paper is to
characterize properties via its asymptotic behavior, as we will now discuss.

\subsection{Characterizations through uniform behavior}\label{s:charub}
Recall that a Noetherian
local ring is \emph{analytically unramified} if its completion is reduced. Any
reduced excellent local ring is analytically unramified (\cite[Theorem 32.2]{Mats}).

\begin{corollary}\label{C:bp}
A Noetherian local ring $R$ is analytically unramified \iff\ there exists a
numerical
function $\nu_R$, such that for every $x\in R$, we have
an inequality
$$
\ord{}{x^2}\leq \nu_R(\ord{}x).
$$ 
\end{corollary}
\begin{proof}
Since   order   remains unaffected by completion, we may assume 
that $R$ is moreover complete. Suppose $R$ is reduced. It suffices to
show that there exists a function
$\nu_R$ such that $x^2\in\maxim^{\nu_R(b)}$ implies $x\in\maxim^b$. By
way of contradiction, suppose this is false for   $b$.
Hence, we can find $\seq xn\in R$ such that $\seq xn^2\in \maxim^n$, but $\seq
xn\notin \maxim^b$. Let $\ul R$
  be the   ultrapower   of $R$
and let $\ul x$ be the ultraproduct of the $\seq xn$.   By \los,     $\ul x^2\in\inf
{\ul R}$ and $\ul x\notin\maxim^b\ul R$. However, $\inf{\ul R}$
is  radical by Corollary~\ref{C:ulsepreg}.
Hence $\ul x^2\in\inf {\ul R}$ implies   $\ul x\in\inf{\ul R}$, contradiction.

Conversely, let the function $\nu_R$ be as asserted. If $x^2$ is zero, then its  order is infinite. The only way that this can be  bounded by
$\nu_R(\ord{}{x})$, is for $x$ to have infinite order too, meaning that
$x=0$. This shows that $R$ is reduced.
\end{proof}

By a similar argument, one easily
shows that if all $\seq Rn$ have bounded squares 
(in the sense of the corollary) with respect to the same function
$\nu=\nu_{\seq Rn}$,
then their cataproduct is reduced.  If $R$ is analytically
irreducible, then the results of H\"ubl-Swanson (see Remark~\ref{R:Swa}) imply
that we may take  $\nu_R(b)$ of the form $k_Rb$ for some $k_R$ and
all $b$. I do not know whether  this is still true in general. Similarly, for
the bounds we are about to prove, is their still some vestige of linearity?

\begin{corollary}\label{C:anirr}
A Noetherian local ring $R$ is analytically irreducible \iff\ there exists a
 numerical function $\xi_R$ such that for every $x\in R$, we have
an inequality 
$$
\op{deg}(x)\leq \xi_R(\ord{}x).
$$
\end{corollary}
\begin{proof}
In  view of Remark~\ref{R:Swa}, the direct implication follows from an
application of Theorem~\ref{T:bdeg}. As for the converse, suppose degree
is bounded in terms of order. Since both order and degree remain the same after 
passing to the completion, we may moreover assume $R$ is complete. Since
a non-zero element has finite order, it has finite degree whence
is generic. This shows that there are no non-zero prime ideals of maximal
dimension, which in turn forces the zero  ideal to be a prime ideal.
\end{proof}

Tweaking \eqref{eq:bmdeg} slightly (for a fixed ring $R$), we can characterize
the following property. Recall
that a Noetherian local ring $R$ is called \emph{unmixed},  if   each
associated prime $\pr$ of its completion $\complet R$ has the same dimension as
$R$; if the above is only true for minimal primes of $\complet R$, then we say
that $R$ is \emph{quasi-unmixed} (also called \emph{formally equidimensional}).

\begin{lemma}\label{L:unm}
If a Noetherian   local ring is (quasi-)unmixed, then so is its catapower.
\end{lemma}
\begin{proof}
By Corollary~\ref{C:compulsep}, we may assume $R$ is a complete (quasi-)unmixed
Noetherian local ring.  Let us first show that the catapower
$\ulsep R$ is quasi-unmixed. In any case, $R$ and $\ulsep R$ have the same
dimension, say $d$.  Since $\ulsep R$ is complete by Lemma~\ref{L:ulcomp},  we need to show that every minimal prime $\mathfrak q\sub \ulsep R$   has dimension
$d$.  Since $R$ is complete, it is of the form $S/I$ for some complete regular local ring $S$ and some ideal
$I\sub S$. By Corollary~\ref{C:ulsepreg}, the catapower $\ulsep S$ of $S$
is regular, whence a domain. Let $\pr:=\mathfrak q\cap R$ and $\mathfrak P :=\pr\cap
S$. By 
flatness,  $\pr$ is a minimal prime of $R$ by \cite[Theorem 15.1]{Mats},
whence has dimension $d$, as $R$ is equidimensional. 

Since   $S\to \ulsep S$ is flat, $\ulsep S/\mathfrak P \ulsep S$ is equidimensional by \cite[Theorem 31.5]{Mats}.
 Since $\ulsep S/I\ulsep S\iso \ulsep R$, we get  $\ulsep
S/\mathfrak P\ulsep S\iso\ulsep R/\pr\ulsep R$. Since $\mathfrak q$ is
necessarily a
minimal prime of $\pr\ulsep R$,   equidimensionality yields that $\ulsep
R/\mathfrak q$ and $\ulsep R/\pr\ulsep R$ have the same dimension. Since $\ulsep
R/\pr\ulsep R$ is the catapower of $R/\pr$, this 
dimension is just $d$, showing that $\mathfrak q$ is a
$d$-dimensional prime.

Assume next that $R$ is unmixed. Since $R$ has no embedded primes, it
satisfies Serre's condition $(\op S_1)$, whence so does $\ulsep R$ by
Corollary~\ref{C:ulsepreg} and \cite[Theorem 23.9]{Mats}. Since we already know
that $\ulsep R$ is quasi-unmixed, it is in fact unmixed. 
\end{proof}

\begin{theorem}\label{T:ubunmix}
A Noetherian  local ring $R$ is unmixed \iff\ there is a binary numerical function
$\chi_R$ such that for every $x,y\in R$, we have
an inequality
\begin{equation}\label{eq:ubunmix}
\ord{}{xy}\leq \chi_R(\op{deg}(x),\ord{}y).
\end{equation}
\end{theorem}
\begin{proof}
Assume first that $R$ is unmixed. Since degree and order remain the same
when we pass to the completion, we may assume $R$ is complete. By
Lemma~\ref{L:unm}, the catapower $\ulsep R$ is then also unmixed. By
way of contradiction,  assume that for some pair $(a,b)$, we can find
elements $\seq xn,\seq yn\in R$ with $\op{deg}(\seq xn)\leq a$ and $\ord{}{\seq
yn}\leq b$,  such that $\seq xn\seq yn\in\maxim^n$. Hence, in the ultrapower
$\ul R$ of $R$, the
 ultraproduct  $\ul x$ of the $\seq xn$ has degree at most $a$ and
the ultraproduct $\ul y$ of the $\seq yn$ has order at most $b$, but $\ul x\ul
y\in\inf{\ul R}$.   Since $\ul x$ has finite degree, it is generic and
hence its image in $\ulsep R$ lies outside any prime of maximal dimension.
Since $\ulsep R$ is unmixed, $\ul x$ is therefore $\ul R$-regular and hence $\ul
y=0$ in $\ulsep R$, contradicting that its order is at most $b$.

Conversely, assume a function $\chi$ with the proscribed properties exists and
let $x$ be a generic element, say, of degree $a$. We have to show that $x$ is
$R$-regular. If not,
then $xy=0$ for some non-zero $y$, say, of order $b$. However, the order of $xy$
is bounded by $\chi(a,b)$, contradiction.
\end{proof}

By the same argument, one easily proves that the cataproduct of
Noetherian local rings $\seq Rn$ of bounded embedding dimension is unmixed,
provided almost each $\seq Rn$ satisfies the hypothesis of the statement with
respect to the same uniformity function $\chi=\chi_{\seq Rn}$. In order to
characterize quasi-unmixedness, we have to introduce one more invariant. Given
a Noetherian local ring $R$, we define its \emph{nilpotency degree}  to 
be the least $t$ such that $\mathfrak n^t=0$, where $\mathfrak n$ is   the
nilradical of $R$. Hence $R$ is reduced
\iff\
its nilpotency degree is one.

\begin{proposition}\label{P:ubqunmix}
A Noetherian local ring $R$ of nilpotency degree at most $t$ is quasi-unmixed
\iff\
there exists a binary numerical function $\theta_R$ such that  for every $x,y\in R$, we have
an inequality
$$
\ord{}{(xy)^t}\leq \theta_R(\op{deg}(x),\ord{}{y^t}).
$$
\end{proposition}
\begin{proof}
Again, we may pass to the completion of $R$, since all invariants remain
unchanged under completion, and assume from the start that
 $R$ is complete. Suppose that $\theta_R$ has the above property. To show that
$R$ is quasi-unmixed, which in the complete case is just being equidimensional,
we need to show that any generic element $x$ lies outside any minimal prime of
$R$. A moment's reflection shows that this is equivalent with showing that $x$
is $\red R$-regular. Hence, towards a contradiction, assume $y\in R$ is a 
non-nilpotent element in $R$ such that $xy$ is nilpotent. By definition of
$t$, this means $y^t\neq 0$, but $(xy)^t=0$. However, the order of $(xy)^t$ is
bounded by the finite number $\theta_R(\op{deg}(x),\ord{}{y^t})$, contradiction.

Conversely, assume $R$ is equidimensional, but no function $\theta_R$ can be
defined for some pair $(a,b)$. Hence we can find counterexamples $\seq xn\in R$
of degree $a$ and $\seq yn\in R$ such that $\seq yn^t\notin\maxim^{b+1}$, but
$(\seq
xn\seq yn)^t\in\maxim^n$. Let $\ul x$, $\ul y$ and $\ul R$ be the respective
ultraproducts, so that  $\ul x$ is generic by Corollary~\ref{C:deggen},
and $\ul y^t\notin\maxim^{b+1}\ul R$, but $(\ul x\ul y)^t\in \inf{\ul R}$ by \los.
However, by Lemma~\ref{L:unm}, the cataproduct $\ulsep R$ is again
equidimensional (note that $\ulsep R$ is complete), and therefore, $\ul x$,
being generic in $\ulsep R$, is $\red{(\ulsep R)}$-regular. Hence $(\ul x\ul
y)^t=0$ in $\ulsep R$ yields that $\ul y$ is nilpotent in $\ulsep R$. Let  $\mathfrak n$
be the nilradical of $R$. Since $\ulsep R/\mathfrak n\ulsep
R$ is the catapower of $\red R=R/\mathfrak n$ by
Corollary~\ref{C:clos}, it is
  reduced by Corollary~\ref{C:ulsepreg}. This proves that the nilradical of
$\ulsep R$ is just $\mathfrak n\ulsep R$ and hence in particular, $\ulsep R$
has nilpotency degree  $t$ too. Therefore, $\ul y^t=0$, contradicting that
$\ul y^t\notin\maxim^{b+1}\ulsep R$.
\end{proof}


\begin{theorem}\label{T:ubCM}
A $d$-dimensional Noetherian local ring $R$ is \CM\ \iff\
there exists a binary numerical function $\delta_R$ such that  for  all
$d$-tuples $\tuple x:=\rij xd$ and $\rij yd$ with $\tuple x$ a system of
parameters,
we have an inequality
\begin{equation}
\label{eq:ubCM}
\ord{R}{x_1y_1+\dots+ x_dy_d}\leq \delta_R(\ell(R/\tuple xR), \ord{R/\rij
x{d-1}R}{y_d}).
\end{equation}
Moreover, the function $\delta_R$ only depends on the dimension and the
multiplicity of $R$.
\end{theorem}
\begin{proof}
Assume first a function $\delta_R$ with the asserted properties exists. In
order to prove that $R$ is \CM, we take a system of parameters $\rij zd$ and
show that it is $R$-regular.  Fix some $i$ and suppose
$a_1z_1+\dots+a_iz_i=0$. We need to show that $a_i\in I:=\rij z{i-1}R$. Fix
some $k$ and define $x_j$ and $y_j$ as follows. If $j=\range 1{d-i}$, then
$x_j:=z_{i+j}^k$ and $y_j:=0$;  if $j=\range{d-i+1}d$, then $x_j:=z_{i+j-d}$
and $y_j:=a_{i+j-d}$.
In other
words, we have 
\begin{equation}\label{eq:xy}
\begin{aligned} 
\tuple x:=\rij xd&=(z_{i+1}^k,\dots,z_d^k,z_1,\dots,z_i)\\ 
\tuple y:=\rij yd&=(0,\dots,0,a_1,\dots,a_i)\\
x_1y_1+\dots+x_dy_d&=a_1z_1+\dots+a_iz_i=0.
\end{aligned}
\end{equation}
 Apply \eqref{eq:ubCM} to these two tuples $\tuple x$ and $\tuple y$. Since
$\tuple x$ is again a system of
parameters, $\ell(R/\tuple xR)$ is finite. Hence, by the last equation in
\eqref{eq:xy}, the order of $y_d=a_i$ in $R/\rij x{d-1}R$   must be infinite, that is to say, 
$$
a_i\in \rij x{d-1}R=I+(z_{i+1}^k,\dots,z_d^k)R.
$$
 Since this holds for all $k$, Krull's intersection theorem yields $a_i\in I$.

To prove the converse, suppose $R$ is \CM, but we cannot
define $\delta_R(a,b)$ for some pair $(a,b)$. This means that there exists for
each $n$, a system of parameters  $\seq{\tuple x}n:=(\seq{x_1}n,\dots,\seq{x_d}n)$
such that $R/\seq{\tuple x}nR$ has length $a$, and a $d$-tuple
$\seq{\tuple y}n:=(\seq{y_1}n,\dots,\seq{y_d}n)$, such that 
$$
\ord{R/(\seq{x_1}n,\dots,\seq{x_{d-1,}}n)R}{\seq{y_d}n}=b
$$
 and 
$\seq{x_1}n\seq{y_1}n+\dots+\seq{x_d}n\seq{y_d}n$ has order at least $n$.
Let  $\tuple{\ul x}:=(\ul{x_1},\dots,\ul{x_d})$ and $\ul {y_i}$ be the respective
ultraproducts of the $\seq{\tuple x}n$ and $\seq{y_i}n$ inside
the ultrapower $\ul R$ of $R$. By \los, the order  of $\ul {y_d}$ in 
$\ul R/(\ul{x_1},\dots,\ul {x_{d-1}})\ul R$ is $b$, the length of $\ul
R/\tuple{\ul x}\ul R$ is  $a$, and  the sum $\ul {x_1}\ul {y_1}+\dots+\ul {x_d}\ul
{y_d}$ is an infinitesimal. In particular, the image
of $\tuple{\ul x}$ in the catapower $\ulsep R$ is a system of
parameters, whence $\ulsep R$-regular, since $\ulsep R$ is \CM\ by
Corollary~\ref{C:ulsepreg}. Since $\ul {x_1}\ul {y_1}+\dots+\ul {x_d}\ul {y_d}=0$
in $\ulsep R$, regularity forces  $\ul {y_d}$
to be in the ideal $(\ul{x_1},\dots,\ul{x_{d-1}})\ulsep R$, contradicting that
its order in $\ulsep R/(\ul{x_1},\dots,\ul{x_{d-1}})\ulsep R$ is finite. 

To prove the final statement, observe that for fixed dimension $d$ and
multiplicity $e$, we may modify the above proof by
taking each counterexample $\seq{\tuple x}n$ and $\seq{\tuple y}n$ in some  
$d$-dimensional local \CM\ ring $\seq Rn$ of multiplicity $e$. Indeed, by
Corollary~\ref{C:CMulsep}, the cataproduct $\ulsep R$ of the $\seq Rn$
is again \CM\ so that we can copy the above argument.
\end{proof}

One can view the previous result as a quantitative version of the unmixedness
theorem. Namely, we can rewrite   condition~\eqref{eq:ubCM} as follows: for any
$d-1$-tuple $\tuple z$ and any $x,y\in R$, 
if $\tuple z$ is part of a system
of parameters, then 
\begin{equation}\label{eq:unifunmix}
\ord{R/\tuple zR}{xy}\leq \delta_R(\op{deg}_{R/\tuple zR}(x),\ord{R/\tuple
zR}{y}). 
\end{equation}
Comparing this   with \eqref{eq:ubunmix}, we can now rephrase
Theorem~\ref{T:ubCM} using the following terminology: by a
\emph{curve}, we mean
a one-dimensional subscheme $C$ of $X:= \spec R$; we call a curve $C$  a 
\emph{complete intersection in $X$} if it is of
the form $\spec{R/I}$ with $I$ an ideal generated by $\op{dim}R-1$
elements; we call $C$ \emph{unmixed}, if its coordinate ring is (note that this is equivalent with $C$ being \CM).

\begin{corollary}
A Noetherian local ring $R$ is \CM\ \iff\   every complete intersection curve
$C$ in $\spec R$  is unmixed with respect to a uniformity function  $\chi=\chi_C$
(as given by Theorem~\ref{T:ubunmix}) 
   independent from   $C$.
\qed
\end{corollary}

We can depart from other criteria for \CM{ness} to get some more uniformity 
characterizations. For instance, we could use the criterion proven in  \cite[Corollary
5.2.11]{SchUlBook} that $R$ is \CM\ \iff\ every system of
parameters $\tuple x:=\rij xd$ is \emph{independent}, in the sense    that a
relation $x_1y_1+\dots+x_dy_d=0$ implies that all $y_i$ lie in $\tuple xR$.
Thus, we get the following modified form of \eqref{eq:ubCM}:  \emph{a $d$-dimensional
Noetherian local ring $R$ is \CM\ \iff\  there exists a binary numerical function
$\delta'_R$ such that for  every two $d$-tuples $\tuple x:=\rij xd$ and $\rij
yd$, we have an inequality}
\begin{equation*}
\ord R{x_1y_1+\dots+ x_dy_d}\leq \delta'_R(\ell(R/\tuple xR),\ord{R/\tuple
xR}{y_d}). 
\end{equation*}
Next, we characterize normality:

\begin{theorem}\label{T:ubnorm}
A   Noetherian local ring $R$ is normal \iff\
there exists a binary numerical function $\varepsilon_R$ such that  for  all
  $x,y,z\in R$, we have an inequality
\begin{equation}\label{eq:norm}
\min_k\{\ord{R/z^kR}{xy^k}\}\leq \varepsilon_R(\ord {}x,\ord{R/zR}y).
\end{equation}
\end{theorem}
\begin{proof}
Suppose $R$ is normal, but $\varepsilon_R$ cannot be defined for a pair
$(a,b)$. Hence,  for each $n$, there exist elements $\seq xn,\seq yn,\seq zn\in
R$ such that $\seq xn$ has order $a$ and $\seq yn$ has order $b$ modulo $\seq
znR$, but    $\ord{R/\seq zn^kR}{\seq xn\seq yn^k}\geq n$ for all $k$. Let $\ul
x,\ul y,\ul z\in \ul R$ be the respective ultraproducts of $\seq xn,\seq yn,\seq zn\in
R$. In particular, $\ul x$ is non-zero in the catapower $\ulsep R$   
and $\ul y\notin\ul z\ulsep R$. On the other hand, since
$\ul x\ul y^k\in\ul z^k\ulsep R$ for all $k$, a well-known criterion shows
that $\ul y$ lies in the integral closure of $\ul z\ulsep R$. Since $\ulsep R$ is normal by
Corollary~\ref{C:ulsepreg}, any principal ideal is integrally closed, so that
$\ul y\in\ul z\ulsep R$, contradiction.

Conversely, assume a numerical function $\varepsilon_R$ exists with the
proscribed
properties. Taking $z=0$ in \eqref{eq:norm}, we see that $R$ is a domain by
Theorem~\ref{T:domCh}. Suppose $y/z$ is an element in the field of fractions
of $R$ which is integral over $R$. We want to show that $y/z\in R$. Since $y$ is
then in the integral closure of $zR$, there exists a non-zero $x$ such that $xy^k\in z^kR$ for
all  $k$. The left hand side in \eqref{eq:norm} is therefore
infinite, whence so must the right hand side be, forcing $y\in zR$. 
\end{proof}

In our last two examples, we show how also tight closure conditions fit in
our present program of characterizing properties by
certain uniform behavior. We will adopt the usual tight closure notation of
writing $I^{[q]}$ as an abbreviation for the ideal $(w_1^q,\dots,w_n^q)R$,
where $I:=\rij wnR$ is some ideal and $q$
is some power of the  prime \ch\ $p$ of $R$. An element $y\in R$ lies in the 
\emph{tight closure}  $I^*$ of $I$, if there exists   $c\in R$ outside all
minimal prime ideals,
such that $cy^q\in I^{[q]}$ for all powers $q$ of $p$. We say that $R$ is
\emph{F-rational} if some parameter ideal is tightly closed, in which case
every parameter ideal is tightly closed (recall that a
\emph{parameter ideal} is an ideal generated by a system of parameters). On the
other hand, if every ideal is tightly closed, then we call $R$
\emph{weakly F-regular}.

\begin{theorem}\label{T:ubFrat}
An    excellent local ring $R$ of \ch\ $p$ is pseudo-rational \iff\
there exists a ternary numerical function $\varphi_R$ such that  for all
elements $x,y\in R$ and every (equivalently, some) parameter ideal  $I$, we have
an inequality
\begin{equation}
\label{eq:Frat}
\min_q\{\ord{R/I^{[q]}}{xy^q}\}\leq \varphi_R(\op{deg}(x),\ell(R/I),\ord{R/I}y)
\end{equation}
where $q$ runs over all powers of $p$.
\end{theorem}
\begin{proof}
We will use Smith's tight closure characterization     \cite{SmFrat} that $R$
is    pseudo-rational \iff\  it is F-rational. Assume first that
$R$ is pseudo-rational whence F-rational, but a numerical function $\varphi_R$
cannot be defined on the triple $(a,b,c)$. Hence there exist for each $n$,
elements $\seq xn,\seq yn\in R$ and a parameter ideal $\seq In$ in 
$R$ such that $\seq xn$ has degree $a$ and 
$R/\seq{I}n$ is an Artinian local ring of length $b$ in which $\seq
yn$ has order $c$, but  $\ord{R/ \seq{I}n^{[q]}}{xy^q}\geq n$ for all
powers $q$ of $p$. Let $\ul x,\ul y,\ul I$ be the respective
ultraproducts of the $\seq xn,\seq yn,\seq{I}n$ and let $\ulsep R$ be
the catapower of $R$. Let $J$ be a parameter ideal in $R$.
Hence $J\ulsep R$ is a parameter ideal in $\ulsep R$. Since
$R\to\ulsep R$ is regular   by
Corollary~\ref{C:ulsepreg} and since $J$ is tightly closed, so is 
$J\ulsep R$ by \cite[Theorem 131.2]{HuTC} or \cite{HHFreg}, showing
that $\ulsep R$ is   F-rational.

Since a pseudo-rational local ring is a domain, $\ul x$ is generic in $\ulsep
R$ and $\ul{I}\ulsep R$ is a   parameter ideal in $\ulsep R$. Moreover,
$\ul y\notin\ul{I}\ulsep R$, but
$\ul x\ul y^q\in \ul{I}^{[q]}\ulsep R$ for all $q$. By definition of
tight closure, $\ul y\in(\ul{I}\ulsep R)^*$. 
In particular, every parameter ideal,
including $\ul{I}\ulsep R$, is tightly closed and hence $\ul
y\in\ul{I}\ulsep R$, contradiction.

Conversely, assume $\varphi_R$ satisfies \eqref{eq:Frat} for some parameter ideal $I$. To verify that
$R$ is F-rational, let  
$y\in I^*$. Hence, for some $x\in R$ not in any minimal prime, $xy^q\in  I^{[q]}$ for all
$q$. The left hand side of  \eqref{eq:Frat} is
therefore infinite whence so is the right hand side. Since $x$ is generic, whence has finite degree, the third argument
must be infinite, that is to say,   $y\in I$. 
\end{proof}

\begin{theorem}\label{T:ubFreg}
A    Noetherian local ring $(R,\maxim)$ of \ch\ $p$ is weakly F-regular \iff\
there exists a ternary numerical function $\psi_R$ such that  for all
elements $x,y\in R$ and all $\maxim$-primary ideals $I$, we have
an inequality
\begin{equation}
\label{eq:Freg}
\min_q\{\ord{R/I^{[q]}}{xy^q}\}\leq \psi_R(\op{deg}(x),\ell(R/I),\ord{R/I}y)
\end{equation}
where $q$ runs over all powers of $p$.
\end{theorem}
\begin{proof}
Note that for $R$ to be weakly F-regular, it suffices that every $\maxim$-primary ideal
is tightly closed, since by Krull's Intersection Theorem, any ideal is an intersection of $\maxim$-primary
ideals. Moreover, if $R$ is weakly F-regular, then so is its catapower
$\ulsep R$ by \cite[Theorem 7.3]{HHFreg} in conjunction with
Corollary~\ref{C:ulsepreg}. In view of these facts, the proof is now almost
identical to the one for Theorem~\ref{T:ubFrat}; details are left to the reader.
\end{proof}

\subsection{Epilogue: characterization of regularity}\label{R:reg}
Let me make a few further observations, although they do no longer relate to
our proof method. If $R$ is regular, then in fact
$\ord{}{xy}=\ord{}x+\ord{}y$. However, the latter condition does not
characterize regularity, but   only the strictly weaker condition that the 
associated graded ring $\gr{} R$ is a domain. The following condition, however,
does characterize regularity: \emph{a Noetherian local ring $R$ is regular \iff\
$\ord{}x=\op{deg}(x)$ for all $x\in R$}. Indeed, if $R$ is regular and
$\ord{}x=a$, then by judiciously choosing  a regular system of parameters
$\rij xd$, we can ensure that $x$ still has order $a$ in $V:=R/\rij x{d-1}R$.
Since $V$ is a \DVR\ with
uniformizing parameter $x_d$, one checks that $\ell(V/xV)=a$. Since
$\op{deg}(x)\leq \ell(R/(x,x_1,\dots,x_{d-1})R)=a$, we get
$\op{deg}(x)\leq\ord{}x$. The other inequality follows from our discussion in
\S\ref{s:orddeg}. 

Conversely, if order and degree agree, then in particular there exists an
element of degree one, and hence a system of parameters $\tuple x$ such that
$R/\tuple xR$ has length one, whence is a field, showing that $\tuple x$ is a regular system of 
parameters.\qed

\section{Asymptotic homological conjectures in mixed \ch}\label{s:mix}

In \cite{SchMixBCMCR,SchMixBCM}, we derived asymptotic versions of the 
homological conjectures for local rings of mixed \ch\ $p$, where by \emph{asymptotic},
we mean that the residual \ch\ $p$ must be large with respect to the
complexity of the data. In the above papers,   complexity was
primarily given in terms of the degrees of the  polynomials
defining the data. In this paper, we phrase complexity   in terms of (natural) invariants of the ring and the data only.

\subsection*{Improved New Intersection Theorem}

To not have to repeat each time the conditions from this theorem, we make the
following definition: given a finite complex $F_\bullet$  of
finitely generated free $R$-modules, a \emph{finite free complex}, for short, we say that its \emph{rank} is at most
$r$, if all $F_i$ have rank at most $r$; and we say that its
\emph{\INIT} is at most $l$, if  each $H_i(F_\bullet)$, for $i>0$, has
length at most $l$, and 
$H_0(F_\bullet)$ has a   minimal generator generating a
submodule of length at most $l$. Recall that the \emph{length} of $F_\bullet$
is the largest $n$ such that $F_n\neq0$.

\begin{theorem}[Asymptotic Improved New Intersection Theorem] 
For 
each triple of non-negative integers $(m,r,l)$, there
exists a bound $\kappa(m,r,l)$ with the following property. Let $R$ be 
a Noetherian local ring of mixed \ch\ $p$ and of embedding dimension at most $m$. 
If $F_\bullet$ is a finite free complex of rank at most $r$ and \INIT\ at
most $l$, then its length   is at 
least the dimension of $R$, provided  $p\geq \kappa(m,r,l)$.
\end{theorem}
\begin{proof} 
Since the dimension of $R$ is at most $m$, there is 
nothing to show for complexes of length $m$ or higher.
Suppose the result is false for some triple $(m,r,l)$. This means 
that for infinitely many distinct prime numbers $\seq pw$, we can find a $\seq
dw$-dimensional Noetherian local ring $(\seq Rw,\seq\maxim 
w)$ of mixed \ch\ $\seq pw$ and  
embedding dimension  at most $m$, and we can find a finite free  complex 
$\seq{F_\bullet}w$ of rank at most $r$, of length $\seq sw\leq m$, and of \INIT\
at most $l$, such
that  $\seq sw<\seq dw$. Choose a non-principal ultrafilter   and let $(\ul R,\ulmax)$ 
be the ultraproduct of the $(\seq Rw,\seq\maxim w)$.  Since $\seq sw<\seq dw\leq m$, their respective ultraproducts satisfy    $s<d\leq m$.
By Theorem~\ref{T:uldim}, the
\pdim\ of $\ul R$ is at least $d$.    Let 
$\ul {F_\bullet}$ be the ultraproduct of the complexes
$\seq{F_\bullet}w$. Since the ranks are at most $r$, each module in 
$\ul {F_\bullet}$ is a free $\ul R$-module of rank at most
$r$. Since ultraproducts commute with homology, and preserve 
uniformly bounded length by Proposition~\ref{P:ullen},
the higher homology  groups $H_i(\ul {F_\bullet})$ have finite length (at most
$l$). Furthermore, by assumption, we can find a minimal generator $\seq\mu w$ of
$H_0(\seq{F_\bullet}w)$ generating a submodule of length at most $l$. Hence the
ultraproduct $\ul\mu$ of the $\seq\mu w$ is by \los\ a minimal generator of $H_0(\ul
{F_\bullet})$, generating a submodule of length at most $l$. In conclusion, $\ul{F_\bullet}$ has \INIT\ at most $l$.  
In particular,
$\ul {F_\bullet}$ is acyclic when localized at a non-maximal prime ideal, 
and hence \eqref{i:EG} from Corollary~\ref{C:INIT}
applies, yielding that $s\geq \pd {\ul R}\geq d$,   contradiction.
\end{proof}

We can even give an asymptotic version of Theorem~\ref{T:INIT}, 
albeit in terms of some less natural bounds.

\begin{theorem}\label{T:INITmix} 
For each triple of non-negative 
integers $(m,r,l)$, there exists a bound
$\sigma(m,r,l)$ with the following property. Let $(R,\maxim)$ be a 
Noetherian local ring of mixed \ch\ $p$ and of embedding dimension at most $m$, and let
$F_\bullet$ be a finite free complex   of rank at most $r$. Let  $M$ be the cokernel of $F_\bullet$, and let $\mu$ be a
non-zero minimal generator of $M$. Assume
  each 
$R/I_k(F_\bullet)$ has  dimension at most
$\op{dim}R-k$ and parameter degree at most $l$, for $k\geq 1$, 
and $R/\ann R\mu$ has parameter degree at most $l$.

If  $p\geq \sigma(m,r,l)$, then the length of the complex $F_\bullet$ is at
least  the codimension of $\ann R\mu$.
\end{theorem}
\begin{proof} 
Suppose the result is false for some triple 
$(m,r,l)$. This means that for  infinitely many distinct prime numbers
$\seq pw$,
we can find a $\seq dw$-dimensional mixed \ch\ Noetherian local ring 
$(\seq Rw,\seq\maxim w)$ whose residue field has
\ch\ $\seq pw$ and whose embedding dimension is at most $m$, and we can 
find a  finite free complex $\seq {F_\bullet}w$ of length $\seq sw$
and of rank at 
most $r$, and a non-zero minimal generator $\seq\mu w$ of its cokernel $\seq Mw$ such that
$\seq Rw/I_k(\seq {F_\bullet}w)$
has dimension at most $\seq dw-k$ and parameter degree at most $l$, for all 
$k=\range 1{\seq sw}$, and such that $\seq Rw/\ann{\seq Rw}{\seq\mu w}$ has
parameter
degree at most $l$, but dimension strictly less than $\seq dw-\seq sw$.  Choose a
non-principal ultrafilter   and let $(\ul R,\ulmax)$ be the 
ultraproduct of the $(\seq Rw,\seq\maxim w)$.  Since
$\seq sw\leq \seq dw\leq m$, their respective ultraproducts satisfy $s\leq d\leq
m$. 
By Theorem~\ref{T:uldim}, the
\pdim\ of $\ul R$ is  at
least $d$. Let $\ul {F_\bullet}$ and $\mu$ be the 
ultraproduct of the complexes $\seq {F_\bullet}w$ and  the minimal generators
$\seq\mu w$ respectively.
Since the ranks are at most $r$, each module in $\ul {F_\bullet}$ will be a free 
$\ul R$-module of rank at most $r$. By
Theorem~\ref{T:isodim}, the \pdim\ of $\ul R/I_k(\ul {F_\bullet})$ is at 
most $d-k$, for all $k=\range 1s$. Also by \los,
$\mu$ is a minimal generator of the 
cokernel of $\ul {F_\bullet}$ and $\ul R/\ann{\ul R}\mu$, being the ultraproduct
of the $\seq Rw/\ann{\seq Rw}{\seq\mu w}$, 
has
\pdim\
strictly less than $d-s$ by Theorem~\ref{T:isodim}. 
However, this is in contradiction with Theorem~\ref{T:INIT}, which yields that $\ul
R/\ann{\ul R}\mu$ has \pdim\ at least $d-s$.
\end{proof}

Using the same techniques, we can deduce from Theorem~\ref{T:CE} the following
asymptotic version (details are left to the reader).

\begin{theorem}[Asymptotic Canonical Element Theorem] 
For 
each triple of non-negative integers $(m,r,l)$, there
exists a bound $\rho(m,r,l)$ with the following property. Let $R$ be 
a $d$-dimensional Noetherian local ring of mixed \ch\ $p$ and embedding dimension at most $m$, and let
$F_\bullet$ be a   free resolution of the residue field $k$ of 
$R$, of rank at most $r$.  

If $\tuple x$ is a system of parameters in $R$ such that $R/\tuple xR$ has
length at most $l$ and if the morphism of complexes $\gamma\colon
K_\bullet(\tuple
x)\to F_\bullet$  extends the natural \homo\ $R/\tuple xR\to k$,
then $\gamma_d\neq 0$,   provided $p\geq \rho(m,r,l)$.\qed
\end{theorem}

\begin{remark}
Perhaps it is not entirely justified to call this theorem a `canonical element
theorem', since it does not necessarily produce a canonical element in local 
cohomology like it does in the
equi\ch\ case. This is due to the fact that we can not apply the theorem to the
various `powers' of a system of parameters as in the discussion in
\cite[p.~346-347]{BH}  without having to raise the bound $\rho(n,r,l)$. In particular, the
above
result does not imply an asymptotic version of the Direct Summand conjecture.
\end{remark}

\subsection*{Ramification}
Instead of requiring that the residual \ch\ is large in the above asymptotic
results, we can also require the ramification to be large, as we will now
explain.  For the proofs, we only need to apply the corresponding versions in 
\S\ref{s:INIT} for infinitely ramified local rings of finite embedding
dimension. The main observation is the following
immediate corollary of \los:

\begin{lemma}\label{L:infram}
Let $\seq Rw$ be Noetherian local rings of    mixed   \ch\ $p$ and
 embedding dimension $m$. If for each $n$, almost all $\seq
Rw$ have ramification index at least $n$, then
their ultraproduct $\ul R$  is infinitely ramified and hence their cataproduct
$\ulsep R$ has equal \ch\ $p$.\qed
\end{lemma}

\begin{theorem}\label{T:asymhc}
For each triple of non-negative 
integers $(m,r,l)$, there exists a non-negative integer
$\kappa(m,r,l)$ with the following properties. Let $(R,\maxim)$ be a $d$-dimensional 
mixed \ch\  Noetherian local ring of embedding dimension at most $m$, and let
$F_\bullet$ be a finite
free complex of rank at most $r$. If the ramification index of $R$ is at least
$\kappa(m,r,l)$, then the following are true:
\begin{enumerate}
\item\label{i:asymhc} If $F_\bullet$ has \INIT\ at most $l$, then the length of
$F_\bullet$ is at 
least $d$.

\item If  each $R/I_k(F_\bullet)$ has 
dimension at most $d-k$ and parameter degree at most $l$, for $k\geq 1$, 
and if $\mu$ is a non-zero minimal
generator of the cokernel of $F_\bullet$ such that $R/\ann R\mu$ has parameter
degree at most $l$, then the length of $F_\bullet$ is at 
least  the codimension of $\ann R\mu$. 

\item If $F_\bullet$ is a   free resolution of  $R/\maxim$, if $\tuple x$ is a
system of parameters in $R$ such that $R/\tuple xR$ has
length at most $l$ and if the morphism of complexes $\gamma\colon
K_\bullet(\tuple x)\to F_\bullet$  extends the natural \homo\ $R/\tuple xR\to
R/\maxim$,
then $\gamma_d\neq 0$.
\end{enumerate}
\end{theorem}
\begin{proof}
Suppose first that such a bound for a triple $(m,r,l)$ cannot be found in a
fixed residual \ch\
$p$. In other words, we can find mixed
\ch\ $p$
Noetherian local rings $\seq Rw$, whose  embedding dimension is at most $m$,
and whose ramification index is
at least $w$, satisfying the negation of one of the above properties. By
Lemma~\ref{L:infram}, their cataproduct is   equi\ch\ and
the
proof follows by the previous discussion; details are left to the reader. To
make this bound independent from $p$ as well, we use the corresponding bounds
from the previous theorems.
\end{proof}

\subsection*{Monomial Theorem} 
By the same process as above, we can derive some asymptotic version of the
Monomial Theorem from Corollary~\ref{C:mon}. Unfortunately, the bounds will
also depend on the monomials involved, and hence does not lead to an
asymptotic version of the Direct Summand conjecture.
More precisely, given
$\nu_0,\dots,\nu_s\in\nat^d$ with $\nu_0$ not a positive linear combination of
the $\nu_i$ and given $l,m$, there is a bound $N$ depending on these data,
such that for every mixed \ch\ $p$ Noetherian local ring $R$ of
embedding dimension
at most $m$ and dimension $d$, and for every system of parameters $\tuple
x:=\rij xd$
in $R$ such that $R/\tuple xR$ has length at most $l$, if either $p$ or the
ramification index of $R$ is at least $N$, then $\tuple
x^{\nu_0}$ does not belong to the ideal in $R$ generated by the $\tuple
x^{\nu_i}$.

In particular, for fixed $m$ and $l$, we get a bound $N_t$, for each $t\geq 1$,
such that \eqref{eq:mon} holds, whenever $\tuple x$ and  $R$ satisfy the
assumptions  from the previous
paragraph. To derive from this an asymptotic version of the Direct Summand
conjecture, we need to show that the $N_t$ can be chosen independently from
$t$. 
To derive this
conclusion, we would like to establish the following result. Let $(\ul R,\ulmax)$ be
an isodimensional ultra-Noetherian local ring, say the
ultraproduct of $d$-dimensional  Noetherian local rings $(\seq Rw,\seq\maxim w)$
of bounded embedding dimension and parameter degree. Let $H_\infty^d(\ul R)$ be the
ultraproduct of the local cohomology groups $H^d_{\seq\maxim w}(\seq Rw)$. There
is a natural
map $H^d_{\ulmax}(\ul R)\to H_\infty^d(\ul R)$. 

\begin{conjecture}
The canonical map $H^d_{\ulmax}(\ul R)\to H_\infty^d(\ul R)$ is injective.
\end{conjecture}

Without proof, I state that the conjecture is  true when $\ul R$ is
ultra-\CM. Let us show how this conjecture implies that the $N_t$ can be
chosen to be
independent from $t$, thus yielding a true asymptotic version of the Monomial
Theorem (whence also of the Direct Summand Theorem) in mixed \ch. Indeed,
assume the conjecture and let $(\ul{x_1},\dots,\ul{x_d})$ be a generic sequence
in $\ul R$ and choose
$\seq x{iw}\in\seq Rw$ so that their ultraproduct is $\ul {x_i}$. Since the
(image of
 the) element $1/(\ul {x_1}\cdots \ul {x_d})$ in
the top local cohomology module $H^d_{\ulmax}(\ul R)$ is non-zero by
Corollary~\ref{C:mon}---here
we realize $H^\bullet_{\ulmax}(\ul R)$ as the cohomology of the \Cech\ complex
associated to $(\ul{x_1},\dots,\ul{x_d})$---its image in $H_\infty^d(\ul R)$
is therefore also
non-zero, whence   almost each $1/(\seq x{1w}\cdots \seq x{dw})$ is non-zero in
$H^d_{\seq\maxim w}(\seq Rw)$. Hence \eqref{eq:mon} is valid for almost each $(\seq
x{1w},\dots,\seq x{dw})$ and  all $t$.

\subsection*{Towards a proof of the full Improved New Intersection
Theorem} 
Although our methods can in principle only prove asymptotic versions, a better
understanding of the bounds can in certain cases lead to a complete solution
of the conjecture. To formulate such a result, let us say that a  numerical
function $f$
\emph{grows
sub-linearly} if there exists some $0\leq \alpha<1$ such that 
$f(n)/n^\alpha$ remains bounded when $n$ goes to infinity.

\begin{theorem}
Suppose that for each pair $(m,r)$ the numerical function
$f_{m,r}(l):=\kappa(m,r,l)$ grows sub-linearly, where $\kappa$ is the numerical
function given in \eqref{i:asymhc}, then the  Improved New
Intersection 
Theorem holds.
\end{theorem}
\begin{proof}
Let $\mathcal I_{m,r,l}$ be the collection of counterexamples   with
invariants $(m,r,l)$,
that is to say, all
mixed
\ch\ Noetherian local rings $R$ of embedding dimension at most $m$, admitting
a   finite free complex $F_\bullet$  of rank at most $r$ and \INIT\ at most
$l$, such that the length of $F_\bullet$ is
strictly less than the dimension of $R$. We have to show that $\mathcal
I_{m,r,l}$ is empty for all $(m,r,l)$, so by way of contradiction, assume it is
not for the triple $(m,r,l)$. For each $n$, let 
$f(n)$ be the supremum of the ramification indexes of  counterexamples in $\mathcal
I_{m,r,n}$ (and equal to $0$ if there is no counterexample). By
Theorem~\ref{T:asymhc}, this supremum is always finite.
By assumption, $f$ grows sub-linearly, so that for some positive real numbers $c$ and 
$\alpha<1$, we have $f(n)\leq cn^\alpha$, for all
$n$. In particular, for $n$  larger
  than the $(1-\alpha)$-th root of $\frac {cl^\alpha}{f(l)}$, we have
\begin{equation}\label{eq:subpr}
f(ln)<nf(l).
\end{equation}
 Let $(R,\maxim)$  be
a counterexample in $\mathcal I_{m,r,l}$ of ramification index $f(l)$,
witnessed by the finite free complex $F_\bullet$ of length 
strictly less than the dimension of $R$.
Since
the completion of $R$ will be again a counterexample in $\mathcal I_{m,r,l}$ of the same ramification
index, we may assume $R$ is complete, whence by Cohen's structure theorem of the
form $R=\pow V\xi/I$ for some \DVR\ $V$, some tuple of indeterminates $\xi$, and
some ideal $I\sub\pow V\xi$.
Let $n\gg0$ so that
\eqref{eq:subpr}
holds, and let $W:=\pol Vt/(t^n-\pi)\pol Vt$, where $\pi$ is a uniformizing
parameter of $V$. Let $S:=\pow W\xi/I\pow W\xi$, so that $R\to S$ is faithfully
flat and $S$ has the same dimension and embedding dimension as $R$. By
construction, its ramification index is equal to $nf(l)$. By faithful flatness,
$F_\bullet\tensor_RS$
is a finite free complex of length strictly less than the dimension of $S$,
with homology equal to $\op H_\bullet(F_\bullet)\tensor_RS$. I claim that
if  $H$ is an $R$-module of   length $a$, then $H\tensor_RS$ has length
$na$. Assuming this claim, it follows that $S$ belongs to $\mathcal
I_{m,r,nl}$, and hence its ramification is by definition at most $f(ln)$,
contradicting \eqref{eq:subpr}.

The claim is easily reduced by induction to the case $a=1$, that is to
say, when $H$ is equal to the residue field $R/\maxim=V/\pi V=k$. In that case,
$H\tensor_RS=S/\maxim S=W/\pi W$, and
this is isomorphic to $\pol kt/t^n\pol kt$, a module of length
$n$.
\end{proof}

 \providecommand{\bysame}{\leavevmode\hbox to3em{\hrulefill}\thinspace}
\providecommand{\MR}{\relax\ifhmode\unskip\space\fi MR }
\providecommand{\MRhref}[2]{%
  \href{http://www.ams.org/mathscinet-getitem?mr=#1}{#2}
}
\providecommand{\href}[2]{#2}

\end{document}